\documentclass[11pt]{amsart}
\usepackage[english]{babel}
\usepackage{enumitem}
\usepackage{amsmath,amsthm}
\usepackage{hyperref}
\usepackage[T1]{fontenc}
\usepackage{lmodern}
\usepackage[capitalize]{cleveref}
\usepackage{amssymb}
\usepackage{mathtools}
\usepackage{amscd}
\usepackage{marginnote}
\usepackage[all]{xy} 
\usepackage{dsfont}
\usepackage{tikz-cd}
\usepackage{stmaryrd} 
\usepackage[mathscr]{eucal}

\setenumerate[1]{label=(\alph*)}
\setenumerate[2]{label=(\roman*)}

\theoremstyle{plain}
\newtheorem{thm}{Theorem}[section]
\newtheorem*{thm*}{Theorem}
\newtheorem{lem}[thm]{Lemma}
\newtheorem{cor}[thm]{Corollary}
\newtheorem{prop}[thm]{Proposition}

\theoremstyle{remark}
\newtheorem{rem}[thm]{Remark}
\newtheorem*{rem*}{Remark}

\theoremstyle{definition}

\newtheorem{defin}[thm]{Definition}

\newcommand{\Ecal}{\mathcal{E}}
\newcommand{\Ccal}{\mathcal{C}}
\newcommand{\Fcal}{\mathcal{F}}

\newcommand{\Hcal}{\mathcal{H}}

\newcommand{\Jcal}{\mathcal{J}}
\newcommand{\Lcal}{\mathcal{L}}
\newcommand{\Mcal}{\mathcal{M}}
\newcommand{\Ncal}{\mathcal{N}}
\newcommand{\Ocal}{\mathcal{O}}

\newcommand{\Ucal}{\mathcal{U}}

\newcommand{\ZZ}{\mathbb{Z}}
\newcommand{\CC}{\mathbb{C}}
\newcommand{\HH}{\mathbb{H}}

\newcommand{\RR}{\mathbb{R}}

\newcommand{\con}[1]{\nabla_{#1}}
\newcommand{\Po}{\mathcal{P}} 
\newcommand{\Ed}{E^\vee} 
\newcommand{\id}{\mathrm{id}}
\newcommand{\righteq}{\stackrel{\sim}{\rightarrow}}
\newcommand{\pr}{\mathrm{pr}}

\newcommand{\dd}{\mathrm{d}}
\newcommand{\triv}{\mathrm{triv}}

\newcommand{\Cinfty}{\Ccal^\infty}
\newcommand{\EN}{E}
\newcommand{\Oan}{\mathcal{O}^{\text{hol}}}

\newcommand{\PoC}{\tilde{\mathcal{P}}}
\newcommand{\MN}{M}

\newcommand{\om}{\underline{\omega}}
\newcommand{\HdR}[2]{\underline{H}^{#1}_{\mathrm{dR}}\left( #2 \right)}
\newcommand{\HdRabs}[2]{H^{#1}_{\mathrm{dR}}\left( #2 \right)}
\newcommand{\Tate}{\mathrm{Tate}}

\newcommand{\scan}{s_{\mathrm{can}}}
\newcommand{\scant}{\tilde{s}_{\mathrm{can}}}
\newcommand{\trv}{\tilde{\mathfrak{t}}}
\newcommand{\thetaD}{\prescript{}{D}\theta}

\newcommand{\EisD}{\prescript{}{D} E^{k,r+1}_{s}}

\newcommand{\Zp}{\mathbb{Z}_p}
\newcommand{\Qp}{\mathbb{Q}_p}
\newcommand{\Gmf}[1]{\widehat{\mathbb{G}}_{m,#1}}
\newcommand{\Gm}[1]{\mathbb{G}_{m,#1}}

\newcommand{\Ef}{\hat{E}}
\newcommand{\Etriv}{\mathrm{E}}
\newcommand{\Edtriv}{\mathrm{E}^\vee}
\newcommand{\Eftriv}{\hat{\mathrm{E}}}
\newcommand{\Edftriv}{\hat{\mathrm{E}}^\vee}
\newcommand{\Pof}{\widehat{\Po}}
\newcommand{\Lnf}{\widehat{\Lcal}_n}
\newcommand{\Mtriv}{\mathrm{M}}
\newcommand{\VpN}{V_p}
\newcommand{\Edf}{\hat{E}^\vee}
\newcommand{\EispD}{\prescript{}{D}{\mathcal{E}}^{k,r+1}_{(a,b)}}
\newcommand{\Frob}{\mathrm{Frob}}
\newcommand{\muEisD}{\mu^{\mathrm{Eis}}_{D,(a,b)}}
\newcommand{\muEisDp}{\mu^{\mathrm{Eis,(p)}}_{D,(a,b)}}
\newcommand{\pthetaD}{\prescript{}{D}\vartheta}

\newcommand{\UR}{\mathsf{U}}

\DeclareMathOperator{\Spec}{Spec}
\DeclareMathOperator{\Res}{Res}
\DeclareMathOperator{\Sym}{\underline{Sym}}
\DeclareMathOperator{\Inf}{Inf}
\DeclareMathOperator{\TSym}{\underline{TSym}}
\DeclareMathOperator{\Tr}{Tr}
\DeclareMathOperator{\im}{im}
\DeclareMathOperator{\dlog}{\mathrm{dlog}}
\DeclareMathOperator{\Pic}{\underline{\mathrm{Pic}}^0_{E/S}}
\DeclareMathOperator{\Meas}{\mathrm{Meas}}

\title[Eisenstein--Kronecker series via the Poincar\'e bundle]{Eisenstein--Kronecker series via the Poincar\'e bundle}

\author{Johannes Sprang}
\address{Fakult\"at f\"ur Mathematik Universit\"at Regensburg  \\ 93040 Regensburg }
\email{johannes.sprang@mathematik.uni-regensburg.de }
\thanks{The author is supported by the CRC 1085 \emph{Higher Invariants} (Universit\"at Regensburg) funded by the DFG}
\date{}

\hypersetup{
pdfauthor = {J. Sprang},
pdftitle = {Eisenstein--Kronecker series via the Poincar\'e bundle}
}

\begin{document}

\begin{abstract}
A classical construction of Katz gives a purely algebraic construction of Eisenstein--Kronecker series using the Gau\ss--Manin connection on the universal elliptic curve. This approach gives a systematic way to study algebraic and $p$-adic properties of real-analytic Eisenstein series. In the first part of this paper we provide an alternative algebraic construction of Eisenstein--Kronecker series via the Poincar\'e bundle. Building on this, we give in the second part a new conceptional construction of Katz' two-variable $p$-adic Eisenstein measure through $p$-adic theta functions of the Poincar\'e bundle.
\end{abstract}

\maketitle

\section[Introduction]{Introduction}
The classical \emph{Eisenstein--Kronecker} series are defined for a lattice $\Gamma=\omega_1\ZZ+\omega_2\ZZ\subseteq \CC$, $s,t\in\frac{1}{N}\Gamma$ and integers $r-2>k\geq 0$ by the absolutely convergent series
	\[
		e^*_{k,r}(s,t;\Gamma):=\sum_{\gamma\in\Gamma\setminus\{-s\}} \frac{(\bar{s}+\bar{\gamma})^k}{(s+\gamma)^r} \langle \gamma,t  \rangle
	\]
with $\langle z,w  \rangle:=\exp\left( \frac{z\bar{w}-w\bar{z}}{A(\Gamma)} \right)$ and $A(\Gamma)=\frac{\im(\omega_1\bar{\omega}_2)}{\pi}$.  For arbitrary integers $k,r$ they can be defined by analytic continuation,  cf.~\cite[\S 1.1]{bannai_kobayashi}.  For our purposes, it is more convenient to normalize the Eisenstein--Kronecker series for integers $k,r\geq 0$, as follows:
\[
	\tilde{e}_{k,r+1}(s,t;\Gamma):=\frac{(-1)^{k+r}r!}{A(\Gamma)^k}e^*_{k,r+1}(s,t;\Gamma)
\]
 By varying the lattice the Eisenstein-Kronecker series define $\Ccal^\infty$-modular forms $\tilde{e}_{k,r+1}(s,t)$ of level $N$ and weight $k+r+1$. Although Eisenstein--Kronecker series are in general non-holomorphic, it turns out that they belong to the well behaved class of \emph{nearly holomorphic modular forms}.\par 
 A classical result of Katz gives a purely algebraic approach towards Eisenstein--Kronecker series by giving an algebraic interpretation of the Maa\ss--Shimura operator on the modular curve in terms of the Gau\ss--Manin connection. His construction has been one of the main sources for studying systematically the algebraic and $p$-adic properties of real-analytic Eisenstein series. Even today it has still influence, e.g.~it plays a key role in the proof of one of the most general known explicit reciprocity laws for Tate modules of formal $p$-divisible groups by Tsuji \cite{tsuji}. Unfortunately, the construction of the Maa\ss-Shimura operator only works in the universal situation and doesn't have good functoriality properties.\par 

The study of Eisenstein--Kronecker series through the Poincar\'e bundle has been initiated by Bannai and Kobayashi \cite{bannai_kobayashi}. Bannai and Kobayashi have proven that Eisenstein--Kronecker series appear as expansion coefficients of certain translates of the \emph{Kronecker theta function}:
	\begin{equation}\label{eq_KroneckerTheta}
		\Theta_{s,t}(z,w)=\sum_{a,b\geq 0}\frac{\tilde{e}_{a,b+1}(s,t)}{a!b!}z^bw^a,\quad s,t\notin\Gamma
	\end{equation}
The Kronecker theta function is a reduced theta function associated to the Poincar\'e bundle of an elliptic curve. For elliptic curves with complex multiplication, Bannai and Kobayashi have used this observation to study many interesting $p$-adic and algebraic properties of Eisenstein--Kronecker series at CM points \cite{bannai_kobayashi}. These results have been fruitfully applied by Bannai, Kobayashi and Tsuji for studying the algebraic de Rham and the syntomic realization of the elliptic polylogarithm for elliptic curves with complex multiplication \cite{BKT}. This approach does not immediately generalize to more general elliptic curves over arbitrary base schemes. The assumption of complex multiplication is essential to deduce the algebraicity of the involved theta function. \par 
Building on the work of Bannai--Kobayashi, we provide in the first part of the paper a purely algebraic construction of Eisenstein--Kronecker series via the Poincar\'e bundle on the universal vectorial bi-extension. This construction works over arbitrary base schemes and has good functoriality properties.
In the second part of the paper, we provide a new approach to the $p$-adic interpolation of Eisenstein--Kronecker series. We will show that Katz' two-variable $p$-adic Eisenstein measure is the Amice transform of a certain $p$-adic theta function of the Poincar\'e bundle. This does not only give a conceptional approach towards the $p$-adic Eisenstein measure, it also provides a direct bridge between $p$-adic theta functions and $p$-adic modular forms.\par
Let us briefly outline the content of this paper in more detail: Theta functions provide a convenient way to describe sections of line bundles over complex abelian varieties analytically. In the first part of this paper, we will define a certain section of the Poincar\'e bundle, called \emph{the Kronecker section}. While the Kronecker section is defined for arbitrary families of elliptic curves, it can be expressed in terms of the Kronecker theta function for elliptic curves over $\CC$. The universal vectorial extension of the dual elliptic curve classifies line bundles of degree zero with an integrable connection. This gives us two integrable connections $\con{\sharp}$ and $\con{\dagger}$ on the Poincar\'e bundle over the universal vectorial bi-extension. By applying both connections iteratively to the Kronecker section and evaluating at torsion sections $s$ and $t$, we obtain a new algebraic construction of Eisenstein--Kronecker series:
\begin{thm*}
Let $\Hcal^{1}_{\mathrm{dR}}$ denote the first relative de Rham cohomology of the universal elliptic curve $\Ecal\rightarrow\Mcal$. The image of the sections
\[
	(s\times t)^{*} \left( \con{\sharp}^{\circ k}\circ\con{\sharp}^{\circ r}(\scan)\right)\in\Gamma(\Mcal,\Sym^{k+r+1}\Hcal^{1}_{\mathrm{dR}}) 
\]
under the Hodge decomposition  gives the classical Eisenstein--Kronecker series $\tilde{e}_{k,r+1}(Ds,Nt)$.
\end{thm*}
Let us refer to \cref{EP_MainThm} for more details. Although the algebraic construction of Eisenstein--Kronecker series goes back to Katz, this new point of view has several advantages. Let us remark that the construction of the sections in $\Gamma(S,\Sym^{k+r+1}\Hcal^{1}_{\mathrm{dR}} )$ applies, contrary to Katz' approach, to elliptic curves $E\rightarrow S$ over arbitrary base schemes $S$ and has good functoriality properties. Furthermore, the geometric setup is much more symmetric than the construction involving the Gau\ss--Manin connection. Indeed, the functional equation of the Eisenstein--Kronecker series is reflected by the symmetry obtained by interchanging the role of the elliptic curve and its dual. Finally, let us note that the Kronecker section provides a new construction of the logarithmic derivatives of the Kato--Siegel functions associated to a positive integer $D$. It might be worth to mention that, contrary to the classical construction of the Kato--Siegel functions, our construction even works if $D$ is not co-prime to $6$.\par 
In the second part of the paper, we apply the methods of the first part to obtain a new approach to Katz' $p$-adic Eisenstein measure. In \cite{katz_padicinterpol}, Katz constructed a two-variable $p$-adic measure with values in the ring of generalized $p$-adic modular forms having certain real-analytic Eisenstein series as moments. Katz' $p$-adic Eisenstein measure is the key for studying $p$-adic congruences between real-analytic Eisenstein series. But the proof of existence of the $p$-adic Eisenstein measure uses all the predicted congruences between Eisenstein series. These congruences in turn have to be checked by hand on the $q$-expansion. We will give a more intrinsic construction of the $p$-adic Eisenstein measure via the Poincar\'e bundle. Norman's theory of $p$-adic theta functions allows us to associate a $p$-adic theta function $\pthetaD_{(a,b)}(S,T)$ to the Kronecker section for any elliptic curve with ordinary reduction over a $p$-adic base. Applying this to the universal trivialized elliptic curve gives us a two-variable power series over the ring of generalized $p$-adic modular forms. A classical result of Amice allows us to associate a two-variable $p$-adic measure $\muEisD$ to this $p$-adic theta function. It turns out that the resulting measure interpolates the  Eisenstein--Kronecker series $p$-adically:

\begin{thm*}
The $p$-adic Eisenstein--Kronecker series $\EispD$ appear as moments
	\[
		\EispD=\int_{\Zp\times\Zp} x^k y^r \dd \muEisD(x,y)
	\]
	of the measure $\muEisD$ associated to the $p$-adic theta function $\pthetaD_{(a,b)}(S,T)$.
\end{thm*}
Let us refer to \cref{thm_padictheta} and \cref{cor_Eis_moments} for details. This construction does not only give a more conceptional approach towards the $p$-adic Eisenstein measure, it also provides a direct bridge between $p$-adic theta functions and $p$-adic modular forms.\par

Last but not least, one of our main leading goals was to obtain a better understanding of polylogarithmic cohomology classes. The specialization of polylogarithmic cohomology classes along torsion sections gives Eisenstein classes which play a key role in proving particular cases of deep conjectures like the Tamagawa Number Conjecture (TNC) of Bloch and Kato and its $p$-adic analogue. In particular, the syntomic Eisenstein classes play an important role for proving particular cases of the $p$-adic Beilinson conjecture\cite{bannai_kings, bannai_kings2}.  We have already mentioned, that Bannai, Kobayashi and Tsuji have used the Kronecker theta function to give an explicit description of the algebraic de Rham realization and the syntomic realization of the elliptic polylogarithm for CM elliptic curves. Our aim is the generalization of these results to families of elliptic curves. In the case of the de Rham realization this goal has already been archived by Scheider in his PhD thesis by analytic methods \cite{rene}. Scheider's PhD thesis has been a great source of inspiration for us. He gave an \emph{analytic} description of the de Rham realization for the universal elliptic curve in terms of certain theta functions. But for all arithmetic applications it is indispensable to also have an explicit \emph{algebraic} description of the de Rham polylogarithm. The techniques developed in this paper allow to get rid of the analytically defined theta functions, to give a purely algebraic reinterpretation of the results of Scheider and to streamline some of his proofs \cite{deRham}.  An explicit algebraic description of the de Rham realization of the elliptic polylogarithm is the cornerstone towards an explicit description of the syntomic realization which is treated in \cite{Syntomic}.\par 
A naturally arising question is the question about higher dimensional generalizations. This question will be the content of future investigations.

\section*{Acknowledgement}
The results presented in this paper are part of my Ph.D. thesis at the Universit\"at Regensburg \cite{PhD}. It is a pleasure to thank my advisor Guido Kings for his guidance during the last years. Further, I would like to thank Shinichi Kobayashi for all the valuable suggestions, remarks and comments on my PhD thesis. The author would also like to thank the collaborative research centre SFB 1085 ``Higher Invariants'' by the Deutsche Forschungsgemeinschaft for its support.

\part{Real-analytic Eisenstein series via the Poincar\'e bundle}
During this work $E$ usually denotes an elliptic curve over an arbitrary base scheme $S$. The structure morphism $G\rightarrow S$ of an $S$-group scheme will be denoted by $\pi=\pi_G$ while its identity morphism will be denoted by $e=e_G$. For a commutative $S$-group scheme $G$ and some section $t\in G(S)$ let us write $T_t\colon G\rightarrow G$ for the translation morphism. Let us write $\om_{G/S}:=e^*\Omega^1_{G/S}$ for the relative co-Lie algebra of a commutative group scheme over $S$.

\section{Nearly holomorphic modular forms}
In this section we will briefly recall the theory of nearly holomorphic modular forms with an emphasize on their geometric interpretation following Urban \cite{urban}. The group of $2\times 2$ matrices with positive determinant $GL_2(\RR)^+$ acts on the upper half plane
\[
	\HH:=\{\tau=x+iy\in\CC: y>0 \}
\] 
by fractional linear transformations
\[
	\gamma.\tau:=\frac{a\tau+b}{c\tau+d} \text{ for } \gamma=\left(\begin{matrix} a & b\\c & d \end{matrix}\right)\text{ and }\tau\in\HH.
\]
For a complex valued function $f$ on $\HH$, a non-negative integer $k$ and $\gamma\in GL_2(\RR)^+$ let us write
\[
	(f|_k\gamma)(\tau):=\frac{\det(\gamma)^{k/2}}{(c\tau+d)^k}f(\gamma.\tau).
\]
Recall that holomorphic modular forms of weight $k$ and level $\Gamma\subseteq SL_2(\ZZ)$ are defined as functions on $\HH$ satisfying the following conditions:
\begin{enumerate}
\item $f$ is a holomorphic on $\HH$,
\item $f|_k\gamma=f$ for all $\gamma\in\Gamma$,
\item $f$ has a finite limit at the cusps.
\end{enumerate}
Weakening the condition $(a)$ to $f \in\Ccal^\infty(\HH)$ leads to the definition of $\Ccal^\infty$ modular forms. We have already seen in the introduction a class of $\Ccal^\infty$ modular forms of great number theoretic interest, namely the Eisenstein--Kronecker series $\tilde{e}_{a,b}(s,t;\ZZ+\tau \ZZ)$. Although the Eisenstein--Kronecker series are not holomorphic in general, they are not so far from being holomorphic. They are \emph{nearly holomorphic} in the following precise sense:
\begin{defin}\label{def_nearlyhol}
	A \emph{nearly holomorphic} modular form of weight $k$ and order $\leq r$ for $\Gamma\subseteq SL_2(\RR)$ is a function on $\HH$ satisfying:
	\begin{enumerate}
		\item $f\in\Ccal^\infty(\HH)$,
		\item $f|_k\gamma=f$ for all $\gamma\in\Gamma$,
		\item There are holomorphic functions $f_0,f_1,...,f_r$ on $\HH$ such that
		\[
			f(\tau)=f_0(\tau)+\frac{f_1(\tau)}{y}+...+\frac{f_r(\tau)}{y^r},
		\]
		\item $f$ has a finite limit at the cusps.
	\end{enumerate}
	Let us write $\Ncal_{k}^r(\Gamma,\CC)$ for the space of nearly holomorphic modular forms of weight $k$, order $r$ and level $\Gamma$.
\end{defin}

Holomorphic modular forms of weight $k$ and level $\Gamma$ can be seen as sections of the $k$-th tensor power of the cotangent-sheaf $\om^{\otimes k}$ of the generalized universal elliptic curve $\bar{\Ecal}$ of level $\Gamma$ over the compactification $\bar{\Mcal}$ of the modular curve of level $\Gamma$. This leads in a natural way to the definition of geometric modular forms and allows to study modular forms from an algebraic perspective. For sure, $\Ccal^\infty$ modular forms allow a similar geometric description as $\Ccal^\infty$ sections of $\om^{\otimes k}$, but by passing to $\Ccal^\infty$ sections, we loose all algebraic information. For nearly holomorphic modular forms, things become much better. They allow an algebraic interpretation which we will recall in the following: The algebraic de Rham cohomology $\Hcal^1_{dR}:=R^1\bar{\pi}_*\Omega^1_{\bar{\Ecal}/\bar{\Mcal}}(\log(\bar{\Ecal}\setminus \Ecal))$ sits in a short exact sequence
\[
	\begin{tikzcd}
		0\ar[r] & \om\ar[r] & \Hcal^1_{dR}\ar[r] & \om^\vee\ar[r] & 0
	\end{tikzcd}
\]
induced by the Hodge filtration
\[
	F^0\Hcal^1_{dR}=\Hcal^1_{dR}\supseteq F^1\Hcal^1_{dR}=\om\supseteq F^2\Hcal^1_{dR}=0.
\]
This filtration does not split algebraically. But after passing to $\Ccal^\infty(\bar{\Mcal})$ sections there is a canonical splitting called the Hodge decomposition.
\[
	\Hcal^1_{dR}(\Ccal^\infty)\righteq \om(\Ccal^\infty)\oplus \om^\vee(\Ccal^\infty)
\]
The Hodge filtration induces a descending filtration on $\Sym^k_{\Ocal_\Mcal}\Hcal^1_{dR}$ and by the Hodge decomposition we obtain an epimorphism
\[
	\Sym^k_{\Ccal^{\infty}(\Mcal)}\Hcal^1_{dR}(\Ccal^\infty)\twoheadrightarrow \om^{\otimes k}(\Ccal^\infty).
\]
It is the Hodge decomposition which allows to pass from algebraic sections to $\Ccal^\infty$-modular forms and thereby gives us the following algebraic interpretation of nearly holomorphic modular forms:
\begin{prop}[{\cite[Prop. 2.2.3.]{urban}}]
The Hodge decomposition induces an isomorphism
\[
	H^0\left(\bar{\Mcal}_{\CC},F^{k-r}\Sym^k \Hcal^1_{dR}\right)\righteq \Ncal^r_k(\Gamma,\CC).
\]
\end{prop}
Instead of working with the compactification of the modular curve one can also work with the open modular curve with a finiteness condition at the cusps. For simplicity let us restrict to the case $\Gamma_1(N)$. The modular curve of level $\Gamma_1(N)$ is the universal elliptic curve with a fixed $N$-torsion section. The finiteness condition at the cusp can be stated in terms of the Tate curve as follows: Let $A$ be an $\ZZ[ \frac{1}{N}]$-algebra. Let $\Tate(q^N)$ be the Tate curve over $A((q))$ with its canonical invariant differential $\omega_{can}$ and its canonical $\Gamma_1(N)$ level structure given by the $N$-torsion section $q$. Applying the Gau\ss-Manin connection
\[
	\nabla\colon \HdR{1}{\Tate(q^N)/A((q))}\rightarrow \HdR{1}{\Tate(q^N)/A((q))}\otimes \Omega^1_{A((q))/A}
\]
to $\omega_{can}$ gives us a basis $(\omega_{can},u_{can})$ with
\[
	u_{can}:=\nabla(q\frac{d}{dq})(\omega_{can}).
\]
  This leads to the notion of geometric nearly holomorphic modular forms: 
\begin{defin}
Let $A$ be a $\ZZ[\frac{1}{N}]$-algebra. A \emph{geometric nearly holomorphic modular form} $F$ of level $\Gamma_1(N)$, weight $k$ and order $\leq r$ defined over $A$ is a functorial assignment
\[
	(E/S,t)\mapsto F_{E,t}\in\Gamma(S, F^{k-r}\Sym^k\HdR{1}{E/S} ),
\]
defined for all test objects $(E/S,t)$ consisting of an elliptic curve $E/S$ with $S$ an $\Spec{A}$-scheme and $t\in E[N](S)$ an $N$-torsion section, satisfying the following finiteness condition at the Tate curve over $A$:
\[
	F_{(\Tate(q^N),q)}=\sum_{r+s=k} a_{r,s}(q)\cdot \omega_{can}^{\otimes r}\otimes u_{can}^{\otimes s}
\]
with $a_{r,s}(q)\in A\llbracket q\rrbracket\subseteq A((q))$.
\end{defin}
Finally let us remark that the two algebraic definitions of nearly holomorphic modular forms coincide: To give a geometric nearly holomorphic modular form of level $\Gamma_1(N)$ defined over $A$ is equivalent to the datum of a section
\[
	H^0\left(\bar{\Mcal}_{A},F^{k-r}\Sym^k \Hcal^1_{dR}\right).
\]
If $A\subseteq \CC$ we can further pass to the analytification and apply the Hodge decomposition to relate geometric nearly holomorphic modular forms back to the $\Ccal^\infty$-definition of nearly holomorphic modular forms.

\section{The Kronecker section}\label{sec_can}
Motivated by the work of Bannai and Kobayashi \cite{bannai_kobayashi}, we would like to give an algebraic approach to Eisenstein--Kronecker series via the Poincar\'e bundle. In particular, we have to find an algebraic substitute for the Kronecker theta function appearing in \cite{bannai_kobayashi}. One way to do this is to study algebraic theta functions as done by Bannai and Kobayashi. But unfortunately the success of this approach is limited to elliptic curves with complex multiplication. Another natural approach is to study the underlying section of the Poincar\'e bundle instead of the analytically defined Kronecker theta function. In the following, we will give a precise definition of this underlying section. It will be called \emph{Kronecker section}.\par 

\subsection{The Poincar\'e bundle}
In the following $E/S$ will be an elliptic curve. Let us write $e\colon S\rightarrow E$ for the unit section and $\pi_E\colon E\rightarrow S$ for the structure morphism. Let us recall the definition of the Poincar\'e bundle and thereby fix some notation. A \emph{rigidification} of a line bundle $\Lcal$ on $E$ is an isomorphism $r:e^*\Lcal \righteq \Ocal_S$. A morphism of rigidified line bundles is a morphism of line bundles respecting the rigidification. The dual elliptic curve $E^\vee$ is the $S$-scheme representing the connected component of the functor
\[
	T\mapsto \mathrm{Pic}(E_T/T):=\{ \text{iso. classes of rigidified line bundles } (\Lcal,r) \text{ on } E_T/T\}
\]
on the category of $S$-schemes. The $S$-scheme $E^\vee$ is again an elliptic curve. Since a rigidified line bundle has no non-trivial automorphisms, an isomorphism class of a rigidified line bundle determines the line bundle up to unique isomorphism. This implies the existence of a universal rigidified line bundle $(\Po,r_0)$ on $E\times_S E^\vee$ with the following universal property: For any rigidified line bundle of degree zero $(\Lcal,r)$ on $E_T/T$ there is a unique morphism
\[
	f:T\rightarrow E^\vee
\]
such that $(\id_E\times f)^*(\Po,r_0)\righteq (\Lcal,r)$. In particular, we obtain for any isogeny
\[
	\varphi:E\rightarrow E'
\]
the \emph{dual isogeny} as the morphism $\varphi^\vee:(E')^\vee\rightarrow E^\vee$ classifying the rigidified line bundle $(\varphi\times\id)^*\Po'$ obtained as pullback of the Poincar\'e bundle $\Po'$ of $E'$. Let $\lambda:E\righteq \Ed$ be the polarization associated with the ample line bundle $\Ocal_E([e])$. More explicitly, $\lambda$ is given as
\begin{equation}\label{EP_eqAutodual}
\begin{tikzcd}[row sep=tiny]
	\lambda: E \ar[r] & \Pic=:\Ed\\
	P\ar[r,mapsto] & \left( \Ocal_E([-P]-[e])\otimes_{\Ocal_E} \pi^*e^*\Ocal_E([-P]-[e])^{-1}, \mathrm{can} \right).
\end{tikzcd}
\end{equation}
Here, $\mathrm{can}$ is the canonical rigidification given by the canonical isomorphism
\[
	e^*\Ocal_E([-P]-[e])\otimes_{\Ocal_S} e^*\Ocal_E([-P]-[e])^{-1} \righteq \Ocal_S.
\]
We fix the identification $E\righteq \Ed$ once and for all and write again $\Po$ for the pullback of the Poincar\'e bundle along $\id\times\lambda$, i.e.
\begin{align}\label{eq_Poincare_bdl}
	(\Po,r_0):&=\left( \Ocal_{E\times E}(-[e\times E]-[E\times e] + \Delta)\otimes_{\Ocal_{E\times E}} \pi_{E\times E}^* \om_{E/S}^{\otimes -1},r_0,s_0  \right)=\\
	&=\left( \pr_1^* \Ocal_E([e])^{\otimes-1}\otimes\pr_2^*\Ocal_{E}([e])^{\otimes-1}\otimes \mu^*\Ocal_E([e])\otimes \pi_{E\times E}^* \om_{E/S}^{\otimes -1} ,r_0 \right).
\end{align}
Here, $\Delta=\ker \left(\mu:E\times E\rightarrow E\right)$ is the anti-diagonal and $r_0$ is the ridification induced by the canonical isomorphism 
\[
	e^*\Ocal_E(-[e])\righteq \om_{E/S}
\]
i.e.
\[
	r_0\colon (e\times\id)^*\Po \cong \pi_E^* e^* \Ocal_E([e])^{\otimes-1}\otimes \Ocal_{E}([e])^{\otimes-1}\otimes \Ocal_E([e])\otimes \pi_{E}^* \om_{E/S}^{\otimes -1}\cong \Ocal_E.
\]
Let us remark that the Poincar\'e bundle is also rigidified along $(\id\times e)$ by symmetry, i.e. there is also a canonical isomorphism
\[
	s_0\colon (\id\times e)^*\Po\cong \Ocal_E.
\]

\subsection{Definition of the Kronecker section}
We can now define the \emph{Kronecker section} which will serve as a purely algebraic substitute of the Kronecker theta function. The above explicit description of the Poincar\'{e} bundle gives the following isomorphisms of locally free $\Ocal_{E\times E}$-modules, i.\,e. all tensor products over $\Ocal_{E\times E}$:
\begin{align}\label{ch_EP_eq5}
	\notag \Po\otimes \Po^{\otimes -1} &= \Po \otimes \left( \Ocal_{E\times E}(-[e\times E]-[E\times e] + \Delta)\otimes \pi_{E\times E}^* \omega_{E/S}^{\otimes-1}\right)^{\otimes-1}\cong \\
	 &\cong \Po \otimes \Omega^1_{E\times E/E}([e\times E]+[E\times e]) \otimes \Ocal_{E\times E}( - \Delta)
\end{align}
The line bundle $\Ocal_{E\times E}( - \Delta)$ can be identified canonically with the ideal sheaf $\Jcal_\Delta$ of the anti-diagonal $\Delta$. If we combine the inclusion 
\[
	\Ocal_{E\times E}( - \Delta)\cong \Jcal_\Delta\hookrightarrow \Ocal_{E\times E}
\]
with \eqref{ch_EP_eq5}, we get a morphism of $\Ocal_{E\times E}$-modules
\begin{equation}\label{ch_EP_eq6}
	\Po\otimes \Po^{\otimes -1} \hookrightarrow \Po \otimes \Omega^1_{E\times \Ed/\Ed}([e\times \Ed]+[E\times e]).
\end{equation}
\begin{defin}
	Let 
	\[
		\scan\in \Gamma\left(E\times_S \Ed, \Po \otimes_{\Ocal_{E\times \Ed}} \Omega^1_{E\times \Ed/\Ed}([e\times \Ed]+[E\times e])\right)
	\]
	be the image of the identity element $\id_{\Po}$ under \eqref{ch_EP_eq6}. The section $\scan$ will be called \emph{Kronecker section}.
\end{defin}
\subsection{Translation operators} In the work of Bannai and Kobayashi, the Eisenstein--Kronecker series appear as expansion coefficients of certain translates of the Kronecker theta function. These translates are obtained by applying translation operators for theta functions to the Kronecker theta function. In the previous section we defined the Kronecker section, which serves as a substitute for the Kronecker theta function. This motivates to find similar translation operators for sections of the Poincar\'e bundle. In this section, we will define these translation operators: Let $\varphi\colon E\rightarrow E'$ be an isogeny of elliptic curves.
By the universal property of the Poincar\'e bundle, we get a unique isomorphism of rigidified line bundles
\begin{equation}\label{eq_Gamma_phi}
	\gamma_{\id,\varphi^\vee}:  (\id_{E}\times \varphi^\vee)^*\Po\righteq (\varphi\times \id_{(E')^\vee})^*\Po'.
\end{equation}
Of particular interest for us is the case $\varphi=[N]$. In this case the dual $[N]^\vee$ is just the $N$-multiplication $[N]$ on $E^\vee$. The inverse of $\gamma_{\id,\varphi^\vee}$ will be denoted by
\[
	\gamma_{\varphi,\id}:(\varphi\times\id)^*\Po'\righteq (\id\times\varphi^\vee)^*\Po.
\]
For integers $N,D\geq 1$ let us define
\[
	\gamma_{[N],[D]}: ([N]\times [D])^{*}\Po \righteq ([D]\times[N])^{*}\Po
\]
as the composition in the following commutative diagram
\begin{equation}\label{EP_diag1}
\begin{tikzcd}[column sep=huge]
	([N]\times [D])^{*}\Po \ar[r,"({[N]}\times \id)^{*}\gamma_{\id,[D]}"]\ar[d,swap,"(\id\times {[D]})^{*}\gamma_{[N],\id}"] & ([ND]\times\id)^{*}\Po \ar[d,"({[D]}\times \id)^{*}\gamma_{[N],\id}"] \\
	(\id\times[DN])^{*}\Po \ar[r,"(\id\times {[N]})^{*}\gamma_{\id,[D]}"] & ([D]\times[N])^{*}\Po.
\end{tikzcd}
\end{equation}
Indeed, this diagram is commutative since all maps are isomorphisms of rigidified line bundles and rigidified line bundles do not have any non-trivial automorphisms, i.\,e. there can be at most one isomorphism between rigidified line bundles. By the same argument we obtain the following identity for integers $N,N',D,D'\geq 1$:
\begin{align}
\label{EP_lem_gamma_c}  ([D]\times[N])^*\gamma_{[N'],[D']}\circ ([N']\times [D'])^* \gamma_{[N],[D]}&=\gamma_{[NN'],[DD']}
\end{align}	

Let us define the following translation operators. Later we will compare these algebraic translation operators for complex elliptic curves to the translation operators for theta functions studied in \cite{bannai_kobayashi}. It will turn out that both operators essentially agree.
\begin{defin}
For integers $N,D\geq 1$ and torsion sections $s\in E[N](S)$, $t\in \Ed[D](S)$ we have the following $\Ocal_{E\times_S \Ed}$-linear isomorphism 
\[
	\begin{tikzcd}
		\Ucal^{[N],[D]}_{s,t}:=\gamma_{[N],[D]}\circ(T_{s}\times T_t)^*\gamma_{[D],[N]}.: (T_s\times T_t)^*([D]\times [N])^*\Po \ar[r] &  ([D]\times [N])^*\Po.
	\end{tikzcd}
\]
 In the most important case $N=1$ we will simply write $\Ucal^{[D]}_{t}:=\Ucal^{\id,[D]}_{e,t}$. 
\end{defin}
If $T$ is an $S$-scheme and $t\in E[D](T)$ is a $T$-valued torsion point of $E$, let us write $Nt$ instead of $[N](t)$. We have the following behaviour under composition.
\begin{cor}\label{ch_EP_cor_compU}
Let $D,D',N,N'\geq1$ be integers, $s\in E[N](S)$, $s'\in E[N'](S)$ and $t\in \Ed[D](S)$, $t'\in\Ed[D'](S)$. Then:
\[
	\left(([D]\times[N])^*\Ucal^{[N'],[D']}_{Ds',Nt'}\right)\circ (T_{s'}\times T_{t'})^*([D']\times [N'])^*\Ucal^{[N],[D]}_{D's,N't}=\Ucal^{[NN'],[DD']}_{s+s',t+t'}
\]
\end{cor}
\begin{proof}
	This follows immediately from the definition of $\Ucal^{[N],[D]}_{s,t}$ and \eqref{EP_lem_gamma_c}.
\end{proof}

Later, we would like to apply the translation operator to the Kronecker section which is a section of the sheaf
\[
	\Po \otimes_{\Ocal_{E\times \Ed}} \Omega^1_{E\times \Ed/\Ed}([e\times \Ed]+[E\times e]).
\]
It is convenient to introduce the notation 
\[
	U^{[N],[D]}_{s,t}(s):=\left(\Ucal^{[N],[D]}_{s,t}\otimes\id_{\Omega^1_{E\times \Ed/\Ed}([e\times \Ed]+[E\times e])}\right)\Big( (T_s\times T_t)^*([D]\times [N])^*s \Big)
\]
for sections $s\in\Gamma(U,\Po \otimes_{\Ocal_{E\times \Ed}} \Omega^1_{E\times \Ed/\Ed}([e\times \Ed]+[E\times e]))$ of the Poincar\'e bundle. The resulting section $U^{[N],[D]}_{s,t}(s)$ is then a section of the sheaf
\[
	([D]\times[N])^*\Big( \Po\otimes \Omega^1_{E\times_S\Ed/\Ed}([(-Ds)\times \Ed]+[E\times(-Ns)])\Big)
\]
over the open subset $(T_s\times T_t)^{-1}([D]\times [N])^{-1}U$.

\section{Real-analytic Eisenstein series via the Poincar\'e bundle}
	The following section is the heart of the first part of the paper. It provides a functorial and purely algebraic construction of geometric nearly holomorphic modular forms which are proven to correspond to the analytic Eisenstein--Kronecker series under the Hodge decomposition on the universal elliptic curve. More precisely, for co-prime integers $N,D>1$ and non-zero torsion sections $s\in E[N](S),t\in \Ed[D](S)$ we will provide a section
		\[
			E^{k,r+1}_{s,t}\in \Gamma\left(S, \Sym^{k+r+1}_{\Ocal_S} \HdR{1}{E/S} \right).
		\]
	This construction is functorial in the test objects $(E/S,s,t)$. We will construct these sections by iteratively applying the universal connections of the Poincar\'e bundle on the universal bi-extension $E^\sharp\times_S E^\dagger$ and evaluating at the zero section. Let us remark that the symmetric powers of the first relative de Rham cohomology appear naturally, since the cotangent space $\omega_{E^\sharp/S}$ of the universal vectorial extension of $E^\sharp$ is canonically isomorphic to $\HdR{1}{E/S}$. 
	 
\subsection{The construction via the Poincar\'e bundle}	 
	  Let $E/S$ be an elliptic curve over some base scheme $S$. We denote by
	\[
		\begin{tikzcd}
			E^\sharp \ar[r,"q^\sharp"] & E & \text{and} & E^\dagger \ar[r,"q^\dagger"] & \Ed
		\end{tikzcd}
	\]
	the universal vectorial extension of $E$ and $\Ed$. Let us write $\Po^\sharp$ resp. $\Po^\dagger$ for the pullbacks of $\Po$ to $E^\sharp\times_S \Ed$ resp. $E\times_S E^\dagger$. Then, $\Po^\sharp$ resp. $\Po^\dagger$ are equipped with canonical integrable $E^\sharp$- resp. $E^\dagger$-connections $\nabla_\sharp$ resp.~$\nabla_\dagger$. Indeed, the universal vectorial extension $E^\dagger$ classifies line bundles of degree zero on $E$ equipped with an integrable connection and $(\Po^\dagger,\nabla_\dagger)$ is the universal such line bundle on $E\times_S E^\dagger$. The same construction for $E$ replaced by $E^\vee$ gives $(\Po^\sharp,\nabla_\sharp)$ on $E^\sharp\times_S E^\vee$, here note that we have a canonical isomorphism $E\righteq (E^\vee)^\vee$. Let us write $\Po^{\sharp,\dagger}$ for the pullback of $\Po$ to the universal bi-extension $E^\sharp\times_S E^\dagger$:
		\[
		\begin{tikzcd}
			& \EN^\sharp \times_{\MN} \EN^\dagger \ar[ld,"\id\times q^\dagger",swap] \ar[rd,"q^\sharp\times \id"] & \\
			\EN^\sharp \times_{\MN} E^\vee \ar[rd,"q^\sharp\times \id",swap] & & \EN \times_{\MN} \EN^\dagger \ar[ld,"\id\times q^\dagger"] \\
			& \EN \times_{\MN} \EN^\vee & 
		\end{tikzcd}
	\]
	Then, $\Po^{\sharp,\dagger}$ is in a natural way equipped with both an integrable $E^\sharp$- and an integrable $E^\dagger$-connection
	\begin{equation*}
	\begin{tikzcd}[row sep=tiny]
		\Po^{\sharp,\dagger}  \ar[rr,"(\id\times q^\dagger)^*\con{\sharp}"] & & \Po^{\sharp,\dagger}\otimes_{\Ocal_{E^\sharp\times E^\dagger}} \Omega^1_{E^\sharp\times E^\dagger/E^\sharp}\\
		\Po^{\sharp,\dagger}  \ar[rr,"(q^\sharp\times\id)^*\con{\dagger}"] & & \Po^{\sharp,\dagger}\otimes_{\Ocal_{E^\sharp\times E^\dagger}} \Omega^1_{E^\sharp\times E^\dagger/E^\dagger}.
	\end{tikzcd}
	\end{equation*}
	By abuse of notation we will write $\nabla_\sharp$ instead of  $(\id\times q^\dagger)^*\con{\sharp}$ and $\nabla_\dagger$ instead of $(q^\sharp\times\id)^*\con{\dagger}$. The cotangent space $\om_{E^\sharp/S}$ of the universal vectorial extension of $E$ is canonically isomorphic to $\HdR{1}{E/S}$, and similarly $\om_{E^\dagger/S}\cong \HdR{1}{\Ed/S}$. It is convenient to use our chosen polarization to identify $\Hcal^1_{dR}:=\HdR{1}{E/S}\cong\HdR{1}{\Ed/S}$. With these identifications, we get
	\begin{equation}\label{EP_OmegaH_eq}
	\begin{tikzcd}[row sep=tiny,column sep=small]
		\Omega^1_{E^\sharp\times E^\dagger/E^\sharp}\ar[r,"\sim"] & \pr_{E^\dagger}^*\Omega^1_{E^\dagger/S}\ar[r,"\sim"] & (\pi_{E^\sharp\times E^\dagger})^*\Hcal^1_{dR}\\
		\Omega^1_{E^\sharp\times E^\dagger/E^\dagger}\ar[r,"\sim"] & \pr_{E^\sharp}^*\Omega^1_{E^\sharp/S}\ar[r,"\sim"] & (\pi_{E^\sharp\times E^\dagger})^*\Hcal^1_{dR}.
	\end{tikzcd}
	\end{equation}
	Since both $\con{\sharp}$ and $\con{\dagger}$ are $(\pi_{E^\sharp\times E^\dagger})^{-1}\Ocal_S$-linear, we can define the following differential operators:
	\[
	\begin{tikzcd}[row sep =tiny]
		\con{\sharp}\colon \Po^{\sharp,\dagger}\otimes_{\Ocal_S}\Sym^{n}_{\Ocal_S}\HdR{1}{E/S} \ar[r,"\con{\sharp}\otimes \id"] & \Po^{\sharp,\dagger}\otimes_{\Ocal_S}\Sym^{n+1}_{\Ocal_S}\Hcal^1_{dR}\\
		\con{\dagger}\colon \Po^{\sharp,\dagger}\otimes_{\Ocal_S}\Sym^{n}_{\Ocal_S}\HdR{1}{E/S} \ar[r,"\con{\dagger}\otimes \id"] & \Po^{\sharp,\dagger}\otimes_{\Ocal_S}\Sym^{n+1}_{\Ocal_S}\Hcal^1_{dR}
	\end{tikzcd}
	\]
	Applying $\con{\sharp}$ and $\con{\dagger}$ iteratively leads to
	\begin{equation*}
		\begin{tikzcd}[column sep=huge]
		\con{\sharp,\dagger}^{k,r}\colon \Po^{\sharp,\dagger}\otimes_{\Ocal_S}\Sym^{n}_{\Ocal_S}\Hcal^1_{dR} \ar[r,"\con{\sharp}^{\circ k}\circ \con{\dagger}^{\circ r}"] & \Po^{\sharp,\dagger}\otimes_{\Ocal_S}\Sym^{n+k+r}_{\Ocal_S}\Hcal^1_{dR}.
	\end{tikzcd}
	\end{equation*}
	\begin{rem*}
	Let us remark that $\con{\sharp}$ and $\con{\dagger}$ do not commute in general, but after pullback along $(e\times e)$ the result is independent of the order of application. This explains why choosing a different order in our construction will not give any new geometric nearly holomorphic modular forms.
	\end{rem*}	
	
	 Next, let us consider the pullback of translates of the Kronecker section to the Poincar\'e bundle on the universal bi-extension. For co-prime integers $N,D> 1$ and non-zero torsion sections $s\in E[N](S),t\in \Ed[D](S)$ let us consider  
		\[
			(q^\sharp\times q^\dagger)^*U^{[N],[D]}_{s,t}(\scan)\in \Gamma\left((q^\sharp\times q^\dagger)^{-1}U, ([N]\times[D])^*\left(\Po^{\sharp,\dagger} \otimes \Omega^1_{E^\sharp\times E^\dagger /E^\dagger}\right)\right).
		\]
		where $U:=([D]\times[N])^{-1}(T_{Ds} \times T_{Nt})^{-1}\left(E\times \Ed \setminus \{ E\times{e} \amalg e\times \Ed \}\right)$. 
	Via the identification in \eqref{EP_OmegaH_eq} we obtain:	
	\[
	(q^\sharp\times q^\dagger)^*U^{N,D}_{s,t}(\scan)\in \Gamma\left((q^\sharp\times q^\dagger)^{-1}U, ([N]\times[D])^*\Po^{\sharp,\dagger} \otimes_{\Ocal_S} \Hcal^1_{dR}\right).
	\]	
	Since we have assumed that $s$ and $t$ are non-zero and $N$ and $D$ are co-prime, the morphism
	\[
		(e\times e):S=S\times_S S \rightarrow E\times_S \Ed
	\]
	factors through the open subset $U$. By iteratively applying the universal connections to the translates of the Kronecker section, we obtain the desired geometric nearly-holomorphic modular forms:
	\begin{defin}\label{def_algebraicEK}
		For co-prime integers $N,D\geq 1$, non-zero torsion sections $s\in E[N](S)$, $t\in \Ed[D](S)$ and integers $k,r\geq 0$ define
		\[
			E^{k,r+1}_{s,t}\in \Gamma\left(S, \Sym^{k+r+1}_{\Ocal_S} \Hcal^1_{dR} \right)
		\]
		via
		\[
			E^{k,r+1}_{s,t}:=(e\times e)^*\left[ \left( ([D]\times [N])^*\con{\sharp,\dagger}^{k,r}\right)\left((q^\sharp\times q^\dagger)^*U^{[N],[D]}_{s,t}(\scan)\right) \right].
		\]
		We call $E^{k,r+1}_{s,t}$ \emph{algebraic Eisenstein--Kronecker series}. This construction is obviously functorial.
		For later reference, let us also define
		\[
			\EisD:=\sum_{e\neq t\in E[D](S)} E^{k,r+1}_{s,t}.
		\]
		
	\end{defin}
	
	\subsection{The comparison result}	
	Let $N,D>1$ be co-prime as above and let $\Ecal\rightarrow \Mcal$ be the universal elliptic curve over $\Spec \ZZ[\frac{1}{ND}]$ with $\Gamma_1(ND)$-level structure. Let $s\in \Ecal[N](\Mcal)$ and $t\in \Ecal[D](\Mcal)$ be the points given by the level structure of order $N$ and $D$ on $\Ecal$. Note that $N,D>1$ co-prime implies $ND>3$ and thus the moduli problem is representable. We can give an explicit description of the analytification of $\Ecal/\Mcal$ as follows: The analytification $\Ecal(\CC)/\Mcal(\CC)$ is given by
	\[
		\Ecal(\CC)=\CC\times \HH/\ZZ\rtimes\Gamma_1(ND) \twoheadrightarrow \Mcal(\CC)=\HH/\Gamma_1(ND) 
	\]
	with coordinates $(z,\tau)\in \CC\times \HH$. The torsion sections $s$ and $t$ are given by $s=(\frac{1}{N},\tau)$ and $t=(\frac{1}{D},\tau)$. In particular, for each $\tau\in \HH$ we will view $s$ and $t$ as elements of $\frac{1}{N}\Gamma_\tau$ resp.~$\frac{1}{D}\Gamma_\tau$ with the varying lattice $\Gamma_\tau=\ZZ+\tau\ZZ$. By slightly abusing notation, let us write $s$ and $t$ for both, the torsion sections, as well as the associated elements in $\frac{1}{N}\Gamma_\tau$ resp.~$\frac{1}{D}\Gamma_\tau$. Let us recall that the classical Eisenstein--Kronecker series are defined for $r>k+1$ by the convergent series
	\[
		\tilde{e}^*_{k,r+1}(s,t;\tau):=\frac{(-1)^{k+r}r!}{A(\Gamma_\tau)^k}\sum_{\gamma\in\Gamma_\tau\setminus\{-s\}} \frac{(\bar{s}+\bar{\gamma})^k}{(s+\gamma)^{r+1}} \langle \gamma,t  \rangle
	\]
and for general integers $k,r$ by analytic continuation. The following result relates the algebraic Eisenstein--Kronecker series constructed in \cref{def_algebraicEK} to the classical Eisenstein--Kronecker series. 
	\begin{thm}\label{EP_MainThm}
		 Assigning to every test object $(E/S,s,t)$ the algebraic Eisenstein--Kronecker series (defined in \Cref{def_algebraicEK})
		 \[
		 	E^{k,r+1}_{s,t}\in\Gamma(S,\Sym_{\Ocal_S}^{k+r+1}\HdR{1}{E/S}),
		 \] 
		 gives a geometric nearly holomorphic modular form of weight $k+r+1$, order $\min(k,r)$ and level $\Gamma_1(ND)$. The classical nearly holomorphic modular forms associated to $E^{k,r+1}_{s,t} $ via the Hodge decomposition on the universal elliptic curve are the analytic Eisenstein--Kronecker series $\tilde{e}_{k,r+1}(Ds,Nt)$.
	\end{thm}
\subsection{The proof of the comparison result}
	The construction of  $E^{k,r+1}_{s,t}$ is compatible with base change. Thus, $E^{k,r+1}_{s,t}$ is uniquely determined by its value on the universal elliptic curve $\Ecal/\Mcal$ with $\Gamma_1(ND)$-level structure. Further, the map
		\[
		\begin{tikzcd}[column sep=small]
			\Sym^{k+r+1}_{\Ocal_{\Mcal}} \HdR{1}{\Ecal/\Mcal}\otimes_{\Ocal_{\Mcal}}\Ccal^\infty(\Mcal(\CC))\ar[r,twoheadrightarrow] & \om_{\Ecal/\Mcal}^{\otimes (k+r+1)}\left( \Ccal^\infty \right).
		\end{tikzcd}
		\]			
	induced by the Hodge decomposition on the universal elliptic curve is injective. It remains to identify the $\Cinfty$-modular form associated with $E^{k,r+1}_{s,t}$ with the Eisenstein-Kronecker series
	\[
			\tilde{e}_{k,r+1}(Ds,Nt)\dd z^{\otimes (k+r+1)}.
	\]
	We can check this identification fiber-wise and reduce the proof to the case of a single elliptic curve $E/\CC$ with $\Gamma_1(ND)$-level structure. More precisely, we fix the following setup: Let $N,D>1$ co-prime integers. Let $E/\CC$ be an elliptic curve with complex uniformization
	\[
		E(\CC)\righteq \CC/\Gamma,\quad \Gamma:=\ZZ+\tau\ZZ, \tau\in \HH
	\]
	and fixed points $s\in E(\CC)[N]$ and $t\in E(\CC)[D]$ of order $N$ resp. $D$. Again by abuse of notation, let us write $s$ and $t$ for both $s=\frac{1}{N}$ and $t=\frac{1}{D}$ and the associated $\CC$-valued points $s,t\in E(\CC)=\CC/\Gamma$.
	\subsubsection{Analytification of the Poincar\'{e} bundle}\label{subsec_Analitification}
	The complex uniformization and our chosen auto-duality isomorphism establish
	\[
		\CC\times \CC \twoheadrightarrow E(\CC)\times \Ed(\CC)
	\]
	as universal covering. Let us denote the coordinates on $\CC\times\CC$ by $(z,w)$. The explicit description of the Poincar\'e bundle in \eqref{eq_Poincare_bdl}	allows us to trivialize its pullback $\PoC$ to this universal covering using the \emph{Kronecker theta function}
	\[
		\Theta(z,w):=\frac{\theta(z+w)}{\theta(z)\theta(w)},\quad  \theta(z):=\exp\left(-\frac{e_2^*}{2}z^2\right)\sigma(z)
	\]
	as follows:
	\[
		\Ocal_{\CC\times\CC}\righteq \PoC,\quad 1\mapsto \trv:=\frac{1}{\Theta(z,w)}\otimes (\dd z)^\vee.
	\]
	The canonical isomorphism $\gamma_{\id,[D]}$ induces an isomorphism
	\[
			\tilde{\gamma}_{\id,[D]}: (\id\times [D])^* \PoC\righteq  ( [D]\times\id)^*\PoC
	\]
	on the pullback $\PoC$ to the universal covering.
	\begin{lem}[{\cite[Lemma 3.5.10]{rene}}]\label{EP_lemRene1}
		The isomorphism $\tilde{\gamma}_{\id,[D]}$ is given by
		\[
			(\id\times [D])^* \PoC\righteq  ( [D]\times\id)^*\PoC,\quad (\id\times [D])^*\trv \mapsto  ( [D]\times\id)^*\trv.
		\]
	\end{lem}
	\begin{proof}
		The proof is an adaptation of \cite[Lemma 3.5.10]{rene}. The map
			\[
			(\id\times [D])^*\trv \mapsto  ( [D]\times\id)^*\trv.
		\]
		on the universal covering descents to an isomorphism
		\[
			(\id\times [D])^* \Po^{an}\righteq  ( [D]\times\id)^*\Po^{an}
		\]
		to the analytification of the Poincar\'e bundle. It is straight-forward to check that this map respects the rigidifications of the Poincar\'e bundle. Indeed, this boils down to the fact that the theta function $\theta(z)$ used in the definition of $\Theta(z,w)$ is a normalized theta function, i.e. $\theta'(z)|_{z=0}=1$. Now the claim follows from the fact that $\gamma_{\id,[D]}$ is the unique isomorphism between $(\id\times [D])^* \Po$ and $( [D]\times\id)^*\Po$ which is compatible with the rigidifications.
	\end{proof}	
	The pullback of the Kronecker section $\scan$ to the universal covering gives a meromorphic section $\scant$ of $\PoC\otimes \Omega^1_{\CC}$. By its very construction it can be written explicitly as
		\[
			\scant=\Theta(z,w)\trv\otimes \dd z.
		\]
	The following result proves that the purely algebraic translation operators are compatible with the analytic translation operators defined in \cite[\S 1.3]{bannai_kobayashi}:
	\begin{prop}\label{ch_EP_prop_Uscan_expl}
		The pullback of $U^{N,D}_{s,t}(\scan)$ to the universal covering is given by the explicit formula
		\[
			([ D]\times[ N])^*\left( \Theta_{Ds,Nt}(z,w) \trv \otimes \dd z \right).
		\]
		Here we denote by $\Theta_{Ds,Nt}(z,w)$ the translates of the Kronecker theta function as defined in \cite[\S 1.3]{bannai_kobayashi}.
	\end{prop}	
	\begin{proof}
		Let us write $\tilde{U}^{N,D}_{s,t}$ for the pullback of the analytification of the translation operator. Recall, that $s$ and $t$ denote the torsion points of $E$ given by $s=\frac{1}{N}$ and $t=\frac{1}{D}$ and that $A:=A(\Gamma):=\frac{\im \tau}{\pi}$ denotes the volume of $E(\CC)$ divided by $\pi$. Before we give an explicit description of $\tilde{U}^{N,D}_{s,t}(\scant)$, let us do the following computation:
		\begin{align}\label{EP_eq1}
			\notag (T_{s}\times T_{t})^*\left([ N]\times [ D]\right)^*(\trv)&=\frac{1}{\Theta(Nz+1,Dw+1)} \otimes ([N]\times[D])^*(\dd z)^\vee\stackrel{(*)}{=}\\
			 &=\exp\left( -\frac{Nz+Dw+1}{A} \right) ([ N]\times[ D])^*(\trv)
		\end{align}
		Here $(*)$ follows from the transformation law of the classical theta function \cite[eq. (8)]{bannai_kobayashi}:
		\[
			\theta(z+\gamma)=\alpha(\gamma)\cdot \exp\left( \frac{z\bar{\gamma}}{A}+\frac{\gamma\bar{\gamma}}{2A} \right) \theta(z)
		\]
		By the definition of $\Theta_{Ds,Nt}(z,w)$ we have
		\begin{align*}
			&\Theta_{D{s},N{t}}(Dz,Nw):=\exp\left( -\frac{Nz+Dw+1}{A} \right) \Theta(Dz+D{s},Nw+N{t},\tau).
		\end{align*}
		The definition
		\[
			U^{N,D}_{s,t}(\scan):=(\gamma_{N,D}\otimes \id_\Omega)\left( (T_{s}\times T_{t})^*\left[(\gamma_{D,N}\otimes \id_\Omega)\left( ([D]\times[N])^*\scan \right)  \right] \right)
		\]
		gives us the following explicit description of $\tilde{U}^{N,D}_{s,t}(\scant)$:\\
\resizebox{\linewidth}{!}{
  \begin{minipage}{\linewidth}		
		\begin{align*}
			&\tilde{U}^{N,D}_{s,t}(\scant)= (\tilde{\gamma}_{N,D}\otimes \id_\Omega)\left( (T_{s}\times T_{t})^*\left[(\tilde{\gamma}_{D,N}\otimes \id_\Omega)\left( ([D]\times[N])^*\scant \right)  \right] \right)\stackrel{\text{Lem.}\ref{EP_lemRene1}}{=}\\
			&=(\tilde{\gamma}_{N,D}\otimes \id_\Omega)\left( (T_{s}\times T_{t})^*\left[ \Theta(Dz,Nw)([N]\times[D])^*(\trv\otimes \dd z)  \right] \right)=\\
			&=(\tilde{\gamma}_{N,D}\otimes \id_\Omega)\big(  \Theta(Dz+D{s},Nw+N{t})(T_{s}\times T_{t})^*([N]\times[D])^*(\trv\otimes \dd z) \big)\stackrel{\eqref{EP_eq1}}{=}\\
			&=(\tilde{\gamma}_{N,D}\otimes \id_\Omega)\left( \Theta_{D{s},N{t}}(Dz,Nw)([N]\times[D])^*(\trv\otimes \dd z)  \right)=\\
			&=\Theta_{D{s},N{t}}(Dz,Nw)([D]\times[N])^*(\trv\otimes \dd z)=\\
			&=([D]\times[N])^*\left( \Theta_{D{s},N{t}}(z,w) \trv\otimes \dd z\right)			
		\end{align*}
	\end{minipage}}
	\end{proof}
	
	The analytification of the universal vectorial extension $E^\dagger$ of $\Ed$ sits in a short exact sequence (cf. \cite[Ch I, 4.4]{mazur_messing})
\[
		\begin{tikzcd}
			0\ar[r] & R^1(\pi_{E}^{\textit{an}})_*(2\pi i\ZZ)\ar[r] & \HdRabs{1}{E}\ar[r] & \EN^{\dagger,an}\ar[r] & 0.
		\end{tikzcd}
	\]
	In particular, the two dimensional complex vector space $\HdRabs{1}{E}$ serves as a universal covering of $E^\dagger(\CC)$. Choosing coordinates on this universal covering is tantamount to choosing a basis of
	\[
		\HdRabs{1}{E}^\vee\righteq \HdRabs{1}{\Ed}.
	\]
	Here, this isomorphism is canonically induced by Deligne's pairing. Let us choose $[\dd w]$ and $[\dd \bar{w}]$ in $\HdRabs{1}{\Ed}$ as a basis and denote the resulting coordinates by $(w,v)$. We can summarize the resulting covering spaces in the following commutative diagram:
	\[
		\begin{tikzcd}
			\CC^2 \ar[r,"\pr_1"]\ar[d] & \CC \ar[d]\\
			E^{\dagger}(\CC)\ar[r] & \Ed(\CC).
		\end{tikzcd}	
	\]
	The pullback of the Poincar\'e bundle $\Po^\dagger$ to $E\times E^\dagger$ is equipped with a canonical integrable $E^\dagger$-connection
	\[
		\con{\dagger}:\Po^\dagger\rightarrow \Po^\dagger\otimes_{\Ocal_{E\times E^\dagger}} \Omega^1_{E\times E^\dagger/E^\dagger}.
	\]	
	Let us write $\PoC^\dagger$ for the pullback of the analytification of $\PoC^\dagger$ to the universal covering. The trivializing section $\trv$ of $\PoC$ induces a trivializing section $\trv^\dagger$ of $\PoC^\dagger$.
	\begin{lem}\label{eq_condagger}
		\[
			\con{\dagger}(\trv^\dagger)=-\frac{v}{A}\trv^\dagger\otimes \dd z
		\]
	\end{lem}
	\begin{proof}
		We use the description of the connection given by Katz \cite[Thm. C.6 (1)]{katz_eismeasure}. Katz uses different coordinates: The basis $([\eta]^\vee,[\omega]^\vee)$ of $\HdRabs{1}{E}^\vee$ gives coordinates $(w_{\mathrm{Katz}},v_{\mathrm{Katz}})$ on the universal covering of $E^\dagger(\CC)$. Comparing both bases it is straight-forward to check that these coordinates are related to our coordinates via
		\begin{align*}
				w_{\mathrm{Katz}}&=-w\\
				v_{\mathrm{Katz}}&=-\frac{v}{A}+w\cdot\left(\frac{1}{A}+\eta(1,\tau)\right).
		\end{align*}
		 The explicit description of the connection in \cite[Thm. C.6 (1)]{katz_eismeasure} immediately implies the following formula:
		\begin{align*}
			\con{\dagger}(\trv^\dagger)&=\left[-\frac{\partial_z\Theta(z,-w_{\mathrm{Katz}})}{\Theta(z,-w_{\mathrm{Katz}})} + (\zeta(z-w_{\mathrm{Katz}})-\zeta(z)+v_{\mathrm{Katz}})) \right] \trv^\dagger\otimes \dd z
		\end{align*}
		Using
		\begin{align*}
			\frac{\partial_z\Theta(z,-w_{\mathrm{Katz}})}{\Theta(z,-w_{\mathrm{Katz}})}&=\partial_z \log \Theta(z,-w_{\mathrm{Katz}})=\\
			&=-w_{\mathrm{Katz}}\cdot\left(\frac{1}{A}+\eta(1,\tau)\right)+\zeta(z-w_{\mathrm{Katz}})-\zeta(z)
		\end{align*}
		we get
		\begin{align*}
			\con{\dagger}(\trv^\dagger)&=\left[ v_{\mathrm{Katz}}+ w_{\mathrm{Katz}}\cdot\left(\frac{1}{A}+\eta(1,\tau)\right)   \right] \trv^\dagger\otimes \dd z
		\end{align*}
		and the result follows by expressing this in our coordinates.
	\end{proof}
	Now, let $E^\sharp$ be the universal vectorial extension of $\Ed$.	Our chosen auto-duality isomorphism $E\righteq \Ed$ induces an isomorphism
	\[
		E^\sharp\righteq E^\dagger.
	\]
	The explicit description of the universal covering of $E^\dagger(\CC)$ gives us a universal covering 
	\[
		\CC^2\twoheadrightarrow E^\sharp(\CC)
	\]
	of $E^\sharp(\CC)$. Let us write $(z,u)$ for the corresponding coordinates. Let us write $\PoC^\dagger$ for the pullback of the analytification of $\PoC^\dagger$ to the universal covering. The trivializing section $\trv$ of $\PoC$ induces a trivializing section $\trv^\sharp$ of $\PoC^\sharp$. By transport of structure we deduce
	\begin{equation}\label{eq_consharp}
			\con{\sharp}(\trv^\sharp)=-\frac{u}{A}\trv^\sharp\otimes \dd w.
	\end{equation}
	\begin{prop}\label{EP_prop_Ekr}
		Let $E=\CC/\ZZ+\tau\ZZ$ and $s,t\in E[ND](\CC)$ as above. The algebraic Eisenstein-Kronecker series $E^{k,r+1}_{s,t}\in \Sym^{k+r+1}\HdRabs{1}{E}$ are given by the explicit formula
		\[
			E^{k,r+1}_{s,t}=\sum_{i=0}^{\min(k,r)} \binom{r}{i}\binom{k}{i} \frac{(-1)^i}{A^i}\tilde{e}_{k-i,r-i+1}(Ds,Nt;\tau) \cdot[\dd \bar{z}]^{\otimes i}\otimes [\dd z]^{\otimes k+r+1-i}.
		\]
	\end{prop}
	\begin{proof}
		In the following we identify the cotangent spaces of $E^\sharp$ and $E^\dagger$ with $\HdRabs{1}{E}$. More concretely, this means that we identify:
		\begin{align*}
			\dd z&=\dd w=[\dd z]\\
			\dd u&=\dd v=[\dd \bar{z}].
		\end{align*}
		Let us write
	\begin{align*}
		\dd_{E^\dagger}&: \Oan_{E^\sharp \times E^\dagger } \rightarrow \Oan_{E^\sharp \times E^\dagger}\otimes_{\CC} \HdRabs{1}{E},\quad f\mapsto \partial_w f[\dd z]+\partial_v f[\dd \bar{z}]\\
		\dd_{E^\sharp}&: \Oan_{E^\sharp \times E^\dagger } \rightarrow \Oan_{E^\sharp \times E^\dagger}\otimes_{\CC} \HdRabs{1}{E},\quad f\mapsto \partial_z f[\dd z]+\partial_u f[\dd \bar{z}]
	\end{align*}
		Our aim is to compute
		\[
			E^{k,r+1}_{s,t}=(e\times e)^*\left[ ([D]\times [N])^*\nabla^{k,r}_{\sharp,\dagger}\left( (q^\sharp\times q^\dagger)^*U^{N,D}_{s,t}(\scan) \right) \right].
		\]
		\cref{ch_EP_prop_Uscan_expl} gives an explicit description of the translation operators:
		\[
			(q^\sharp\times q^\dagger)^*U^{N,D}_{s,t}(\scan)=([D]\times[N])^*\left(\Theta_{Ds,Nt}(z,w)\trv^{\sharp,\dagger}\otimes [\dd z]\right).
		\]
		Using this we compute using the Leibniz rule and \cref{eq_condagger}:
\begin{align*}
		&E^{k,r+1}_{s,t}=(e\times e)^*\left[  ([D]\times [N])^*\left( \nabla_{\sharp,\dagger}^{k,r}\left( \Theta_{Ds,Nt}(z,w) \trv \right)\right)\right]\otimes[\dd z]=\\
		=&  (e\times e)^*\left[ \nabla_{\sharp}^{\circ k}\nabla_{\dagger}^{\circ r}\left( \Theta_{Ds,Nt}(z,w) \trv \right)\right]\otimes[\dd z]=\\
		=&  (e\times e)^*\left[\nabla_{\sharp}^{\circ k}\left( \sum_{i=0}^r  \binom{r}{i} \left( -\frac{u}{A} \right)^i \partial_z^{\circ(r-i)} \Theta_{Ds,Nt}(z,w) \trv \right)\right]\otimes[\dd z]^{\otimes (r+1)}=\\
		=&  \left.\sum_{j=0}^k \binom{k}{j} \sum_{i=0}^r  \binom{r}{i} \dd_{E^\sharp}^{\circ (k-j)} \left[ \left(-\frac{u}{A}\right)^i \partial_z^{\circ(r-i)} \Theta_{Ds,Nt}(z,w)\right]\right|_{\substack{z=v=0\\ w=u=0}}\cdot \\
		&\quad \quad\quad \cdot \left.\left( -\frac{v}{A} \right)^j\right|_{v=0} \otimes [\dd z]^{\otimes (r+j+1)}=
\end{align*}
	at this point let us observe, that $\left.\left( -\frac{v}{A} \right)^j\right|_{v=0}=0$ for $j>0$. Using this we continue:
\begin{align*}
		=& \left. \sum_{i=0}^r  \binom{r}{i} \dd_{E^\sharp}^{\circ k} \left[ \left(-\frac{u}{A}\right)^i \partial_z^{\circ(r-i)} \Theta_{Ds,Nt}(z,w)\right]\right|_{\substack{z=v=0\\ w=u=0}} \otimes [\dd z]^{\otimes (r+1)}=\\
		=& \left. \sum_{i=0}^r  \binom{r}{i}  \sum_{j=0}^k  \binom{k}{j}  \left. \dd_{E^\sharp}^{\circ k}\left[ \left(-\frac{u}{A}\right)^i\right]\right|_{w=u=0}\cdot \partial_z^{\circ(r-i)}\partial_w^{\circ(k-j)}  \Theta_{Ds,Nt}(z,w)\right|_{z=w=0} \otimes [\dd z]^{\otimes (r+k-j+1)}=		
\end{align*}	
	again observe that $\left. \dd_{E^\sharp}^{\circ k}\left[ \left(-\frac{u}{A}\right)^i\right]\right|_{w=u=0}=0$ for $i\neq j$
\begin{align*}
		=& \left. \sum_{i=0}^{\min(r,k)}  \binom{r}{i} \binom{k}{i}  \partial_z^{\circ(r-i)}\partial_w^{\circ(k-i)}  \Theta_{Ds,Nt}(z,w)\right|_{z=w=0} \left(-\frac{[\dd \bar{z}]}{A}\right)^{\otimes i} \otimes [\dd z]^{\otimes (r+k-i+1)}	
\end{align*}
	Now the result follows from the result of Bannai-Kobayashi, i.e. \cite[Theorem 1.17]{bannai_kobayashi}:
	\[
		\Theta_{s,t}(z,w)=\sum_{a,b\geq 0}\frac{\tilde{e}_{a,b+1}(s,t)}{a!b!}z^bw^a,\quad s,t\notin\Gamma
	\]
	\end{proof}
	
	\begin{proof}[Proof of \cref{EP_MainThm}] The construction of $E^{k,r+1}_{s,t}$ is obviously functorial on test object $(E/S,t)$. The explicit formula in \cref{EP_prop_Ekr} proves that $E^{k,r+1}_{s,t}$ is contained in the $k+r+1-\min(k,r)$-th filtration step of the Hodge filtration. This can also be seen without using the transcendental description as follows: By the symmetry of the situation, we may assume $k\leq r$. Keeping in mind that we agreed to denote the pullback of the connection $\nabla_\dagger\colon \Po^\dagger \rightarrow \Po^\dagger\otimes\Omega^1_{E\times E^\dagger/E^\dagger}$ again by $\nabla_\dagger$, we get the formula
	\begin{align*}
		&\left(([D]\times[N])^*\nabla_\dagger^{\circ r}\right)(q^\sharp\times q^\dagger)^*U_{s,t}^{[N],[D]}(\scan)\\
		=&(q^\sharp\times \id_{E^\dagger})^*\left[ \left(([D]\times[N])^*\nabla_\dagger^{\circ r}\right)(\id_E\times q^\dagger)^*U_{s,t}^{[N],[D]}(\scan) \right].
	\end{align*}
	Since $\Po^\dagger\otimes\Omega^1_{E\times E^\dagger/E^\dagger}=\Po^\dagger\otimes_{\Ocal_S}\omega_{E/S}=\Po^\dagger\otimes_{\Ocal_S}F^1\Hcal^1_{dR}$ we deduce that the section $\sigma:=\left(([D]\times[N])^*\nabla_\dagger^{\circ r}\right)(\id_E\times q^\dagger)^*U_{s,t}^{[N],[D]}(\scan)$ is contained in the filtration step
	 \[
	 ([D]\times[N])^*\left( \Po^\dagger \otimes_{\Ocal_S}F^{r+1}\Sym^{r+1}\Hcal^1_{dR}\right)\subseteq ([D]\times[N])^*\left( \Po^\dagger \otimes_{\Ocal_S}\Sym^{r+1}\Hcal^1_{dR}\right).
	 \]
	 We deduce
	 \[
	 E^{k,r+1}_{s,t}=(e\times e)^*(\nabla_{\sharp}^{\circ k}(q^\sharp\times\id)^*\sigma)\in \Gamma(S,F^{r+1}\Sym^{r+1}\Hcal^1_{dR})
	 \]
	 as desired.
		
	 It remains to identify the image of $E^{k,r}_{s,t}$ under the Hodge decomposition with the Eisenstein-Kronecker series for a single elliptic curve $E/\CC$ as above. By \cref{EP_prop_Ekr} the algebraic Eisenstein-Kronecker series $E^{k,r}_{s,t}$ for $E/\CC$ is given by
		\[
			E^{k,r+1}_{s,t}=\sum_{i=0}^{\min(k,r)} \binom{r}{i}\binom{k}{i} \frac{(-1)^i}{A^i}\tilde{e}_{k-i,r-i+1}(Ds,Nt) \cdot[\dd \bar{z}]^{\otimes i}\otimes [\dd z]^{\otimes k+r+1-i}.
		\]
		The Hodge decomposition is the projection to the $[\dd z]^{\otimes k+r+1}$-part which is the $\Cinfty$ modular form
		\[
			\tilde{e}_{k,r+1}(Ds,Nt)\dd z^{\otimes k+r+1}.
		\]
		Since the analytic Eisenstein--Kronecker series are finite at the cusps, we deduce the finiteness at the Tate curve from the analytic comparison.
	\end{proof}

\section{The Kronecker section and Kato-Siegel functions}
Kato--Siegel functions $\thetaD\in \Gamma(E\setminus E[D],\Ocal_E^\times)$ as introduced by Kato in \cite{kato} play an important role in modern number theory. While the values of $\thetaD$ at torsion points are closely related to elliptic units, the value of iterated derivatives of
\[
	d\log \thetaD\in\Gamma(E,\Omega^1_{E/S}(E[D]))
\]
at torsion points give classical algebraic Eisenstein series. The aim of this section is to construct the logarithmic derivatives of the Kato--Siegel functions via the Kronecker section of the Poincar\'e bundle. Let us emphasize that it is not necessary that $6$ is co-prime to $D$ for the construction of the Kronecker section. This gives a new construction of the logarithmic derivatives of the Kato--Siegel functions even if the Kato--Siegel functions are not defined. This section might be skipped on the first reading. The comparison to the Kato--Siegel functions builds on some tedious computations involving the translation operators. Some results of this section are needed to prove the distribution relation of the Kronecker section which is stated in the Appendix. 
\subsection{Kato-Siegel functions via the Poincar\'e bundle}
Let us slightly generalize the definition of the translation operators of \cref{sec_can}.

\begin{defin}
Let $\psi:E\rightarrow E'$ be an isogeny of elliptic curves over $S$. Let us write $\Po'$ for the Poincar\'e bundle of $E'$. For $t\in\ker\psi^\vee(S)$ define
\[
	\Ucal^{\psi^\vee}_t: (\psi\times T_t)^*\Po'\righteq (\psi\times\id)^*\Po'
\]
by
\[
	\Ucal^{\psi^\vee}_t:=\gamma_{\id,\psi^\vee}\circ (\id\times T_t)^*\gamma_{\psi,\id}.
\]
\end{defin}
\begin{rem}
 In the case $\psi=[D]$ this coincides with our previous definition of $\Ucal^{D}_t$.
\end{rem}
For $f\in\Gamma\left(E'\times E'^\vee, \Po'\otimes \Omega^1_{E'\times E'^\vee/E'^\vee}([E'\times e]+[e\times E'^\vee]) \right)$ set
\[
	U^{\psi^\vee}_t(f):=\Ucal^{\psi^\vee}_t\left( (\psi\times T_t)^*f \right)
\]
thus $U^{\psi^\vee}_t(f)$ is a section of
\[
	(\psi\times\id)^*\left[\Po'\otimes \Omega^1_{E'\times E'^\vee/E'^\vee}([E'\times (-t)]+[e\times E'^\vee])\right].
\]
For $e\neq t$ the identification
\begin{align}\label{ch_EP_eq4}
	&(\id\times e)^*(\psi\times\id)^*\left[\Po'\otimes \Omega^1_{E'\times E'^\vee/E'^\vee}([E'\times (-t)]+[e\times E'^\vee])\right]=\\
	\notag\cong& \psi^*\left(\Omega^1_{E'/S}\left([e] \right)\right) \notag\cong \Omega^1_{E/S}(\ker \psi)
\end{align}
allows us to view $(\id\times e)^*\left( U^{\psi^\vee}_t(f) \right)$ as a global section of $\Omega^1_{E/S}(\ker\psi)$. We will implicitly use this identification in the following. Let us assume that the degree $\deg\psi$ of the isogeny $\psi$ is invertible on $S$. For $\tilde{t}\in (\ker\psi)(S)\subseteq E(S)$ let us write
\[
	\Res_{\tilde{t}}:\Gamma(E,\Omega^1_{E/S}(\ker\psi))\rightarrow \Gamma(S,\Ocal_S)
\]
for the residue map along $\tilde{t}$. We keep our notation $\lambda:E\righteq E^\vee$ for the auto-duality associated to the ample line bundle of the zero section.
\begin{prop}\label{ch_EP_propomegaexists}
	Let $\psi:E\rightarrow E'$ be an isogeny of elliptic curves over $S$. Let us assume that $\deg\psi$ is not a zero-divisor on $S$. For a non-zero section $t\in \ker \psi^\vee(S)\subseteq E'^\vee (S)$ the section
	\[
		\omega^{\psi}_{t}:=(\id\times e)^*U^{\psi^\vee}_{t}(s_{\mathrm{can},E'})\in\Gamma(E,\Omega^1_{E/S}(\ker\psi))
	\]
	satisfies the following properties:
	\begin{enumerate}
		\item\label{ch_EP_propomegaexists_a} For each finite \'{e}tale $S$-scheme $T$ with $\left|(\ker \psi) (T) \right|=\deg\psi$ we have
		\[
			\Res_{\tilde{t}}\omega^{\psi}_{t} = \langle \tilde{t},t\rangle
		\]
		for all $\tilde{t}\in \ker\psi(T)$. Here, 
		$$
		\langle \cdot, \cdot \rangle\colon \ker\psi\times_S \ker\psi^\vee\rightarrow \mathbb{G}_m
		$$
		denotes Oda's pairing, cf.~\cite{oda}.
		\item\label{ch_EP_propomegaexists_b} The section $\omega^{\psi}_{t}\in\Gamma\left(E,\Omega^1_{E/S}(\ker\psi)\right)$ is contained in the $\Ocal_E$-submodule
		\[
			\Omega^1_{E/S}\left(\psi^*([e]-[t])\right)
		\]
		of $\Omega^1_{E/S}(\ker\psi) $. Here, we have used our chosen auto-duality isomorphism $E'\cong E'^\vee$ to view $t$ as a section of $E'$.
	\end{enumerate}
	Further, $\omega^{\psi}_{t}$ is the unique section of $\Omega^1_{E/S}(\ker\psi)$ satisfying \ref{ch_EP_propomegaexists_a} and \ref{ch_EP_propomegaexists_b}.
\end{prop}
\begin{proof}
	For uniqueness let $\tilde{\omega}_1$ and $\tilde{\omega}_2$ both satisfy \ref{ch_EP_propomegaexists_a} and \ref{ch_EP_propomegaexists_b}. By \ref{ch_EP_propomegaexists_a} the difference satisfies:
	\[
		\tilde{\omega}_1-\tilde{\omega}_2\in \Gamma(E,\ker(\bigoplus_{\tilde{t}}\Res_{\tilde{t}}))=\Gamma(E,\Omega^1_{E/S})
	\]
	On the other hand, \ref{ch_EP_propomegaexists_b} shows that $\tilde{\omega}_1-\tilde{\omega}_2$ vanishes along the divisor $\psi^*[t]$ and we conclude $\tilde{\omega}_1-\tilde{\omega}_2=0$.\par
	Let us now prove that $\omega^{\psi}_{t}$ satisfies \ref{ch_EP_propomegaexists_b}. By its definition $s_{\mathrm{can},E'}$ is contained in the submodule 
	\[
		\Po' \otimes \Omega^1_{E'\times E'/E'}([e\times E']+[E'\times e]) \otimes \Ocal_{E'\times E'}( - \Delta)
	\]
	of $\Po' \otimes \Omega^1_{E'\times E'/E'}([e\times E']+[E'\times e])$. By the definition of the translation operator the global section $(\id\times e)^*U^{\psi^\vee}_{t}(s_{\mathrm{can},E'})$ of $\Omega^1_{E/S}(\ker\psi)$  is a global section of the $\Ocal_E$-submodule
	\[
		\Omega^1_{E/S}\left(\psi^*([e]-[t])\right).
	\]
	This proves \ref{ch_EP_propomegaexists_b}.\par 
	The residue map is compatible with base change. Combining this with the isomorphism 
	\[
		f^*\Omega^1_{E/S}(\ker\psi)\righteq \Omega^1_{E_T/T}(\ker\psi_T)
	\] 
	for $f:T\rightarrow S$ finite \'{e}tale, allow us to check \ref{ch_EP_propomegaexists_a} after finite \'{e}tale base change. Thus, we may assume that $|\ker\psi(S)|=\deg\psi$. Before we do the residue computation, let us recall the definition of Oda's pairing
	\[
		\langle\cdot,\cdot\rangle: \ker\psi\times_S\ker\psi^\vee \rightarrow \mathbb{G}_{m,S}.
	\]
	Let $t\in (\ker\psi)(S)$ and $[\Lcal]\in (\ker\psi^\vee)(S)$. Since we have assumed $[\Lcal]\in (\ker\psi^\vee)(S)$, the line bundle $\psi^*\Lcal$ is trivial and we can choose an isomorphism
	\[
		\alpha: \psi^*\Lcal\righteq  \Ocal_E.
	\]
	The chosen isomorphism $\alpha$ gives rise to a chain of isomorphisms
	\[
	\begin{tikzcd}
		\Ocal_E \ar{r}{\alpha^{-1}}[swap]{\sim} & \psi^*\Lcal=T_t^*\psi^*\Lcal \ar{r}{T_t^*\alpha}[swap]{\sim} & T_t^*\Ocal_E=\Ocal_E
	\end{tikzcd}
	\]
	and $\langle t,[\Lcal] \rangle_\psi$ is defined as the image of $1$ under this isomorphism. Our first aim is to prove
	\begin{equation*}
		T_{\tilde{t}}^*\omega^{\psi}_{t}=\langle \tilde{t},t \rangle_\psi\cdot \omega^{\psi}_{t}
	\end{equation*}
	for every $\tilde{t}\in \ker\psi(S)$. The section $t\in E'^\vee(S)$ corresponds to the isomorphism class $[(\id\times t)^* \Po']$ of line bundles. We apply Oda's pairing to $\tilde{t}$ and $[\Lcal]$ with $\Lcal:=(\id\times t)^* \Po'$. We have the following canonical choice for $\alpha$:
	\[
	\begin{tikzcd}[column sep=huge]
		\alpha: {\psi}^* \Lcal =({\psi}\times t)^* \Po' \ar[r,"(\id\times t)^*\gamma_{\psi,\id}"] & (\id\times t)^*(\id\times {\psi^\vee})^* \Po=(\id\times e)^*\Po\cong \Ocal_E.
	\end{tikzcd}
	\]
	Note that
	\begin{align}\label{ch_EP_eq7}
		\omega^{\psi}_{t}&:=(\id\times e)^* U^{\psi^\vee}_{t}(\scan)\stackrel{\mathrm{def}}{=}\notag\\
		&=(\id\times e)^*\left(\left[ \left(\gamma_{\id,\psi^\vee}\circ (\id\times T_{t})^*\gamma_{\psi,\id}\right)\otimes \id_{\Omega^1} \right]\left( (\psi\times T_{t})^*(\scan) \right)\right)  =\\
		&=\left((\id\times t)^*\gamma_{\psi,\id} \otimes\id_{\Omega^1}\right)((\psi\times t)^*\scan)=\left( \alpha \otimes \id_{\Omega^1}\right)((\psi\times t)^*\scan).\notag
	\end{align}
	After tensoring
	\[
	\begin{tikzcd}
		\Ocal_E \ar[rr,"\cdot \langle \tilde{t}{,}t \rangle_\psi",bend right]\ar[r,"\alpha^{-1}"] & \psi^*\Lcal=T_{\tilde{t}}^*\psi^*\Lcal \ar[r,"T_{\tilde{t}}^*\alpha"] & \Ocal_E
	\end{tikzcd}
	\]
	with $\otimes_{\Ocal_E}\Omega^1_{E/S}(\ker\psi)$, we obtain
	\[
	\begin{tikzcd}
		\Omega^1_{E/S}(\ker\psi) \ar[rr,"\cdot \langle \tilde{t}{,}t \rangle",bend right] & \psi^*\Lcal\otimes\Omega^1_{E/S}(\ker\psi)=T_{\tilde{t}}^*\psi^*\Lcal\otimes\Omega^1_{E/S}(\ker\psi)\ar[l,"\alpha\otimes\id_\Omega",swap] \ar[r,"T_{\tilde{t}}^*{\alpha\otimes\id_\Omega}"] & \Omega^1_{E/S}(\ker\psi)
	\end{tikzcd}
	\]
	This diagram together with \eqref{ch_EP_eq7} proves
	\begin{align*}
		T_{\tilde{t}}^* \omega^{\psi}_{t}&= T_{\tilde{t}}^*\Big[\left( \alpha \otimes \id_{\Omega^1}\right)((\psi\times t)^*\scan)\Big]= \left( T_{\tilde{t}}^*\alpha \otimes \id_{\Omega^1}\right)\left((\psi\times t)^*\scan \right)=\\
		&=\langle \tilde{t}{,}t\rangle\cdot\left( \alpha \otimes \id_{\Omega^1}\right)((\psi\times t)^*\scan)=\langle \tilde{t}{,}t \rangle\cdot\omega^{\psi}_t
	\end{align*}
	as desired.\par
	The equation
	\begin{equation*}
		T_{\tilde{t}}^*\omega^{\psi}_t=\langle \tilde{t},t \rangle_\psi \cdot \omega^{\psi}_t
	\end{equation*}	
	reduces the proof of \ref{ch_EP_propomegaexists_a} to the claim
	\[
		\Res_e \omega^{\psi}_t=1
	\]
	which can be checked by an explicit and straightforward computation locally in a neighbourhood of the zero section.
\end{proof}

	The following result was obtained during the proof of the above proposition.
	\begin{cor}\label{ch_EP_cortranslation} Let $\psi\colon E\rightarrow E'$ be an isogeny of elliptic curves and assume that $\deg \psi$ is not a zero-divisor on $S$. Then we have the following equality for all $\tilde{t}\in \ker\psi(S)$ and $t\in \ker\psi^\vee(S)$:
	\[
		T_{\tilde{t}}^*\omega^{\psi}_t=\langle \tilde{t},t \rangle_\psi \cdot \omega^{\psi}_t
	\]
	\end{cor}
	The most important case is the case $\psi=[D]$ in this case we have produced for each $t\in\Ed[D](S)$ sections
	\[
		\omega^{[D]}_t\in\Gamma(E,\Omega^1_{E/S}(E[D]))
	\]
	via the Poincar\'e bundle. Before we can relate $\omega^{[D]}_t$ to logarithmic derivatives of Kato-Siegel functions, let us study certain compatibility relations among the $\omega^\psi_t$:
	\begin{lem}\label{ch_EP_lemTrace}
		Let $\psi:E\rightarrow E'$ and $\varphi:E'\rightarrow E''$ be isogenies of elliptic curves over $S$. Let us further assume that $\deg \varphi\circ\psi$ is not a zero-divisor on $S$. 
		\begin{enumerate}
			\item\label{ch_EP_lemTrace_a} For $e\neq s\in\ker\varphi^\vee$:
			\[
				\omega_{s}^{\varphi\circ\psi}=\psi^*\omega_s^\varphi
			\]
			\item\label{ch_EP_lemTrace_b}  Assume that $|\ker \varphi^\vee(S)|=\deg\varphi^\vee$. For $t\in\ker(\varphi\circ\psi)^\vee(S)$ with $\varphi^\vee(t)\neq e$ we have
			\[
				\sum_{s\in\ker\varphi^\vee(S)}\omega_{t+s}^{\varphi\circ\psi}=\deg\varphi \cdot \omega_{\varphi^\vee(t)}^\psi
			\]			
		\end{enumerate}
	\end{lem}
	\begin{proof}\ref{ch_EP_lemTrace_b}: Both sides of the claimed equality are elements in
	\[
		\Gamma(E,\Omega^1_{E/S}(\ker\varphi\circ\psi)).
	\]	
	Since $\deg \phi\circ\psi$ is not a zero-divisor on $S$, we may check the equality after inverting $\deg \phi\circ\psi$. We may further assume that there is a finite \'{e}tale map $f:T\rightarrow S$ s.t. $|\ker\varphi\circ\psi(T)|=\deg \varphi\circ\psi$. The canonical map
		\[
			f^*\Omega^1_{E/S}(\ker\varphi\circ\psi)\rightarrow \Omega^1_{E_T/T}(\ker\varphi_T\circ\psi_T)
		\]
		is an isomorphism. Since all constructions are compatible with base change, we may assume during the proof that $|\ker\varphi\circ\psi(T)|=\deg \phi\circ\psi$. In a first step we show that the difference of both sides has no residue, i.\,e.
		\[
			\omega_0:=\left( \sum_{s\in\ker\varphi^\vee(S)}\omega_{t+s}^{\varphi\circ\psi}-\deg\varphi \cdot \omega_{\varphi^\vee(t)}^\psi \right)\in \Gamma\Big(E,\ker\bigoplus_{\tilde{t}\in\ker \varphi\circ\psi (S) }\Res_{\tilde{t}}\Big).
		\]
	For $\tilde{t}\in \ker\varphi\circ \psi(S)$ we compute
	\begin{align*}
		\Res_{\tilde{t}}\sum_{s\in\ker\varphi^\vee(S)}\omega_{t+s}^{\varphi\circ\psi}&=\sum_{s\in\ker\varphi^\vee(S)} \langle \tilde{t},t+s\rangle_{\varphi\circ\psi}=\\
		&=\langle\tilde{t}, t\rangle_{\varphi\circ\psi}\cdot\sum_{s\in\ker\varphi^\vee(S)} \langle \psi(\tilde{t}),s\rangle_{\varphi} =\\
		&=\begin{cases} \deg\varphi \cdot \langle\tilde{t},\varphi^\vee(t)\rangle_{\psi} & \tilde{t}\in\ker \psi \\ 0 & \tilde{t}\notin\ker\psi. \end{cases}
	\end{align*}
	But this coincides with the residue of $\deg\varphi\cdot \omega_{\varphi^\vee(t)}^\psi$:
	\[
		\deg\varphi \Res_{\tilde{t}} \omega_{\varphi^\vee(t)}^\psi= \begin{cases} \deg\varphi\cdot \langle\tilde{t},\varphi^\vee(t)\rangle_{\psi} & \tilde{t}\in\ker \psi \\ 0 & \tilde{t}\notin\ker\psi. \end{cases}
	\]
	This shows $\omega_0\in\Gamma(E,\Omega^1_{E/S})$. In particular, $\omega_0$ is translation-invariant. On the other hand, we can use the behaviour of $\omega^{D}_{t}$ under translation (cf.~\cref{ch_EP_cortranslation}) to compute:
	\begin{align*}
		\deg \psi \cdot \omega_0&=\sum_{\tilde{s}\in \ker\psi }T^*_{\tilde{s}}\omega_0= \sum_{\tilde{s}\in \ker\psi }\left( \sum_{s\in\ker \varphi^\vee} T_{\tilde{s}}^* \omega_{t+s}^{\varphi\circ\psi}-\deg \varphi\cdot T_{\tilde{s}}^*\omega_{\varphi^\vee(t)}^\psi \right)=\\
		&=\sum_{\tilde{s}\in \ker\psi }\left( \sum_{s\in\ker \varphi^\vee} \langle \tilde{s},t+s\rangle_{\varphi\circ\psi}\cdot \omega_{t+s}^{\varphi\circ\psi}-\deg \varphi\cdot \langle \tilde{s},\varphi^\vee(t)\rangle_{\psi}\cdot \omega_{\varphi^\vee(t)}^\psi \right)=\\
		&=\underbrace{\left(\sum_{\tilde{s}\in \ker\psi } \langle \tilde{s},\varphi^\vee(t)\rangle_{\psi} \right)}_{=0}\cdot\left( \sum_{s\in\ker \varphi^\vee}  \omega_{t+s}^{\varphi\circ\psi}-\deg \varphi\cdot \omega_{\varphi^\vee(t)}^\psi \right)=0
	\end{align*}
	Since $\deg\psi$ is not a zero-divisor on $S$, we conclude $\omega_0=0$ as desired.\par
	\ref{ch_EP_lemTrace_a} can be proven along the same lines. 
	\end{proof}
	
	\begin{lem}\label{ch_EP_lemTrace2}
		Let
		\[
			\begin{tikzcd}
				E\ar[r,"\varphi"]\ar[d,"\psi"] & E' \ar[d,"\psi'"] \\
				E'\ar[r,"\varphi'"] & E
			\end{tikzcd}
		\]
		be a commutative diagram of isogenies of elliptic curves over $S$. Let us further assume that $\deg\varphi'\circ\psi$ is not a zero-divisor on $S$. Let $t\in\ker(\varphi')^\vee(S)$ with $(\psi')^\vee(t)\neq e$.
		\begin{enumerate}
			\item\label{ch_EP_lemTrace2_a} For $s\in(\ker\psi'^\vee)(S)$ we have
			\[
				\omega^{\varphi'\circ\psi}_{s+t}=\sum_{\tilde{s}\in \ker\psi(S)} \langle  \tilde{s},\varphi'^\vee (s)  \rangle_\psi \cdot (T_{-\tilde{s}})^*\omega^{\varphi}_{\psi'^\vee(t)}
			\]
			\item\label{ch_EP_lemTrace2_b}  We have
			\[
				\sum_{\tilde{s}\in \ker\psi(S)} (T_{\tilde{s}})^*\omega^{\varphi}_{\psi'^\vee(t)}=\psi^* \omega^{\varphi'}_{t}.
			\]		
		
		\end{enumerate}
	\end{lem}
	\begin{proof}\ref{ch_EP_lemTrace2_a}: Both sides of the claimed equality are elements in
	\[
		\Gamma(E,\Omega^1_{E/S}(\ker\varphi'\circ\psi)).
	\]	
	After inverting $\deg \phi'\circ\psi$ there is a finite \'{e}tale map $f:T\rightarrow S$ s.t. $|\ker\varphi'\circ\psi(T)|=\deg \varphi'\circ\psi$. The canonical map
		\[
			f^*\Omega^1_{E/S}(\ker\varphi'\circ\psi)\rightarrow \Omega^1_{E_T/T}(\ker\varphi'_T\circ\psi_T)
		\]
		is an isomorphism and since all constructions are compatible with base change, we may assume during the proof that $|\ker\varphi'\circ\psi(T)|=\deg \phi\circ\psi$. In a first step we show that the difference of both sides has no residue, i.\,e.
		\[
			\omega_0:=\left(\omega^{\varphi'\circ\psi}_{s+t}-\sum_{\tilde{s}\in \ker\psi(S)} \langle  \tilde{s},\varphi'^\vee (s)  \rangle_\psi \cdot (T_{-\tilde{s}})^*\omega^{\varphi}_{\psi'^\vee(t)} \right)\in \Gamma\Big(E,\Omega^1_{E/S}\Big).
		\]
	For $\tilde{t}\in (\ker\varphi'\circ \psi)(S)$ we compute
	\begin{align*}
		& \Res_{\tilde{t}}\sum_{\tilde{s}\in \ker\psi(S)} \langle  \tilde{s},\varphi'^\vee (s)  \rangle_\psi \cdot (T_{-\tilde{s}})^*\omega^{\varphi}_{\psi'^\vee(t)}\\
		=& \begin{cases} \langle  \tilde{s},\varphi'^\vee (s)  \rangle_\psi \cdot \langle  \tilde{t}-\tilde{s},\psi'^\vee (t)  \rangle_\varphi & \tilde{t}-\tilde{s}\in\ker\varphi \\ 0 &  \tilde{t}-\tilde{s}\notin\ker\varphi  \end{cases}\\
		=& \begin{cases} \langle  \varphi(\tilde{s}),s  \rangle_{\psi'} \cdot \langle  \psi(\tilde{t}),t  \rangle_{\varphi'} & \varphi(\tilde{t})=\varphi(\tilde{s}) \\ 0 & \varphi(\tilde{t})\neq\varphi(\tilde{s})\end{cases}\\
		=& \langle  \varphi(\tilde{t}),s  \rangle_{\psi'} \cdot \langle  \psi(\tilde{t}),t  \rangle_{\varphi'}=\langle  \tilde{t},s+t  \rangle_{\varphi'\circ\psi}
	\end{align*}
	But this coincides with the residue of $\omega^{\varphi'\circ\psi}_{s+t}$. This shows $\omega_0\in\Gamma(E,\Omega^1_{E/S})$. In particular, $\omega_0$ is translation-invariant.
	\begin{align*}
		\deg\varphi\cdot \omega_0&=\sum_{\tilde{t}\in\ker\varphi(S)}T_{\tilde{t}}^*\omega_0=\\
		&=\sum_{\tilde{t}\in\ker\varphi(S)}\left( T_{\tilde{t}}^*\omega_{s+t}^{\varphi'\circ\psi} -\sum_{\tilde{s}\in\ker\psi(S)} \langle \tilde{s},\varphi'^\vee(s)\rangle_{\psi}T_{-\tilde{s}}^*T_{\tilde{t}}^*\omega^\varphi_{\psi'^\vee(t)}\right)=\\
		&=\underbrace{\left(\sum_{\tilde{t}\in\ker\varphi(S)}\langle \tilde{t},\psi'^\vee(t)\rangle_{\varphi}\right)}_{=0}\cdot\omega_0
	\end{align*}
	Since $\deg\varphi$ is not a zero-divisor on $S$, we conclude $\omega_0=0$ as desired.\par
	\ref{ch_EP_lemTrace2_b}: Follows by setting $s=e$ in \ref{ch_EP_lemTrace2_a} and using \cref{ch_EP_lemTrace} \ref{ch_EP_lemTrace_b}.
	\end{proof}	

The classical Kato--Siegel functions are only defined if $D$ is co-prime to $6$, while the definition of $\omega_D$ (defined below) makes sense without this hypothesis:
	\begin{cor}\label{PE_propKatoSiegel}
	For $D>1$ invertible on $S$, let $T\rightarrow S$ be finite \'etale with $|E[D](T)|=D^2$. Let us define
	\[
			\omega^D:=\sum_{e\neq t\in \Ed_T[D](T)} \omega^{[D]}_t \in\Gamma(E,\Omega_{E/S}(E[D])).
	\]
	 If $D$ is furthermore co-prime to $6$, then $\omega^D$ coincides with the logarithmic derivative of the Kato--Siegel function $\thetaD$, i.e.
		\[
			\omega^D=\dlog \thetaD.
		\]
	\end{cor}
	\begin{proof}
		The logarithmic derivative $\dlog \thetaD\in\Gamma(E,\Omega^1_{E/S}(E[D]))$ is uniquely determined by the following two properties:
		\begin{enumerate}
			\item Its residue is
			\[			
			 \Res (\dlog \thetaD)=D^2\mathds{1}_e-\mathds{1}_{E[D]}
			\]			 
			 where
			\[
				\Res: \Omega^1_{E/S}(\log E[D])\rightarrow (i_{E[D]})_*\Ocal_{E[D]}
			\]
			is the residue map and $\mathds{1}_e$ resp. $\mathds{1}_{E[D]}$ are the functions in $(i_{E[D]})_*\Ocal_{E[D]}$ which have the constant value one along $e$ resp. $E[D]$.
			\item It is trace compatible, i.\,e. for each $N$ coprime to $D$ we have 
			\[\Tr_{[N]}\dlog \thetaD=\dlog \thetaD.\]
		\end{enumerate}
		The residue condition for $\omega^D$ follows from \cref{ch_EP_propomegaexists}. The trace compatibility follows by applying \cref{ch_EP_lemTrace2} \ref{ch_EP_lemTrace2_b} with $\psi$ and $\psi'$ equal to $[N]$ and $\varphi$ and $\varphi'$ equal to $[D]$.
	\end{proof}

\part{p-adic interpolation of real-analytic Eisenstein series}
\section[$p$-adic theta functions for sections of the Poincar\'e bundle]{p-adic theta functions for sections of the Poincar\'e bundle}
In \cite{norman_padictheta} P.~Norman discussed constructions of $p$-adic theta functions associated to sections of line bundles on Abelian varieties over algebraically closed $p$-adic fields. A similar method can be used to construct $p$-adic theta functions for sections of the Poincar\'e bundle for Abelian schemes with ordinary reduction over more general $p$-adic base schemes. We will discuss the construction of $p$-adic theta functions for elliptic curves but it immediately generalizes to the higher dimensional case.\par 

Let $p$ be a fixed prime. Let $R$ be a $p$-adic ring, i.\,e. $R$ is complete and separated in its $p$-adic topology, and set $S:=\Spec R$. An elliptic curve $E/S$ will be said to have ordinary reduction if $E\times_S \Spec R/pR$ is fiber-wise an ordinary elliptic curve. For the moment let $n\geq 1$ be a fixed integer. Let $C:=C_n$ resp.~$D:=D_n$ be the connected components of $E[p^n]$ resp.~$E^\vee[p^n]$. Let us write $i: C\hookrightarrow E$ and $j:D\hookrightarrow \Ed$ for the inclusions.  We define
	 \[
	 	\varphi:E\twoheadrightarrow E/C=:E'
	 \]
	 and note that its dual $\varphi^\vee\colon (E')^\vee\rightarrow E^\vee$ is \'{e}tale since we assumed $E/S$ to have ordinary reduction. Let us further write $j'\colon D'\hookrightarrow (E')^\vee$ for the inclusion of the connected component of $(E')^\vee[p^n]$. Since $\varphi^\vee$ is \'etale it induces an isomorphism on connected components of $p^n$-torsion groups, i.e.:
	 \[
	 	\begin{tikzcd}
	 		D'\ar[r,hook,"j'"] \ar[d,"\cong"] & (E')^\vee \ar[d,"\varphi^\vee"]\\
	 		D \ar[r,hook,"j"] & E^\vee 
	 	\end{tikzcd}
	 \]
	 Let us write $\Phi:\Ocal_D\righteq \Ocal_{D'}$ for the induced isomorphism of structure sheaves. In particular by pulling back the Poincar\'e bundle along this diagram we obtain an $\id\times \Phi$-linear isomorphism
	 \[
	 		(\id\times j)^*\Po\righteq (\id\times j')^*(\id\times\varphi^\vee)^*\Po.
	 \]
	 Let us write $\Po'$ for the Poincar\'e bundle on $E'\times_S(E'^\vee)$. Restricting along $(i\times\id)$ and composing with $(i\times\id)^*\gamma_{\id,\varphi^\vee}$ gives an $\id\times\Phi$-linear isomorphism
	\begin{equation}\label{eq_Triv1}
		(i \times j)^*\Po\righteq (i\times j')^*(\id\times\varphi^\vee)^*\Po\righteq (i\times j')^*(\varphi\times\id)^*\Po'.
	\end{equation}
	Since $\varphi\circ i$ factors through the zero section, we have the identity
	\[
		\varphi\circ i=\varphi \circ e\circ \pi_C.
	\]
	Using this, we obtain an $\id\times\Phi^{-1}$-linear isomorphism
	\begin{equation}\label{eq_Triv2}
		(i\times j')^*(\varphi\times\id_{E'^\vee})^*\Po'=(\pi_C\times \id_{D'})^*(e\times\id_{D'})^*(\varphi\times j')^*\Po\stackrel{(a)}{\cong} \pi_C^*\Ocal_{D'}\stackrel{(b)}{\cong} \Ocal_C\otimes_{\Ocal_S}\Ocal_D.
	\end{equation}
	Here, we have used the rigidification of the Poincar\'e bundle in $(a)$ and $\Phi^{-1}$ in $(b)$. Finally, the composition of \eqref{eq_Triv1} and \eqref{eq_Triv2} gives an $\Ocal_{C_n}\otimes_{\Ocal_S}\Ocal_{D_n}$-linear isomorphism
	\[
		\triv_n: \Po|_{C_n\times D_n}=(i\times j)^*\Po \righteq \Ocal_{C_n}\otimes_{\Ocal_S}\Ocal_{D_n}.
	\]
	It is straightforward to check that this isomorphism is compatible with restriction along $C_n\hookrightarrow C_m$ for $n\leq m$. Let us write $\hat{E}$ and $ \Edf$ for the formal groups obtained by completion of $E$ and $\Ed$ along the zero section. By passing to the limit over $n$ we obtain
	\begin{equation}\label{eq_Triv3}
		\triv:\Po|_{\hat{E}\times_S \Edf}\righteq \Ocal_{\hat{E}}\hat{\otimes}_{\Ocal_S}\Ocal_{\Edf}.
	\end{equation}
	Let us introduce the notation $\Pof:=\Po|_{\hat{E}\times_S \Edf}$ for the restriction of the Poincar\'e bundle to $\hat{E}\times_S \Edf$.
	\begin{defin}
		For a section $s\in\Gamma(U,\Po)$ with $U\subseteq E\times_S\Ed$ an open subset containing the zero section $e\times_S e$, let us define the \emph{$p$-adic theta function associated to $s$} by 
		\[
			\vartheta_s:=\triv(s)\in \Gamma(\hat{E}\times_S\Edf, \Ocal_{\hat{E}}\hat{\otimes}_{\Ocal_S}\Ocal_{\Edf}).
		\]
	\end{defin}

\section[$p$-adic Eisenstein-Kronecker series]{p-adic Eisenstein-Kronecker series}
In the first part of this paper we have given a construction of real-analytic Eisenstein series via the Poincar\'e bundle. More precisely we constructed geometric nearly holomorphic modular forms
\[
	E^{k,r}_{s,t} \in \Gamma(S,\Sym^{k+r}\HdR{1}{E/S})
\]
which give rise to the classical Eisenstein-Kronecker series after applying the Hodge-decomposition on the modular curve. It was first observed by Katz in \cite{katz_padicinterpol} that one can get $p$-adic modular forms associated to geometric nearly holomorphic modular forms by applying the unit root decomposition on the universal trivialized elliptic curve instead of the Hodge decomposition. Let us recall this construction. For more details we refer to Katz' paper \cite{katz_padicinterpol}.\par 
Let $R$ be a $p$-adic ring and let us write $S=\Spec R$. A trivialization of an elliptic curve $E/S$ is an isomorphism
	 \[
	 	\beta: \Ef\righteq \Gmf{S}
	 \]
	 of formal groups over $S$. For a natural number $N\geq 1$ coprime to $p$, a trivialized elliptic curve with $\Gamma(N)$-level structure is a triple $(E,\beta,\alpha_N)$ consisting of an elliptic curve $E/S$, a trivialization $\beta$ and a level structure $\alpha_N:(\ZZ/N\ZZ)^2_S\righteq E[N]$. Let $(\Etriv,\beta,\alpha_N)$ be the universal trivialized elliptic curve over the moduli scheme $\Mtriv$ of trivialized elliptic curves of level $\Gamma(N)$. The scheme $\Mtriv$ is affine. Let us write $V_p(\Gamma(N))$ or sometimes just $\VpN$ for the ring of global sections of $\Mtriv$. Following Katz, the ring $\VpN$ will be called ring of generalized $p$-adic modular forms. For more details we refer to \cite[Ch. V]{katz_padicinterpol}. Let us recall the definition of the unit root decomposition. Dividing $\Etriv$ by its canonical subgroup $C$, again gives a trivialized elliptic curve
	\[
		(E'=\Etriv/C,\beta',\alpha_N')
	\]	 
	  with $\Gamma(N)$-level structure over $\Spec \VpN$. The corresponding morphism 
	\[
		\Frob: \VpN\rightarrow \VpN
	\]	 
	 classifying this quotient will be called \emph{Frobenius morphism} of $\Mtriv=\Spec \VpN$. In particular, the quotient map $\Etriv\rightarrow \Etriv'=\Etriv/C$ induces a $\Frob$-linear map
	 \[
	 	F: \Frob^*\HdR{1}{\Etriv/\Mtriv}=\HdR{1}{\Etriv'/\Mtriv}\rightarrow \HdR{1}{\Etriv/\Mtriv}
	 \]  
	 which is easily seen to respect the Hodge filtration
	 \[
	 	\begin{tikzcd}
	 	 0 \ar[r] & \om_{\Etriv/\Mtriv}\ar[r] & \HdR{1}{\Etriv/\Mtriv}\ar[r] & \om_{\Etriv^\vee/\Mtriv}^\vee\ar[r] & 0.
	 	\end{tikzcd}
	 \]
	 Further, the induced $\Frob$-linear endomorphism of $\om_{\Etriv^\vee/\Mtriv}^\vee$ is bijective while the induced $\Frob$-linear map on $\om_{\Etriv/\Mtriv}$ is divisible by $p$. This induces a decomposition
	 \begin{equation}\label{IP_eq1}
	 	\HdR{1}{\Etriv/\Mtriv}=\om_{\Etriv/\Mtriv}\oplus \UR
	 \end{equation}
	 where $\UR\subseteq \HdR{1}{\Etriv/\Mtriv}$ is the unique $F$-invariant $\Ocal_\Mtriv$-submodule on which $F$ is invertible. $\UR$ is called the unit root space and \eqref{IP_eq1} is called unit root decomposition. Let us write
	 \[
	 	u\colon \HdR{1}{\Etriv/\Mtriv}\rightarrow \om_{\Etriv/\Mtriv}
	 \]
	 for the projection induced by the unit root decomposition. Let us further observe, that the isomorphism
	 \[
	 	\beta: \Eftriv\righteq \Gmf{S}
	 \]
	 gives a canonical generator $\omega:=\beta^{-1}(\frac{d T}{1+T})$ of $\om_{\Etriv/\Mtriv}$. This generator gives an isomorphism
	 \[
	 	\om_{\Etriv/\Mtriv}\righteq \Ocal_{\Mtriv}, \quad \omega\mapsto 1.
	 \]
	Let us consider the following variant of the Eisenstein--Kronecker series $E^{k,r+1}_{s,t}$: For a positive integer $D$ let us define
	\[
		\EisD:=\sum_{e\neq t\in \Etriv[D]} E^{k,r+1}_{s,t}
	\]	 
	One could equally well work with the Eisenstein--Kronecker series $E^{k,r+1}_{s,t}$. The main reason to concentrate on $\EisD$ is, that we want to compare the Eisenstein--Kronecker series with the real-analytic Eisenstein series studied by Katz. For this comparison it is convenient to work with the variant $\EisD$. 
	\begin{defin}
	Let $\Etriv/\Mtriv$ be the universal trivialized elliptic curve with $\Gamma(N)$-level structure. Let $D>0$, $(0,0)\neq (a,b)\in\ZZ/N\ZZ$ and $s\in \Etriv[N]$ the associated $N$-torsion section. The $p$-adic Eisenstein-Kronecker series 
	\[
	\EispD\in\VpN
	\]
	are defined as the image of $\EisD$ under the unit root decomposition
	\[
		\Sym^{k+r+1}\HdR{1}{\Etriv/\Mtriv} \twoheadrightarrow \om_{\Etriv/\Mtriv}^{k+r+1} \righteq \Ocal_{\Mtriv}.
	\]
	\end{defin}	 
Katz defines generalized $p$-adic modular forms $2\Phi_{k,r,f}\in\VpN$ for $k,r\geq 1$ and $f:(\ZZ/N\ZZ)^2\rightarrow \Zp$. For the precise definition we refer to \cite[\S 5.11]{katz_padicinterpol}. Essentially, he applies the differential operator
	\[
		 \Theta:\Sym^k \HdR{1}{\Etriv/\Mtriv}\rightarrow \Sym^k \HdR{1}{\Etriv/\Mtriv}\otimes_{\Ocal_\Mtriv} \Omega^1_{\Mtriv/\Zp} \hookrightarrow \Sym^{k+2} \HdR{1}{\Etriv/\Mtriv}
	\]
	obtained by Gau\ss--Manin connection and Kodaira--Spencer isomorphism to classical Eisenstein series and finally uses the unit root decomposition in order to obtain $p$-adic modular forms. We have the following comparison result.
	\begin{prop}\label{PI_propKatz} We have the following equality of $p$-adic modular forms:
	\[
		\EispD=2N^{-k}\left[ D^{k-r+1} \Phi_{r,k,\delta_{(a,b)}}- \Phi_{r,k,\delta_{(Da,Db)}} \right]
	\]
	where $\delta_{(a,b)}$ is the function on $(\ZZ/N\ZZ)^2$ with $\delta_{(a,b)}(a,b)=1$ and zero else.
	\end{prop}
	\begin{proof}
		Since both sides of the equation are given by applying the unit root decomposition to classes in $\Sym^{k+r+1}\HdR{1}{\Etriv/\Mtriv}$ it suffices to compare these classes. This can be done on the universal elliptic curve  of level $\Gamma(N)$. It is further enough to compare the associated $\Ccal^\infty$-modular forms obtained by applying the Hodge decomposition on the universal elliptic curve. The $\Ccal^\infty$-modular form associated with $2\phi_{k,r,f}$ is according to Katz \cite[3.6.5, 3.0.5]{katz_padicinterpol} given by
		 \[
		 	(2\phi_{k,r,f})^{an}=(2G_{k+r+1,-r,f})^{an}=(-1)^{k+r+1}k!\left( \left(\frac{N}{A(\tau)}\right)^r \zeta_{k+r+1}\left( \frac{k-r+1}{2},1,\tau,f \right) \right)
		 \]
		 where $\zeta_{k+r+1}$ is the Epstein zeta function obtained by analytic continuation of
		 \begin{equation}\label{PI_eq4}
		 	\zeta_k(s,1,\tau,f)=N^{2s}\sum_{(0,0)\neq (n,m)}\frac{f(n,m)}{(m\tau+n)^k|m\tau+n|^{2s-k}},\quad \mathrm{Re}(s)>1.
		 \end{equation}
		On the other hand the $\Ccal^\infty$-modular form corresponding to $\EisD$ is according to \cref{EP_MainThm} given by
		\begin{align*}
			&\sum_{(0,0)\neq (c,d)\in (\ZZ/D\ZZ)^2} \tilde{e}_{k,r+1}(\frac{Da}{N}\tau+\frac{Db}{N},\frac{Nc}{D}\tau+\frac{Nd}{D})=\\
			=&\frac{(-1)^{k+r}r!}{A^k}\sum_{(0,0)\neq (c,d)\in (\ZZ/D\ZZ)^2}\sum_{(m,n)\in\ZZ^2} \frac{(\frac{Da}{N}\bar{\tau}+\frac{Db}{N}+m\bar{\tau}+n)^k}{(\frac{Da}{N}\tau+\frac{Db}{N}+m\tau+n)^{r+1}}\langle\gamma,\frac{Nc}{D}\tau+\frac{Nd}{D}\rangle=\\
			=&\frac{(-1)^{k+r}r!}{A^k}\left(\sum_{(m,n)\in\ZZ^2} \frac{(\frac{Da}{N}\bar{\tau}+\frac{Db}{N}+m\bar{\tau}+n)^k}{(\frac{Da}{N}\tau+\frac{Db}{N}+m\tau+n)^{r+1}}\underbrace{\sum_{(c,d)\in (\ZZ/D\ZZ)^2}\exp\left(\frac{2\pi i}{D}N(dm-cn)\right)}_{=0\text{ if }(m,n)\notin (D\ZZ)^2} \right)-\\
			&-\tilde{e}_{k,r+1}(\frac{Da}{N}\tau+\frac{Db}{N},0)=\\
			=&\frac{(-1)^{k+r}r!}{A^k}\left(\sum_{\gamma\in\Gamma} \frac{(\frac{Da}{N}\bar{\tau}+\frac{Db}{N}+Dm\bar{\tau}+Dn)^k}{(\frac{Da}{N}\tau+\frac{Db}{N}+Dm\tau+Dn)^{r+1}}\cdot D^2\right)-\tilde{e}_{k,r+1}(\frac{Da}{N}\tau+\frac{Db}{N},0)=\\
			=&D^{k-r+1}\tilde{e}_{k,r+1}(\frac{a}{N}\tau+\frac{b}{N},0)-\tilde{e}_{k,r+1}(\frac{Da}{N}\tau+\frac{Db}{N},0)
		\end{align*}
		The analytic Eisenstein--Kronecker series $\tilde{e}_{k,r+1}(\frac{a}{N}\tau+\frac{b}{N},0)$ appearing in the description of the $\Ccal^\infty$-modular form $\EisD$ are defined by
 		\begin{align*}
		  \tilde{e}_{k,r+1}(\frac{a}{N}\tau+\frac{b}{N},0)&:=(-1)^{k+r+1}r!\frac{K^*_{k+r+1}(\frac{a}{N}\tau+\frac{b}{N},0,r+1,\tau)}{A(\tau)^k}\stackrel{\text{\cite[Prop. 1.3]{bannai_kobayashi}}}{=}\\
		  &=(-1)^{k+r+1}k!\frac{K^*_{k+r+1}(0,\frac{a}{N}\tau+\frac{b}{N},k+1,\tau)}{A(\tau)^r}
		 \end{align*}
			with the Eisenstein--Kronecker--Lerch $K^*_{k}(0,\frac{a}{N}\tau+\frac{b}{N},s,\tau)$ series which is given by analytic continuation of
		 \begin{equation}\label{PI_eq5}
		 	K^*_{k}\left(0,\frac{a}{N}\tau+\frac{b}{N},s,\tau\right):=\sum_{(0,0)\neq(m,n)} \frac{(m\bar{\tau}+n)^k}{|m\tau+n|^{2s}}\exp\left(2\pi i\frac{ma-nb}{N} \right).
		 \end{equation}
		 Comparing \eqref{PI_eq4} and \eqref{PI_eq5} shows
		 \begin{equation}\label{PI_eq6}
		 	K^*_{k}\left(0,\frac{a}{N}\tau+\frac{b}{N},s+\frac{k}{2},\tau\right)=N^{1-2s}\zeta_k(s,1,\tau,\hat{\delta}_{(a,b)}).
		 \end{equation}
		 Using this, we compute
		 \begin{align*}
		 	N^{-k}(2\phi_{k,r,\hat{\delta}_{a,b}})^{an}&=N^{-k}(-1)^{k+r+1}k!\left( \frac{N}{A(\tau)}\right)^r\zeta_{k+r+1}(\frac{k-r+1}{2},1,\tau,\hat{\delta}_{a,b})=\\
		 	&=N^{-k}(-1)^{k+r+1}k!\left(\frac{N}{A(\tau)}\right)^r N^{k-r}K^*_{k+r+1}(0,s,k+1;\tau)=\\
		 	&=\tilde{e}_{k,r+1}(\frac{a}{N}\tau+\frac{b}{N},0)
		 \end{align*}
		 Finally, let us recall from \cite{katz_padicinterpol} the identity $\phi_{k,r,f}=\phi_{r,k,\hat{f}}$. Now, the analytic identity
		 \begin{align*}
		 	N^{-k}&\left[ D^{k-r+1} \left(2 \phi_{r,k,\delta_{a,b}}\right)^{an}-\left(2\phi_{r,k,\delta_{Da,Db}}\right)^{an} \right]=\\
		 	=&D^{k-r+1}\tilde{e}_{k,r+1}(\frac{a}{N}\tau+\frac{b}{N},0)-\tilde{e}_{k,r+1}(D\frac{a}{N}\tau+D\frac{b}{N},0)
		 \end{align*}
		 proves the desired algebraic identity on the universal elliptic curve and thereby the proposition.
	\end{proof}
\section[$p$-adic Eisenstein--Kronecker series and $p$-adic theta functions]{p-adic Eisenstein--Kronecker series and p-adic theta functions}
Let $N,D$ be positive integers co-prime to $p$. Let us again write $\Etriv/\Mtriv$ for the universal trivialized elliptic curve of level $\Gamma(N)$. Let $s$ be the $N$-torsion section given by $(0,0)\neq(a,b)\in(\ZZ/N\ZZ)^2$. Let us write $\pthetaD_{(a,b)}\in \Gamma(\Eftriv\times_\Mtriv \Edftriv,\Ocal_{\Eftriv\times \Edftriv})$ for the $p$-adic theta function associated to the section
\begin{equation}\label{eq_scanD}
	\sum_{e\neq t \in \Eftriv[D]}U_{s,t}^{N,D}(\scan).
\end{equation}
More precisely: The trivialization $\beta:\Eftriv \righteq \Gmf{\Mtriv}$ gives us a canonical invariant differential $\omega:=\beta^* \frac{d S}{1+S}\in\Gamma(\Mtriv,\om_{\Etriv/\Mtriv})$. Since $\Gamma(\Etriv,\Omega^1_{\Etriv/\Mtriv})=\Gamma(\Mtriv,\om_{\Etriv/\Mtriv})$ we obtain an isomorphism
\[
	\Omega^1_{\Etriv/\Mtriv}\righteq \Ocal_{\Mtriv}.
\]
Using this isomorphism allows us to view
\[
	\sum_{e\neq t \in \Eftriv[D]}U_{s,t}^{N,D}(\scan)\Big|_{\Eftriv\times\Edftriv}
\]
as a section of 
\[
\Gamma(\Eftriv\times\Edftriv, ([D]\times[N])^*\Pof)\cong \Gamma(\Eftriv\times\Edftriv,\Pof).
\]
The last isomorphism is induced by $N$ resp. $D$ multiplication on the formal groups. Finally, we define $\pthetaD_{(a,b)}\in \Gamma(\Eftriv\times_\Mtriv \Edftriv,\Ocal_{\Eftriv\times \Edftriv})$ as the image of \eqref{eq_scanD} under the trivialization map
\[
	\Gamma(\Eftriv\times_\Mtriv \Edftriv,\Pof)\righteq \Gamma(\Eftriv\times_\Mtriv \Edftriv,\Ocal_{\Eftriv\times \Edftriv}).
\]
Let us write
\[
	\partial_{\Eftriv}\colon \Ocal_{\Eftriv} \rightarrow \Ocal_{\Eftriv},\quad \partial_{\Edftriv}\colon \Ocal_{\Edftriv} \rightarrow \Ocal_{\Edftriv}
\]
for the invariant derivations associated to the invariant differential $\omega$. The following result relates invariant derivatives of our $p$-adic theta function $\pthetaD_{(a,b)}$ to the $p$-adic Eisenstein--Kronecker series, thus it can be seen as a $p$-adic version of the Laurent expansion 
\begin{equation*}
		\Theta_{s,t}(z,w)=\sum_{k,r\geq 0}\frac{\tilde{e}_{k,r+1}(s,t)}{k!r!}z^rw^k,\quad s,t\notin\Gamma
\end{equation*}
of the Kronecker theta function due to Bannai and Kobayashi. 

\begin{thm}\label{thm_padictheta}
Let $\Etriv/\Mtriv$ be the universal trivialized elliptic curve of level $\Gamma(N)$. We have the following equality of generalized $p$-adic modular forms
\[
	\EispD=(e\times e)^*\left(\partial_{\Edftriv}^{\circ k} \partial_{\Eftriv}^{\circ r} \pthetaD_{(a,b)}\right).
\]
\end{thm}
We will present the proof in \cref{sec_proofpadic}. As an immediate consequence of this result we get a new construction of Katz' two-variable $p$-adic Eisenstein measure: Let $R$ be a $p$-adic ring. A $R$-valued $p$-adic measure on a pro-finite Abelian group $G$ is an $R$-linear map $C(G,R)\rightarrow R$, where $C(G,R)$ denotes the $R$-module of $R$-valued continuous functions on $G$. Let us write $\Meas(\Zp^2,R)$ for the set of all $R$-valued measures on $\Zp^2$. According to a theorem of Y. Amice there is an isomorphism of $R$-algebras:
\[
	R\llbracket S,T \rrbracket \righteq \Meas(\Zp^2,R), \quad f\mapsto \mu_f
\]
which is uniquely characterized by
\[
	\int_{\Zp^2} x^k y^l\dd \mu_f(x,y)=\left.\partial_S^{\circ k}\partial_T^{\circ l}  f \right|_{S=T=0}
\]
where $\partial_T=(1+T)\frac{\partial}{\partial T}$ and $\partial_S=(1+S)\frac{\partial}{\partial S}$ are the invariant derivations on the two copies of $\Gmf{R}$. Let us call $\mu_f$ the \emph{Amice transform} associated to $f$.\par
The trivialization $\beta: \Eftriv\righteq \Gmf{\Mtriv}$ together with the autoduality $\Eftriv\righteq \Edftriv$ allows us to view $\pthetaD_{(a,b)}$ as a two-variable power series with coefficients in the ring $R:=\VpN$ of generalized $p$-adic modular forms:
\[
	\pthetaD_{(a,b)}(S,T)\in R\llbracket S,T\rrbracket.
\]
Here, $S$ is the variable of $\Gmf{\Mtriv}$ corresponding to $\Etriv$ and the variable $T$ corresponds to the dual elliptic curve. The Amice transform of the $p$-adic theta function $\pthetaD_{(a,b)}$ gives us a $p$-adic measure on $\Zp\times\Zp$ which will be called $\muEisD$. As an immediate corollary of the above result we get the $p$-adic Eisenstein--Kronecker series as moments of the $p$-adic measure $\muEisD$:
\begin{cor}\label{cor_Eis_moments} The $p$-adic Eisenstein-Kronecker series $\EispD$ appear as moments
	\[
		\EispD=\int_{\Zp\times\Zp} x^k y^r \dd \muEisD(x,y)
	\]
	of the measure $\muEisD$ associated to the $p$-adic theta function $\pthetaD_{(a,b)}(S,T)$.
\end{cor}
This corollary gives a more concise construction of the $p$-adic Eisenstein measure then the original one in \cite{katz_padicinterpol}: In \cite{katz_padicinterpol} the existence of the $p$-adic Eisenstein measure has been proven by checking all predicted $p$-adic congruences between the corresponding $p$-adic modular on the $q$-expansion. In our construction, the $p$-adic Eisenstein measure appears more naturally as the Amice transform of a $p$-adic theta function. We obtain the $p$-adic congruences for free as a by product without checking them in the first instance.

\section{The geometric logarithm sheaves}
For the proof of \cref{thm_padictheta} it is necessary to study the structure of the restrictions of the Poincar\'e bundle to $\Ef\times_S \Edf$ more carefully. In his PhD thesis \cite{rene} Scheider has proven that the de Rham logarithm sheaves appear naturally by restricting the Poincar\'e bundle $\Po^\dagger$ to infinitesimal thickenings of the elliptic curve. At this place it is not necessary to develop the theory of the de Rham logarithm sheaves, but keeping this relation in mind motivates many of the properties of $\hat{\Po}:=\Po|_{\Ef\times \Edf}$.
\subsection{Basic properties} Let $E/S$ be an elliptic curve over a $p$-adic ring $S=\Spec R$ with fiber-wise ordinary reduction. As before, let us write $E^\dagger$ for the universal vectorial extension of the dual of the elliptic curve $E$. The pullback of the Poincar\'e bundle $\Po$ to $E\times_{S} E^\dagger$ is denoted by $\Po^\dagger$ and carries the universal integrable connection $\nabla_{\dagger}$. Motivated by Scheider's results, let us define
\[
	\Lnf:=(\pr_{\Ef})_* \left(\Po|_{\Ef\times_S \Inf^n_e \Ed}\right)
\]
and
\[
	\Lnf^\dagger:=(\pr_{\Ef})_* \left(\Po^\dagger|_{\Ef\times_S \Inf^n_e E^\dagger}\right).
\]
The connection on $\Po^\dagger$ induces an $\Ocal_S$-linear connection $\nabla_\dagger^{(n)}$ on the $\Ocal_{\Ef}$-module $\Lnf^\dagger$. Let us write $\Hcal_{\Ef}:=\HdR{1}{\Ed/S}\otimes\Ocal_{\Ef}$ and $\om_{\Ef}:=\om_{\Ed/S}\otimes\Ocal_{\Ef}$. Since $\Lnf$ is obtained by restriction of $\Po$ to $\Ef\times_S\Inf^n_e\Ed$, we obtain transition maps
\[
	\Lnf\twoheadrightarrow \hat{\Lcal}_{n-1}
\]
by further restriction along $\Ef\times_S\Inf^{n-1}_e \Ed\hookrightarrow \Ef\times_S\Inf^n_E\Ed$. Similarly, we obtain
\[
	\Lnf^\dagger\twoheadrightarrow \hat{\Lcal}_{n-1}^\dagger.
\]
The fact that $\Po^\dagger$ is the pullback of $\Po$ along $E\times_S \Ed\twoheadrightarrow E\times_S E^\dagger$ gives inclusions
\[
	\Lnf\hookrightarrow \Lnf^\dagger.
\]
The decompositions $\Ocal_{\Inf^1_e \Ed}=\Ocal_S\oplus \om_{\Ed/S}$ and $\Ocal_{\Inf^1_e E^\dagger}=\Ocal_S\oplus \Hcal$ induce short exact sequences
\[
	\begin{tikzcd}
		0\ar[r] & \om_{\Ef}\ar[r] & \widehat{\Lcal}_1\ar[r] & \Ocal_{\Ef}\ar[r] & 0
	\end{tikzcd}
\]
and
\begin{equation}\label{eq_ses1}
	\begin{tikzcd}
		0\ar[r] & \Hcal_{\Ef}\ar[r] & \widehat{\Lcal}^\dagger_1\ar[r] & \Ocal_{\Ef}\ar[r] & 0.
	\end{tikzcd}
\end{equation}
The maps in the last exact sequence are horizontal if we equip $\Ocal_{\Ef}$ and $\Hcal_{\Ef}$ with the trivial pullback connections relative $S$, i.e.
\[
	\Hcal_{\Ef}=\Ocal_{\Ef}\otimes_{\Ocal_S}\Hcal\rightarrow \Omega^1_{\Ef/S}\otimes_{\Ocal_S} \Hcal,\quad (f\otimes h)\mapsto df\otimes h
\]
and similarly for $\Ocal_{\Ef}$. In \eqref{eq_Triv3} we have defined a trivialization isomorphism
\[
	\triv\colon \Pof\righteq \Ocal_{\Ef}\hat{\otimes}\Ocal_{\Edftriv}.
\]
This isomorphism induces a $\Ocal_{\Ef}$-linear trivialization map
\[
	\triv^{(n)}\colon \Lnf\righteq \Ocal_{\Ef}\otimes_{\Ocal_S} \Ocal_{\Inf^n_e \Ed}.
\]
Since $\Po^\dagger$ is the pullback of $\Po$ we also obtain
\[
	\triv^{(n)}\colon \Lnf^\dagger\righteq \Ocal_{\Ef}\otimes_{\Ocal_S} \Ocal_{\Inf^n_e E^\dagger}.
\]
Let us observe that in the special case $n=1$ the trivialization map splits the above short exact sequences, i.e. we get
\[
	\triv^{(1)}\colon  \widehat{\Lcal}_1\righteq \Ocal_{\Ef}\oplus \om_{\Ef}
\]
and
\[
	\triv^{(1)}\colon  \widehat{\Lcal}^\dagger_1\righteq \Ocal_{\Ef}\oplus \Hcal_{\Ef}.
\]
The last map is not horizontal if we equip the right hand side with the trivial $S$-connections.

\subsection{Comultiplication maps} In this subsection let us introduce certain canonical comultiplication maps on the infinitesimal geometric logarithm sheaves. Using the $\Gm{S}$-biextension structure of the Poincar\'e bundle let us construct certain natural comultiplication maps on $\Lnf$ and $\Lnf^\dagger$. Such a construction already appeared in the PhD thesis of Ren\'e Scheider \cite[\S 2.4.2]{rene}. As before, let $E/S$ be an elliptic curve over a $p$-adic ring $S=\Spec R$ with fiber-wise ordinary reduction and let us write $E^\dagger$ for the universal vectorial extension of $\Ed$. Let
	\[
		\iota^{\dagger}_n:E_n^\dagger:=\Inf^n_e E^{\dagger}\hookrightarrow E^{\dagger}
	\]
	denote the inclusion of the $n$-th infinitesimal neighbourhood $E_n^\dagger$ of $e$ in $E^\dagger$. For the time being we will use the convention to denote by $\times$ and $\otimes$ the product and tensor product over $S$. Recall that the Poincar\'{e} bundle $\Po^{\dagger}$ is equipped with a natural $\Gm{S}$-biextension structure, i.\,e. isomorphisms
	\begin{alignat}{3}\label{GL_eq1}
		(\mu_E\times \id_{E^{\dagger}})^* \Po^{\dagger} &\righteq \pr_{1,3}^*\Po^\dagger \otimes \pr_{2,3}^*\Po^{\dagger}  &&\quad\text{ on }\quad E\times E\times E^{\dagger} \\
		(\id_E\times \mu_{E^{\dagger}})^* \Po^{\dagger} &\righteq \pr_{1,2}^*\Po^\dagger \otimes \pr_{1,3}^*\Po^{\dagger}  &&\quad\text{ on }\quad E\times E^{\dagger}\times E^{\dagger}\nonumber
	\end{alignat}
	satisfying certain compatibilities, cf. \cite[exp. VII]{SGA7_I}. Here, $\mu$ denotes the multiplication and $\pr_{i,j}$ the projection on the $i$-th and $j$-th component of the product. Now, fix some integers $n,m\geq 1$ and define $\Po^{\dagger}_n:=(\id\times \iota^{\dagger}_n)^*\Po^{\dagger}$. Restricting
	\[
		\mu_{E^{\dagger}}: E^{\dagger}\times E^{\dagger}\rightarrow E^{\dagger}
	\]
	to $E_n^{\dagger}\times E^{\dagger}_m$ gives 
	\[
		\mu_{n,m}: E_n^{\dagger}\times E^{\dagger}_m\rightarrow E^{\dagger}_{n+m}.
	\]
	Restricting \eqref{GL_eq1} along
	\[
		E\times E_n^{\dagger}\times E_m^{\dagger} \hookrightarrow E\times E^{\dagger}\times E^{\dagger}
	\]
	results in
	\[
		\Po^{\dagger}_{n+m}\rightarrow (\pr_{12})^*\Po^{\dagger}_n \otimes_{\Ocal_{E\times E_n^{\dagger}\times E_m^{\dagger} }} (\pr_{13})^*\Po^{\dagger}_m.
	\]
	Using the unit of the adjunction between $(\id\times \mu_{n,m})_*$ and $(\id\times \mu_{n,m})^*$, we obtain
	\[
		\Po^{\dagger}_{n+m}\rightarrow (\id\times \mu_{n,m})_*\left[ (\pr_{12})^*\Po^{\dagger}_n\otimes  (\pr_{13})^*\Po^{\dagger}_m \right].
	\]
	Taking the pushforward along $\pr_E$ gives:
	\[
		\xi_{n,m}: \widehat{\Lcal}_{n+m}^\dagger\rightarrow \widehat{\Lcal}_{n}^\dagger\otimes_{\Ocal_E}\widehat{\Lcal}_{m}^\dagger
	\]
	Since the $\Gm{S}$-biextension structure is compatible with the connection, we get that $\xi_{n,m}$ is horizontal. Using the compatibilities of the $\Gm{S}$-biextension structure, one deduces the following commutative diagrams:
	\begin{equation}\label{GL_eq2}
		\begin{tikzcd}
			\widehat{\Lcal}_{n+m}^\dagger \ar[r,"\xi_{n,m}"]\ar[rd,"\xi_{m,n}",swap] & \widehat{\Lcal}_{n}^\dagger \otimes_{\Ocal_E} \widehat{\Lcal}_{m}^\dagger \ar[d,"can"]  \\
			& \widehat{\Lcal}_{m}^\dagger \otimes_{\Ocal_E} \widehat{\Lcal}_{n}^\dagger
		\end{tikzcd}
	\end{equation}
	and 
	\begin{equation}\label{GL_eq3}
		\begin{tikzcd}
			\widehat{\Lcal}_{n+m+l}^\dagger \ar[r,"\xi_{n+m,l}"]\ar[d,"\xi_{n,m+l}",swap] & \widehat{\Lcal}_{n+m}^\dagger \otimes_{\Ocal_E} \widehat{\Lcal}_{l}^\dagger \ar[d,"\xi_{n,m}\otimes \id"]  \\
			 \widehat{\Lcal}_{n}^\dagger \otimes_{\Ocal_E} \widehat{\Lcal}_{l+m}^\dagger \ar[r,"\id\otimes \xi_{m,l}"] & \widehat{\Lcal}_{n}^\dagger \otimes_{\Ocal_E} \widehat{\Lcal}_{m}^\dagger \otimes_{\Ocal_E} \widehat{\Lcal}_{l}^\dagger.
		\end{tikzcd}
	\end{equation}
	Thus, we obtain well-defined maps
	\[
		\widehat{\Lcal}_{n}^\dagger \rightarrow \underbrace{\widehat{\Lcal}_{1}^\dagger\otimes_{\Ocal_E} ... \otimes_{\Ocal_E} \widehat{\Lcal}_{1}^\dagger }_{n \text{ times}}.
	\]
	The diagram \eqref{GL_eq2} shows that this map is invariant under transposing any of the $n$ factors on the right hand side. Thus, letting the symmetric group $S_n$ act by permuting the factors we see that $\Lnf^\dagger \rightarrow \left( \widehat{\Lcal}_1^\dagger \right)^{\otimes n}$ factors through the invariants of the $S_n$ action. We denote the resulting map by
	\begin{equation}\label{GLdagger_eq4}
		\widehat{\Lcal}_{n}^\dagger \hookrightarrow \TSym^n_{\Ocal_{\Ef}} \widehat{\Lcal}_{1}^\dagger:=\left[ \left( \widehat{\Lcal}_{1}^\dagger\right)^{\otimes n} \right]^{S_n}.
	\end{equation}	 
	This map is horizontal, if we equip the right hand side with the tensor product connection induced by $\nabla^{(1)}_\dagger$. Similarly, we get $\Ocal_{\Ef}$-linear maps
	\begin{equation}\label{GL_eq4}
		\widehat{\Lcal}_{n} \hookrightarrow \TSym^n_{\Ocal_{\Ef}} \widehat{\Lcal}_{1}.
	\end{equation}	
	A similar construction applies to the multiplication of the formal groups $\Ef^\vee$:	
	\[
		\Ocal_{\Edftriv}\rightarrow \Ocal_{\Edftriv}\hat{\otimes} \Ocal_{\Edftriv}.
	\]
	This multiplication induces comultiplication maps
	\begin{equation}\label{eq_formal_comult}
		\Ocal_{\Inf^n_e \Ed}\rightarrow \TSym^n_{\Ocal_S} \Ocal_{\Inf^1_e\Ed}
	\end{equation}
	which are compatible with the comultiplication maps on $\widehat{\Lcal}_{n}$:
	\begin{equation}\label{eq_comult}
		\begin{tikzcd}
			\widehat{\Lcal}_{n}\ar[d,"\cong"]\ar[r] &  \TSym^n_{\Ocal_{\Ef}} \widehat{\Lcal}_{1} \ar[d,"\cong"] \\
			\Ocal_{\Ef}\otimes_{\Ocal_S}\Ocal_{\Inf^n_e E^\vee}\ar[r] & \TSym^n_{\Ocal_{\Ef}}\left( \Ocal_{\Ef}\otimes_{\Ocal_S}\Ocal_{\Inf^1_e E^\vee}\right).
		\end{tikzcd}
	\end{equation}
	In this diagram the lower horizontal map is induced by tensoring \eqref{eq_formal_comult} with $\Ocal_{\Ef}$. The comultiplication maps for $\Ef^\vee$ can be identified with taking iterated invariant derivatives. More precisely, the map
	\[
		\Ocal_{\Inf^n_e \Ed}\rightarrow \TSym^n_{\Ocal_S} \Ocal_{\Inf^1_e\Ed}= \TSym^n_{\Ocal_S} (\Ocal_S\oplus \om_{\Ed/S})\cong \bigoplus_{k=0}^n \TSym^k_{\Ocal_S}\om_{\Ed/S}
	\]
	coincides with the map $f\mapsto \left(e^*(\partial^{\circ k} f\right)_{k=0}^n$ where $\partial$ is the map induced by the invariant derivative $\partial\colon \Ocal_{\Edf}\rightarrow \om_{\Ed/S}\otimes\Ocal_{\Edf}$.
	\begin{lem}\label{lem_TrivializeLn} Let us assume that $S=\Spec R$ is flat over $\Zp$. Under this assumption, the comultiplication maps
		\[
			\widehat{\Lcal}_{n} \hookrightarrow \TSym^n \widehat{\Lcal}_{1},\quad \widehat{\Lcal}^\dagger_{n} \hookrightarrow \TSym^n \widehat{\Lcal}^\dagger_{1}
		\]
		are injective and isomorphisms on the generic fiber $E_{\Qp}:=E\times_S\Spec \Qp$
		\[
			\widehat{\Lcal}_{n,E_{\Qp}} \righteq \TSym^n \widehat{\Lcal}_{1,E_{\Qp}},\quad \widehat{\Lcal}^\dagger_{n,E_{\Qp}} \righteq \TSym^n \widehat{\Lcal}^\dagger_{1,E_{\Qp}}.
		\]
	\end{lem}
	\begin{proof}
		We give the proof for $\widehat{\Lcal}_{n}$. The proof for $\widehat{\Lcal}^\dagger_{n}$ is completely analogous. By \eqref{eq_comult} it is enough to prove that the map
		\[
			\Ocal_{\Inf^n_e E^\vee}\rightarrow \TSym^n_{\Ocal_{S}}\left( \Ocal_{\Inf^1_e E^\vee}\right)
		\]
		is injective and an isomorphism after inverting $p$. The isomorphism $\Ocal_{\Inf^1_e E^\vee}\cong \Ocal_{S}\oplus \om_{\Ed/S}$ gives
		\[
			\Ocal_{\Inf^n_e E^\vee}\rightarrow \TSym^n_{\Ocal_{S}}\left( \Ocal_{\Inf^1_e E^\vee}\right)=\bigoplus_{k=0}^n \TSym^k_{\Ocal_S}\om_{\Ed/S}.
		\]
		But this map is just the map sending $f\in\Ocal_{\Inf^n_e E^\vee}$ to $(e^*(\partial^{\circ k} f) )_{k=0}^n$ which is injective if $R$ is flat over $\Zp$ and an isomorphism if $p$ is invertible.
	\end{proof}
	By combining the comultiplication with the trivialization $\triv^{(1)}\colon \widehat{\Lcal}_{1}\righteq \Ocal_{\Eftriv}\oplus \om_{\Eftriv}$ we obtain
	\begin{equation}\label{eq_Lntriv}
		\widehat{\Lcal}_{n}\rightarrow \TSym^n_{\Ocal_{\Ef}} \widehat{\Lcal}_{1}=\TSym^n_{\Ocal_{\Ef}} (\Ocal_{\Ef}\oplus \om_{\Ef})\righteq \bigoplus_{k=0}^n \TSym^k_{\Ocal_{\Ef}} \om_{\Ef}.
	\end{equation}
	and
	\begin{equation}\label{eq_Lndtriv}
		\widehat{\Lcal}^\dagger_{n}\rightarrow \TSym^n_{\Ocal_{\Ef}} \widehat{\Lcal}^\dagger_{1}=\TSym^n_{\Ocal_{\Ef}}( \Ocal_{\Ef}\oplus \Hcal_{\Ef})\righteq \bigoplus_{k=0}^n \TSym^k_{\Ocal_{\Ef}} \Hcal_{\Ef}.
	\end{equation}
	In particular, whenever we are in a situation where $\Hcal$ and $\om_{\Ed/S}$ can be generated by global sections, such generators and \cref{lem_TrivializeLn} give us an explicit $\Ocal_{\Ef}$-basis of $\Lnf$ and $\Lnf^\dagger$ on the generic fiber.  
\subsection{The Frobenius structure}
Let $\Etriv/\Mtriv$ be the universal trivialized elliptic curve with $\Gamma(N)$-level structure. Our next aim is to define a \emph{Frobenius structure} on $\Lnf$, i.e. an isomorphism
\[
	 	\Lnf\righteq \phi_{\Eftriv}^* \Lnf
\]
where $\phi_{\Eftriv}$ is a Frobenius lift. Let $\varphi:\Etriv\rightarrow \Etriv'=\Etriv/C$ be the quotient of the universal trivialized elliptic curve by its canonical subgroup. Since $\Etriv'$ is again a trivialized elliptic curve with $\Gamma(N)$-level structure, it is the pullback of $E/M$ along a unique map $\Frob\colon M\rightarrow M$. This gives us a diagram 	 
	 \begin{equation}\label{PI_eq7}
	 	\begin{tikzcd}
	 		\Etriv \ar[r,"\varphi"]\ar[rd,"\pi"] & \Etriv'\ar[r]\ar[r,"\widetilde{\mathrm{Frob}}"]\ar[d,"\pi_{\Etriv'}"] & \Etriv\ar[d,"\pi"]\\
	 		& \Mtriv \ar[r,"\Frob"] & \Mtriv
	 	\end{tikzcd}
	 \end{equation}
	 with the square being Cartesian. Let us define 	 
	 \[
	 \phi_{\Eftriv}:=\left(\widetilde{\mathrm{Frob}}\circ \varphi\right)\Big|_{\Eftriv}:\Eftriv\rightarrow \Eftriv
	 \]
	 as the restriction of the upper horizontal composition in the above diagram to the formal group $\Eftriv$. The map $\phi_{\Eftriv}$ gives us a Frobenius lift on the formal group $\Eftriv$. We want to construct an $\Ocal_{\Eftriv}$-linear isomorphism
	 \[
	 	(\pr_{\Eftriv})_*\Pof \righteq (\pr_{\Eftriv})_*(\phi_{\Eftriv}\times \id_{\Edftriv})^*\Pof.
	 \]
	We will do this in two steps: Let us write $\Pof'$ for the Poincar\'e bundle of $\Etriv'$ restricted to the formal scheme $\widehat{\Etriv'}\times_\Mtriv \widehat{\Etriv'}^\vee$.  Restricting the map $\gamma_{\id,\varphi^\vee}$ to $\widehat{\Etriv'}\times_\Mtriv \widehat{\Etriv'}^\vee$ gives the isomorphism
	\[
		(\id\times\varphi^\vee|_{\widehat{\Etriv'}^\vee})^*\Pof\righteq (\varphi|_{\Eftriv}\times\id)^*\Pof'
	\]
	 The dual isogeny $\varphi^\vee$ is \'etale, hence it induces an isomorphism of formal groups over $\Mtriv$:
	\[
		\varphi^\vee|_{(\widehat{\Etriv'})^\vee}\colon (\widehat{\Etriv'})^\vee\righteq \Edftriv.
	\]
	We get
	\begin{equation}\label{eq_Phi1}
		(\pr_{\Eftriv})_*\Po\Big|_{\Eftriv\times\Edftriv}\righteq (\pr_{\Eftriv})_* (\varphi|_{\Eftriv}\times\id)^*\left(\Po'\Big|_{\widehat{\Etriv'}\times \widehat{\Etriv'}^\vee}\right).
	\end{equation}
	On the other hand, by the compatibility of the Poincar\'e bundle with base change along the Cartesian diagram
	 \begin{equation*}
	 	\begin{tikzcd}
	 		\Etriv'\ar[r]\ar[r,"\widetilde{\mathrm{Frob}}"]\ar[d,"\pi_{\Etriv'}"] & \Etriv\ar[d,"\pi"]\\
	 		\Mtriv'=\Mtriv \ar[r,"\Frob"] & \Mtriv,
	 	\end{tikzcd}
	 \end{equation*}
	 and using the identification $\Etriv'\times_{\Mtriv'} \Etriv'^\vee=\Etriv'\times_\Mtriv \Etriv^\vee$ we get an isomorphism
	 \begin{equation}\label{eq_Phi2}
	 	(\widetilde{\mathrm{Frob}}\times_\Mtriv \id_{\Edtriv})^*\Po\righteq \Po'.
	 \end{equation}
	 Composing \eqref{eq_Phi1} with \eqref{eq_Phi2} gives the desired isomorphism of $\Ocal_{\Eftriv}$-modules:
	  \[
	 	(\pr_{\Eftriv})_*\Pof \righteq (\pr_{\Eftriv})_*(\phi_{\Eftriv}\times \id_{\Edftriv})^*\Pof.
	 \]
	 Replacing $(\pr_{\Eftriv})_*\Pof$ by $\Lnf=(\pr_{\Eftriv})_*(\Pof|_{\Eftriv\times \Inf^n_e \Edtriv})$ in the above construction gives an $\Ocal_{\Eftriv}$-linear isomorphism
	 \[
	 	\Psi\colon \Lnf\righteq \phi_{\Eftriv}^* \Lnf.
	 \]
	 Let us write $\Po^{\dagger}$ for the pullback of $\Po$ along $\Etriv\times_\Mtriv\Etriv^\dagger\rightarrow \Etriv\times_\Mtriv \Edtriv$, where $\Etriv^\dagger$ is the universal vectorial extension of $\Edtriv$. Let us write $\Pof^\dagger$ for the restriction along the formal completion $\Eftriv\times_\Mtriv\Eftriv^\dagger$. Along the same lines as above, we obtain $\Ocal_{\Eftriv}$-linear morphisms
	 \begin{equation}\label{eq_Phidagger}
	 	(\pr_{\Eftriv})_*\Pof^\dagger\rightarrow (\pr_{\Eftriv})_*(\phi_{\Eftriv}\times \id_{\Eftriv^\dagger})^*\Pof^\dagger.
	 \end{equation}
	 and
	 \begin{equation}\label{eq_Phidagger_Ln}
	 	\Psi\colon\Lnf^\dagger\rightarrow \phi_{\Eftriv}^*\Lnf^\dagger.
	 \end{equation}
	 This map is horizontal if both sides are equipped with the canonical $\Mtriv$-connections.\par 

\subsection{A basis for the geometric logarithm sheaves} Let $\Etriv/\Mtriv$ be the universal trivialized elliptic curve with $\Gamma(N)$-level structure. Now let us construct a canonical basis for the infinitesimal geometric logarithm sheaves on the universal trivialized elliptic curve. In  \eqref{eq_Lntriv} and \eqref{eq_Lndtriv} we have defined maps
	\[
		\widehat{\Lcal}_{n}\rightarrow \bigoplus_{k=0}^n \TSym^k_{\Ocal_{\Eftriv}} \om_{\Eftriv}
	\]
	and
	\[
		\widehat{\Lcal}^\dagger_{n}\rightarrow  \bigoplus_{k=0}^n \TSym^k_{\Ocal_{\Eftriv}} \Hcal_{\Eftriv}.
	\]
	The rigidification $\Edftriv\righteq \Eftriv\righteq \Gmf{\Mtriv}$ gives a canonical generator $\omega\in\Gamma(\Mtriv,\om_{\Edtriv/\Mtriv})$. Now $\omega$ generates $\om_{\Eftriv}$ as $\Ocal_{\Eftriv}$-module. The tensor symmetric algebra $\TSym^\bullet \om_{\Eftriv}$ is a graded ring with divided powers given by
\[
	(\cdot)^{[k]}\colon  \om_{\Eftriv}\rightarrow \TSym^{k} \om_{\Eftriv}, \quad x\mapsto x^{[k]}:=\underbrace{x\otimes...\otimes x}_{k\text{-times}}.
\]	
	for any positive integer $k$. Using the isomorphism $\widehat{\Lcal}_1\cong \Ocal_{\Eftriv}\oplus \om_{\Eftriv}$, we may view $\omega$ as a section of $\widehat{\Lcal}_1$. The divided power structure give a canonical $\Ocal_{\Eftriv}$-basis $(\hat{\omega}^{[k]})_{k=0}^n$ of $\TSym^n \widehat{\Lcal}_{1}$ defined by
	\[
		\hat{\omega}^{[k]}:=\omega^{[k]}\in \Gamma(\Eftriv,\TSym^n \widehat{\Lcal}_{1}).
	\]
	Similarly, let us write $[\omega]$ for the image of $\omega$ under the inclusion of the Hodge filtration:
	\[
		\om_{\Edtriv/\Mtriv}\hookrightarrow \Hcal.
	\]
	There is a unique section $[u]\in\Gamma(\Mtriv,\UR)$ in the unit root part of $\Hcal$ with $\langle [u],[\omega]\rangle=1$. This gives a basis $([\omega],[u])$ of $\Hcal$. Let us write
	\[
		\hat{\omega}^{[k,l]}:=[\omega]^{[k]}\cdot [u]^{[l]}\in \Gamma(\Eftriv,\TSym^n \widehat{\Lcal}^\dagger_{1})
	\]
	for the basis induced by $([\omega],[u])$ using the divided power structure on the tensor symmetric algebra $\TSym^{\bullet} \widehat{\Lcal}^\dagger_{1}$.
	\begin{lem}
		We have canonical $\Ocal_{\Eftriv}$-linear decompositions:
		\[
			\widehat{\Lcal}_{n,\Etriv_{\Qp}}\righteq \bigoplus_{k=0}^n \hat{\omega}^{[k]}\cdot \Ocal_{\Eftriv_{\Qp}},\quad \widehat{\Lcal}^\dagger_{n,\Etriv_{\Qp}}\righteq \bigoplus_{k+l\leq n} \hat{\omega}^{[k,l]}\cdot \Ocal_{\Eftriv_{\Qp}}.
		\]
		These decompositions are compatible with the transition maps 
		\[ 
		\widehat{\Lcal}_{n,\Etriv_{\Qp}}\twoheadrightarrow \widehat{\Lcal}_{n-1,\Etriv_{\Qp}},\quad \widehat{\Lcal}^\dagger_{n,\Etriv_{\Qp}}\twoheadrightarrow \widehat{\Lcal}^\dagger_{n-1,\Etriv_{\Qp}}
		\] 
		and the inclusion $\widehat{\Lcal}_{n,\Etriv_{\Qp}}\hookrightarrow \widehat{\Lcal}^\dagger_{n,\Etriv_{\Qp}}$.
	\end{lem}
	\begin{proof}
		The decomposition is an immediate consequence of the bases $(\hat{\omega}^{[k,l]})_{k+l\leq n}$ and $(\hat{\omega}^{[k]})_{k=0}^n$ of $\TSym^{n} \widehat{\Lcal}^\dagger_{1}$ and $\TSym^{n} \widehat{\Lcal}_{1}$ together with \cref{lem_TrivializeLn}. The compatibility with transition maps and the canonical inclusion follow by tracing back the definitions.
	\end{proof}

	\begin{lem}
		The Frobenius structure
		\[
			\Psi\colon\Lnf^\dagger\rightarrow \phi_{\Eftriv}^*\Lnf^\dagger
		\]
		is explicitly given by the formula
		\[
			\Psi(\hat{\omega}^{[k,l]})=p^l\cdot \phi_{\Eftriv}^*\hat{\omega}^{[k,l]}.
		\]
	\end{lem}
	\begin{proof}
		By the commutativity of the diagram
		\[
			\begin{tikzcd}
				\Lnf^\dagger\ar[r]\ar[d,"\Psi"] & \TSym^n\hat{\Lcal}_1^\dagger\ar[d,"\TSym^n\Psi"]\\
				\phi_{\Eftriv}^*\Lnf^\dagger\ar[r] & \TSym^n\phi_{\Eftriv}^*\hat{\Lcal}_1^\dagger
			\end{tikzcd}
		\]
		we are reduced to prove the formulas in the case $n=1$, i.e.
			\begin{align*}
	 	\Psi(\hat{\omega}^{[0,0]})&=\phi_{\Eftriv}^*\hat{\omega}^{[0,0]}\\
	 	\Psi(\hat{\omega}^{[1,0]})&=\phi_{\Eftriv}^*\hat{\omega}^{[1,0]}\\
	 	\Psi(\hat{\omega}^{[0,1]})&=p\cdot\phi_{\Eftriv}^*\hat{\omega}^{[0,1]}.
	 \end{align*}
	 Since $\Psi$ is compatible with the decomposition $\hat{\Lcal}_1^\dagger=\Ocal_{\Eftriv}\oplus \Hcal_{\Eftriv}$, the Frobenius structure on 
	 $\Ocal_{\Eftriv}$ is the canonical isomorphism $\Ocal_{\Eftriv}\cong \phi_{\Eftriv}^*\Ocal_{\Eftriv}$ induced by $\phi_{\Eftriv}\colon\Eftriv\rightarrow \Eftriv$. Since $\hat{\omega}^{[0,0]}$ corresponds to $1\in \Ocal_{\Eftriv}$, we get $\Psi(\hat{\omega}^{[0,0]})=\phi_{\Eftriv}^*\hat{\omega}^{[0,0]}$. The restriction of $\Psi$ to $\Hcal_{\Eftriv}$ is by construction the $\Ocal_{\Eftriv}$-linear extension of the map
	 \[
	 	\Hcal=\HdR{1}{\Edtriv/\Mtriv}\rightarrow \HdR{1}{(\Etriv')^\vee/\Mtriv}\cong \Frob^* \HdR{1}{\Edtriv/\Mtriv}
	 \]
	 induced by the diagram
	 	 \begin{equation*}
	 	\begin{tikzcd}
	 		\Etriv \ar[r,"\varphi"]\ar[rd,"\pi"] & \Etriv'\ar[r]\ar[r,"\widetilde{\mathrm{Frob}}"]\ar[d,"\pi_{\Etriv'}"] & \Etriv\ar[d,"\pi"]\\
	 		& \Mtriv \ar[r,"\Frob"] & \Mtriv.
	 	\end{tikzcd}
	 \end{equation*}
	 Since $\hat{\omega}^{[0,1]}$ corresponds to $[u]$ and $\hat{\omega}^{[1,0]}$ corresponds to $[\omega]$, we have to prove that $[u]$ maps to $p\cdot \Frob^*[u]$ and $[\omega]$ maps to $\Frob^*[\omega]$. But this is known to be true, see for example \cite[p. 23]{bannai_kings}.
	\end{proof}

	\begin{lem}
		The $\Mtriv$-connection on $\widehat{\Lcal}^\dagger_{1}$ induces an $\Mtriv$-connection $\nabla_{(n)}$ on $\TSym^n \widehat{\Lcal}^\dagger_{1}$ which is explicitly given by the formula
		\[
			\con{(n)}(\hat{\omega}^{[k,l]})=c\cdot (l+1)\hat{\omega}^{[k,l+1]}\otimes \omega
		\]
		for some $c\in \Zp$.
	\end{lem}
	\begin{proof}
	The connection on $\TSym^n \widehat{\Lcal}^\dagger_{1} $ is the tensor product connection induced from $\widehat{\Lcal}^\dagger_{1}$. Thus, it is enough to prove the formula in the case $n=1$. Here, $\TSym^1\widehat{\Lcal}^\dagger_{1}=\widehat{\Lcal}^\dagger_{1} $ and $\con{(1)}=\nabla_{\dagger}^{(1)}$. By the horizontality of \eqref{eq_ses1}, the $\Mtriv$-connection is trivial on $\Hcal_{\Eftriv}$, i.e. $\nabla^{(1)}_\dagger(\hat{\omega}^{[1,0]})=\nabla^{(1)}_\dagger(\hat{\omega}^{[0,1]})=0$. It remains to determine $\nabla^{(1)}_\dagger(\hat{\omega}^{[0,0]})$. We will use the horizontality of 
	\[
		\Psi\colon \widehat{\Lcal}^\dagger_{1} \rightarrow \phi_{\Eftriv}^* \widehat{\Lcal}^\dagger_{1}
	\]	
	 to determine $\nabla^{(1)}_\dagger(\hat{\omega}^{[0,0]})$: In terms of the basis $\hat{\omega}^{[*,*]}$ we have
	\begin{align*}
	 	\Psi(\hat{\omega}^{[0,0]})&=\phi_{\Eftriv}^*\hat{\omega}^{[0,0]}\\
	 	\Psi(\hat{\omega}^{[1,0]})&=\phi_{\Eftriv}^*\hat{\omega}^{[1,0]}\\
	 	\Psi(\hat{\omega}^{[0,1]})&=p\cdot\phi_{\Eftriv}^*\hat{\omega}^{[0,1]}.
	 \end{align*}
	We already know that $$\nabla^{(1)}_\dagger(\hat{\omega}^{[0,0]})=f_{1,0}\cdot \hat{\omega}^{[1,0]} + f_{0,1}\cdot \hat{\omega}^{[0,1]}$$ for suitable $f_{0,1},f_{1,0}\in\Gamma(\Eftriv,\Ocal_{\Eftriv})$. Now, the horizontality of \eqref{eq_Phidagger} expresses as the explicit formula
	\[
		f_{1,0}\cdot \phi_{\Eftriv}^*\hat{\omega}^{[1,0]} + p\cdot f_{0,1}\cdot \phi_{\Eftriv}^*\hat{\omega}^{[0,1]}=p\cdot \phi_{\Eftriv}^*\left(f_{1,0}\hat{\omega}^{[1,0]}\right)+p\cdot \phi_{\Eftriv}^*\left(f_{0,1}\hat{\omega}^{[0,1]}\right).
	\]
	 i.e. $f_{1,0}=p\cdot \phi_{\Eftriv}^*f_{1,0}$ and $f_{0,1}=\phi_{\Eftriv}^*f_{0,1}$. In particular, we get $f_{1,0}=0$ and $f_{0,1}\in\Zp$ which proves the claim.
	\end{proof}
	\begin{rem}
		With a little bit more effort, one can indeed prove that $c=1$ in the above lemma. Since the proof needs a more careful study of the (formal) logarithm sheaves, we  will not present it here. The full proof will be given in the upcoming work \cite{Syntomic}.
	\end{rem}
	By restricting sections of the Poincar\'e bundle $\Po$ to $\Ef\times_S\Edf$ we obtain sections of $\hat{\Lcal}_n:=\pr_*(\Po|_{\Ef\times_S\Edf})$. We would like to give an explicit description of such elements in terms of the basis $(\hat{\omega}^{[k]})_{k=0}^n$ of $\TSym^n\hat{\Lcal}_n$.
	\begin{lem}
		Let $s\in\Gamma(\Eftriv\times\Edftriv,\Pof)$ and write $\vartheta\in\Gamma(\Eftriv\times\Edftriv,\Ocal_{\Eftriv\times\Edftriv})$ for the associated section obtained by the trivialization. Then
		\[
			s\Big|_{\Eftriv\times \Inf^n_e \Edtriv}\mapsto \left( (\id\times e)^*\left(\partial_{\Edftriv}^{\circ k}\vartheta\right)\cdot \omega^{[k]} \right)_{k=0}^n
		\]
		under
		\[
			\Gamma(\Eftriv\times\Inf^n_e \Edtriv,\Pof)=\Gamma(\Eftriv,\widehat{\Lcal}^\dagger_{n})\rightarrow \bigoplus_{k=0}^n \Gamma(\Eftriv,\TSym^k\om_{\Eftriv}).
		\]
		Here, $\partial_{\Edftriv}\colon\Ocal_{\Edftriv}\rightarrow \Ocal_{\Edftriv} $ denotes the canonical invariant derivation.
	\end{lem}
	\begin{proof}
		This follows from \eqref{eq_comult} and the fact that
		\[
			\Ocal_{\Inf^n_e \Etriv^\vee}\rightarrow \TSym^n_{\Ocal_{\Mtriv}}\left( \Ocal_{\Inf^1_e \Etriv^\vee}\right)=\bigoplus_{k=0}^n \TSym^k \om_{\Edtriv/\Etriv}.
		\]
		maps $f\in\Ocal_{\Inf^n_e \Etriv^\vee}$ to $(e^*\partial_{\Edftriv}^{\circ k} f \cdot\omega^{[k]})_{k=0}^n$.
	\end{proof}

\section[Proof of the Theorem]{Proof of the Theorem}\label{sec_proofpadic}

	 \begin{proof}[Proof of \cref{thm_padictheta}]
	 The Poincar\'e bundle $\Po^{\sharp,\dagger}$ is equipped with two integrable connections $\con{\sharp}$ and $\con{\dagger}$. Let us write $\Pof^{\sharp,\dagger}$ for the restriction of $\Po^{\sharp,\dagger}$ to the formal completion of $\Etriv^\sharp\times_\Mtriv \Etriv^\dagger$ along the the zero section. The unit root decomposition and the canonical generator of $\om_{\Edtriv/\Mtriv}$ induce a projection:
	 \[
	 	\HdR{1}{\Edtriv/\Mtriv}\twoheadrightarrow \om_{\Edtriv/\Mtriv}\righteq \Ocal_\Mtriv.
	 \] 
	 With this identification we obtain differential operators
	 \[
	 	\con{\sharp}: \Pof^{\sharp,\dagger}\rightarrow \Pof^{\sharp,\dagger}\otimes_{\Ocal_\Mtriv} \HdR{1}{\Edtriv/\Mtriv} \twoheadrightarrow \Pof^{\sharp,\dagger}
	 \]
	 and similarly
	 \[
	 	\con{\dagger}: \Pof^{\sharp,\dagger}\rightarrow \Pof^{\sharp,\dagger}\otimes_{\Ocal_\Mtriv} \HdR{1}{\Etriv/\Mtriv} \twoheadrightarrow \Pof^{\sharp,\dagger}.
	 \]
	 Tensoring the trivialization $\Pof\righteq \Ocal_{\Eftriv\times\Edftriv}$ with $\Ocal_{\Eftriv^\sharp\times\Eftriv^\dagger}$ gives a trivialization $\Pof^{\sharp,\dagger}\righteq \Ocal_{\Eftriv^\sharp\times\Eftriv^\dagger}$. Further, the unit root decomposition induces canonical projections $\Ocal_{\Eftriv^\sharp}\twoheadrightarrow \Ocal_{\Eftriv}$ and $\Ocal_{\Eftriv^\dagger}\twoheadrightarrow \Ocal_{\Eftriv^\vee}$. By the definition of the geometric nearly holomorphic modular forms $E^{k,r+1}_{s,t}$ it suffices to prove the commutativity of the following diagrams
	 \[
	 	\begin{tikzcd}
	 		\Pof^{\sharp,\dagger}\ar[d,"\cong"] \ar[r,"\con{\sharp}"] & \Pof^{\sharp,\dagger}\ar[d,"\cong"]\\
	 		\Ocal_{\Eftriv^\sharp}\hat{\otimes}_{\Ocal_\Mtriv}\Ocal_{\Eftriv^\dagger}\ar[d,two heads] & \Ocal_{\Eftriv^\sharp}\hat{\otimes}_{\Ocal_\Mtriv}\Ocal_{\Eftriv^\dagger}\ar[d,two heads]\\
	 		\Ocal_{\Eftriv}\hat{\otimes}_{\Ocal_\Mtriv}\Ocal_{\Eftriv^\vee} \ar[r,"\id\otimes \partial_{\Eftriv^\vee}"] & \Ocal_{\Eftriv}\hat{\otimes}_{\Ocal_\Mtriv}\Ocal_{\Eftriv^\vee}
	 	\end{tikzcd}
	 	\quad
	 	\begin{tikzcd}
	 		\Pof^{\sharp,\dagger}\ar[d,"\cong"] \ar[r,"\con{\dagger}"] & \Pof^{\sharp,\dagger}\ar[d,"\cong"]\\
	 		\Ocal_{\Eftriv^\sharp}\hat{\otimes}_{\Ocal_\Mtriv}\Ocal_{\Eftriv^\dagger}\ar[d,two heads] & \Ocal_{\Eftriv^\sharp}\hat{\otimes}_{\Ocal_\Mtriv}\Ocal_{\Eftriv^\dagger}\ar[d,two heads]\\
	 		\Ocal_{\Eftriv}\hat{\otimes}_{\Ocal_\Mtriv}\Ocal_{\Eftriv^\vee} \ar[r,"\partial_{\Eftriv}\otimes \id"] & \Ocal_{\Eftriv}\hat{\otimes}_{\Ocal_\Mtriv}\Ocal_{\Eftriv^\vee}
	 	\end{tikzcd}
	 \]
	 We prove the commutativity of the right diagram, the other case is completely symmetric. Recall that $\Po^{\sharp,\dagger}$ with the connection $\con{\dagger}$ is obtained by pullback of the universal connection $\con{\Po^\dagger}$ on the Poincar\'e bundle $\Po^\dagger$ along the projection $\Etriv^\sharp\times_\Mtriv \Etriv^\dagger\rightarrow \Etriv\times_\Mtriv \Etriv^\dagger$. Thus we can deduce the commutativity of the right diagram from the commutativity of
	 \begin{equation*}
	 	 \begin{tikzcd}
	 		\Pof^{\dagger}\ar[d,"\cong"] \ar[r,"\con{\dagger}"] & \Pof^{\dagger}\ar[d,"\cong"]\\
	 		\Ocal_{\Eftriv}\hat{\otimes}_{\Ocal_\Mtriv}\Ocal_{\Eftriv^\dagger}\ar[d,two heads] & \Ocal_{\Eftriv}\hat{\otimes}_{\Ocal_\Mtriv}\Ocal_{\Eftriv^\dagger}\ar[d,two heads]\\
	 		\Ocal_{\Eftriv}\hat{\otimes}_{\Ocal_\Mtriv}\Ocal_{\Eftriv^\vee} \ar[r,"\partial_{\Eftriv}\otimes \id"] & \Ocal_{\Eftriv}\hat{\otimes}_{\Ocal_\Mtriv}\Ocal_{\Eftriv^\vee}.
	 	\end{tikzcd}
	 \end{equation*}
	 It is enough to prove the commutativity of this diagram restricted to $\Eftriv\times_\Mtriv \Inf^n_e \Eftriv^\dagger$ for all $n\geq 1$, i.e. we have to prove for all $n\geq 1$ the commutativity of
	 \begin{equation}\label{diag_Pof_n}
	 	 \begin{tikzcd}
	 		\Lnf^\dagger\ar[d,"\cong"] \ar[r,"\nabla^{(n)}_\dagger"] & \Lnf^\dagger\ar[d,"\cong"]\\
	 		\Ocal_{\Eftriv}\hat{\otimes}_{\Ocal_\Mtriv}\Ocal_{\Inf^n_e \Etriv^\dagger}\ar[d,two heads] & \Ocal_{\Eftriv}\hat{\otimes}_{\Ocal_\Mtriv}\Ocal_{\Inf^n_e \Etriv^\dagger}\ar[d,two heads]\\
	 		\Ocal_{\Eftriv}\hat{\otimes}_{\Ocal_\Mtriv}\Ocal_{\Inf^n_e \Etriv^\vee} \ar[r,"\partial_{\Eftriv}\otimes \id"] & \Ocal_{\Eftriv}\hat{\otimes}_{\Ocal_\Mtriv}\Ocal_{\Inf^n_e \Etriv^\vee}.
	 	\end{tikzcd}
	 \end{equation}	 

	For the commutativity of the diagram \eqref{diag_Pof_n} let us consider the following diagram:
	 \[
	 \begin{tikzcd}[row sep=tiny, column sep=tiny]
	 	\Lnf^\dagger\ar[rr,"\nabla_{\dagger}^{(n)}"]\ar[dd,"\triv^{(n)}"]\ar[rd,hook] & & \Lnf^\dagger\ar[rd,hook]\ar[dd] & \\
	 	& \TSym^n_{\Ocal_{\Eftriv}} \widehat{\Lcal}_1^\dagger \ar[crossing over]{rr}[near start]{\con{(n)}}\ar[dd,"\triv^{(1)}"] & & \TSym^n_{\Ocal_{\Eftriv}} \widehat{\Lcal}_1^\dagger\ar[dd] \\
	 	\Ocal_{\Eftriv}\otimes_{\Ocal_\Mtriv}\Ocal_{\Inf^n_e \Etriv^\dagger}\ar[dd] & & \Ocal_{\Eftriv}\otimes_{\Ocal_\Mtriv}\Ocal_{\Inf^n_e \Etriv^\dagger}\ar[dd] & \\
	 	& \TSym^n_{\Ocal_{\Eftriv}} \left(\Ocal_{\Eftriv}\otimes_{\Ocal_\Mtriv}\Ocal_{\Inf^1_e \Etriv^\dagger}\right) && \TSym^n_{\Ocal_{\Eftriv}} \left(\Ocal_{\Eftriv}\otimes_{\Ocal_\Mtriv}\Ocal_{\Inf^1_e \Etriv^\dagger}\right)\ar[dd]\\
	 	\Ocal_{\Eftriv}\otimes_{\Ocal_\Mtriv}\Ocal_{\Inf^n_e \Etriv^\vee}\ar[rr]\ar[rd,hook] & & \Ocal_{\Eftriv}\otimes_{\Ocal_\Mtriv}\Ocal_{\Inf^n_e \Etriv^\vee}\ar[rd,hook]& \\
	 	& \TSym^n_{\Ocal_{\Eftriv}}\left( \Ocal_{\Eftriv}\otimes_{\Ocal_\Mtriv}\Ocal_{\Inf^1_e \Etriv^\vee}\right)\ar[crossing over, leftarrow, near start]{uu}{}\ar[rr] & & \TSym^n_{\Ocal_{\Eftriv}}\left( \Ocal_{\Eftriv}\otimes_{\Ocal_\Mtriv}\Ocal_{\Inf^1_e \Etriv^\vee}\right)
	 \end{tikzcd}
	 \]
	 In this diagram, the commutativity of the left and the right face is just the compatibility of the co-multiplication maps. The commutativity of the lower face is obvious. The upper face commutes by the horizontality of the co-multiplication maps. The commutativity of the front face follows from the explicit formula of $\con{(n)}$. Now, the commutativity of the back face is deduced from the commutativity of the other faces and the injectivity of the comultiplication maps.
	 \end{proof}


\section[$p$-adic interpolation of $p$-adic Eisenstein--Kronecker series]{p-adic interpolation of p-adic Eisenstein--Kronecker series}
As always when one has a $p$-adic measure $\mu$ on $\Zp$ it is only possible to define the moment function
	\[
		\Zp\ni s\mapsto \int_{\Zp^\times} \langle x\rangle^s \dd \mu(x)
	\]
	with $\langle\cdot\rangle\colon \Zp^\times\twoheadrightarrow (1+p\Zp)$ after restriction to $\Zp^\times$. Let us consider again the universal trivialized elliptic curve $(\Etriv/\Mtriv,\beta,\alpha_N)$ with $\Gamma(N)$-level structure. For $e\neq s\in \Etriv[N](\Mtriv)$ corresponding to $(a,b)$ via the level structure we have defined the $p$-adic measure $\muEisD$. Let us write $\muEisDp:=\muEisD|_{\Zp^\times\times \Zp} $ for the restriction of the measure $\muEisDp$ to $\Zp^\times\times \Zp$. It will be convenient to view $\muEisDp$ as a measure on $\Zp\times\Zp$ by extending by zero. We can easily deduce the following statement from Katz \cite[\S 6.3]{katz_padicinterpol} and \cref{PI_propKatz}. Katz deduces this result by comparing the $q$-expansions of the moments. In this section we sketch a different proof using the geometry of the Poincar\'e bundle. \par 
	\begin{thm}
	\[
		\int_{\Zp^\times\times \Zp} f(x,y)\dd \muEisD=\int_{\Zp\times \Zp} f(x,y)\dd \muEisD-\Frob \int_{\Zp\times \Zp} f(p\cdot x,y)\dd \muEisD.
	\]
	\end{thm}
	
	\begin{proof}[Sketch of the proof:] In the \cref{appendix_dist} we prove a distribution relation for the Kronecker section. The distribution relation for the isogeny $\varphi\colon E\rightarrow E'$ implies the formula
	\[
		p\cdot\Frob(\pthetaD([p](S),T))=\sum_{\zeta\in \Gmf{S}[p]} \pthetaD(S+_{\Gmf{S}}\zeta,T).
	\]
	Thus we get
	\begin{equation}\label{eq_Theta1}
		\pthetaD(S,T)-\frac{1}{p}\sum_{\zeta\in \Gmf{S}[p]} \pthetaD(S+_{\Gmf{S}}\zeta,T)=\pthetaD(S,T)-\Frob(\pthetaD([p](S),T)).
	\end{equation}
	It is well known, that the Amice transform of $f(S)-\frac{1}{p}\sum_{\zeta\in \hat{\mathbb{G}}_m[p]} f(S+_{\hat{\mathbb{G}}_m}\zeta)$ is the measure $\mu_f|_{\Zp^\times}$. In particular, the left hand side is the Amice transform of $\muEisDp$. Now the statement follows from \eqref{eq_Theta1}, since the inverse of the Amice transform can be computed by
	\[
		f(T)=\int_{\Zp}(1+T)^x\dd \mu_f(x).
	\]
	\end{proof}

\appendix

\section{The distribution relation}\label{appendix_dist}
The Kronecker theta function is known to satisfy a distribution relation \cite{bannai_kobayashi}. The aim of this appendix is to prove a similar distribution relation for the underlying Kronecker section. In order to state the general distribution relation we need a further generalization of the translation operators. Let us consider the following commutative diagram
		\begin{equation}\label{diag1}
			\begin{tikzcd}
				E\ar[r,"{[D']}"]\ar[d,"\psi"] & E \ar[d,"\psi"] \\
				E'\ar[r,"{[D']}"] & E'
			\end{tikzcd}
		\end{equation}
of isogenies of elliptic curves. The most important case is the case $E=E'$. Let us write $\gamma_{[D'],\psi^\vee}$ for the diagonal of the commutative diagram
\[
			\begin{tikzcd}[column sep=huge, row sep=huge]
				([D']\times\psi^\vee)^*\Po \ar[rd,"\gamma_{[D'],\psi^\vee}"]\ar[r,"({[D']}\times\id)^*\gamma_{\id,\psi^\vee}"]\ar[d,swap,"(\id\times\psi^\vee)^*\gamma_{[D'],\id}"] & (\psi\circ[D']\times\id)^*\Po' \ar[d,"(\psi\times\id)^*\gamma_{[D'],\id}"] \\
				(\id\times (\psi\circ[D'])^\vee)^*\Po' \ar[r,swap,"(\id\times{[D']})^*\gamma_{\id,\psi^\vee}"] & (\psi\times [D'])^*\Po'.
			\end{tikzcd}
\]
and $\gamma_{\psi,[D']}$ for its inverse.
\begin{defin}
	Let $s\in (\ker\psi)(S)$ and $t\in(E^{\prime,\vee}[D'])(S)$. The translation operator
	\[
		\Ucal_{s,t}^{\psi,[D']}\colon (T_s\times T_t)^*([D']\times\psi^\vee)^*\Po \rightarrow ([D']\times\psi^\vee)^*\Po
	\]
	is defined as the composition $\Ucal_{s,t}^{\psi,[D']}:=\gamma_{\psi,[D']}\circ(T_s\times T_t)^*\gamma_{[D'],\psi^\vee}$.
\end{defin}
As before let us write
\[
	U_{s,t}^{\psi,[D']}(f):=(\Ucal_{s,t}^{\psi,[D']}\otimes\id)\left( (T_s\times T_t)^*([D']\times\psi^\vee)^*f \right)
\]
for sections  $f\in\Gamma\left(E\times E^\vee, \Po\otimes \Omega^1_{E\times E^\vee/E^\vee}([E\times e]+[e\times E^\vee]) \right)$.
	Now, we can state the distribution relation which is motivated from the theta function distribution relation given in \cite[Proposition 1.16]{bannai_kobayashi}.
	\begin{thm}\label{ch_EP_thmdist}Let $E$ and $E'$ be elliptic curves fitting into a commutative diagram as \eqref{diag1}. Let $N,D$ and $D'$ be integers and let us assume that $N,D,D'$ and $\ker\psi$ are non-zero-divisors on $S$. Let us further assume that all torsion sections of the finite group schemes $\ker\psi$ and $E'^\vee[D']$ are already defined over $S$, i.e.
\[
	|(\ker\psi)(S)|=\deg\psi,\quad |(E'^\vee[D'])(S)|=(D')^2.
\]	
Then, for $t\in E'^\vee[D](S)$, $s\in E[N](S)$:
	\begin{align*}
		&\sum_{\substack{\alpha\in (\ker\psi)(S),\\ \beta\in (E'^\vee[D'])(S)}} U^{[N]\circ \psi,~[D]\circ [D']}_{s+\alpha,~t+\beta}(\scan )=\\
		=&((D')^2)\cdot \left( ([D]\times[N])^*\gamma_{\psi,[D']} \right)\left((\psi\times [D'])^*U^{N,D}_{\psi(s),[D'](t)}(s_{\mathrm{can},E'})\right)
	\end{align*}
	\end{thm}

	In its simplest but still interesting form the distribution relation specializes to the following equality:
	\begin{cor}\label{ch_EP_cor_dist1}
	For $E/S$ with $\tilde{D}$ invertible on $S$ and $|E[\tilde{D}](S)|=\tilde{D}^2$:
	\[
			\sum_{e\neq t\in \Ed[\tilde{D}](S)} U^{\tilde{D}}_{t}(\scan)=\tilde{D}^2\cdot \gamma_{1,\tilde{D}}\left( (\id \times [\tilde{D}])^*(\scan)\right)-([\tilde{D}]\times \id)^*(\scan)
	\]
	\end{cor}
	\begin{proof}
		This is the special case $D=N=1$, $\psi=\psi=\id$ and $[D']=[D']=[\tilde{D}]$ of \cref{ch_EP_thmdist}.
	\end{proof}	
	Let us define
	\[
		\scan^D:=D^2\cdot \gamma_{1,D}\left( (\id \times [D])^*(\scan)\right)-([D]\times \id)^*(\scan).
	\]
	the above Corollary states that
	\[
		\sum_{e\neq t\in \Ed[D](S)} U^{D}_{t}(\scan)=\scan^D.
	\]
	If $\varphi\colon E\rightarrow E'$ is an isogeny with $\deg\varphi$ being a non-zero-divisor on $S$, we obtain by summing over all $t\in\Ed$ and all $s\in \ker\varphi$ the following special case of the distribution relation:
\begin{cor}\label{ch_EP_cor_dist2}
	Let $\varphi\colon E\rightarrow E'$ be an isogeny of elliptic curves over $S$ with $\deg \varphi$ a non-zero-divisor on $S$. For $D$ a non-zero-divisor on $S$ we have
	\[
		\sum_{\tau\in \ker\varphi(S)} ([D]\times\id)^*\Ucal_{\tau,e}^{\varphi,\id}\left( (T_\tau\times \varphi^\vee)^*\scan^D \right) =([D]\times\id)^* \gamma_{\varphi,\id}\left( (\varphi \times \id)^*(s_{\mathrm{can},E'}^D)\right).
	\]
	\end{cor}
	\begin{proof}
		After inverting $D$ and making a finite \'etale base change, we may assume that $|E[D](S)|=D^2$. The general distribution relation gives us the identity
		\begin{align}\label{eq_dist1}
		\sum_{\substack{\tau\in (\ker\varphi)(S),\\ t\in (E'^\vee[D])(S)}} U^{\varphi,D}_{\tau,t}(\scan )=(D^2)\cdot \gamma_{\varphi,[D]}\left((\varphi\times[D])^*s_{\text{can},E'}\right).
	\end{align}
	On the other hand, again by the general distribution relation we have
	\begin{align}\label{eq_dist2}
		\sum_{\substack{\tau\in (\ker\varphi)(S)}} U^{\varphi,D}_{\tau,e}(\scan )&=([D]\times\id)^*\gamma_{\varphi,\id}\left((\varphi\times\id)^*U_{e,e}^{\id,[D]}(s_{\text{can},E'})\right)=\\
		&=([D]\times\id)^*\Big[\gamma_{\varphi,\id}\left((\varphi\times\id)^*s_{\text{can},E'})\right)\Big].
	\end{align}
	By \cref{ch_EP_cor_dist1} and a straight forward computation, we can identify the following two sums
	\[
		\sum_{\substack{\tau\in (\ker\varphi)(S),\\ e\neq t\in (E'^\vee[D])(S)}} U^{\varphi,D}_{\tau,t}(\scan )=\sum_{\tau\in \ker\varphi(S)} ([D]\times\id)^*\Ucal_{\tau,e}^{\varphi,\id}\left( (T_\tau\times \varphi^\vee)^*\scan^D \right).
	\]
	Subtracting \eqref{eq_dist2} from \eqref{eq_dist1} and using the last identification proves the corollary.
	\end{proof}		
	
	\begin{rem}
		For a CM elliptic curve $E$ over $\CC$, one can use the methods from section \ref{subsec_Analitification} to describe the analytification of translations of the Kronecker section in terms of translations of the Kronecker theta function. In this case, the distribution relation specializes to the analytic distribution relation in \cite{bannai_kobayashi}.
	\end{rem}
	\subsection{Density of torsion sections}
Before we prove the distribution relation let us recall the density of torsion sections for elliptic curves:
\begin{lem}
	Let  $N>1$ and $E/S$ be an elliptic curve with $N$ invertible on $S$. For $\Fcal$ a locally free $\Ocal_E$-module of finite rank, $U\subseteq E$ open and $s\in\Gamma(U,\Fcal)$ we have: The section $s$ is zero, if and only if $t^*s=0$ for all $T\rightarrow S$ finite \'{e}tale, $n\geq 0$ and $t\in E[N^n](T)$.
\end{lem}
\begin{proof}
By the sheaf property we may prove this locally and reduce to the case  $\Fcal=\Ocal_E^r$, $r\geq 0$. By \cite[Thm. 11.10.9]{EGA4_3} we are further reduced to prove the result in the case $S=\Spec k$ for a field $k$. In this case the result is well-known, cf.~\cite[(5.30) Thm, and the remark (2)  afterwards]{moonen_vdgeer}.
\end{proof}

\begin{rem}
	If we take all torsion points different from zero, we still get a universally schematically dense family. Indeed, a priori the family is then only universally schematically dense in the open subscheme $U=E\setminus\{e(S)\}$, but the inclusion $U\hookrightarrow E$ is also universally schematically dense, since it is the complement of a divisor \cite[cf. the remark after Lemma 11.33]{goertz}.
\end{rem}
\subsection{Proof of the distribution relation}	
For the proof of the distribution relation we will need the following lemma:
	\begin{lem}\label{ch_EP_lem_Ustomega}
	Let $\psi\colon E\rightarrow E'$ be an isogeny of elliptic curves and $D,\tilde{D}$ be positive integers.
	\begin{enumerate}
		\item\label{ch_EP_lem_Ustomega_b} For $\tilde{t}\in E'^\vee[\tilde{D}](S)$ we have:
		\begin{align*}
			&([\tilde{D}]\times\id)^*\gamma_{\psi,[D']}\circ(\psi\times [D'])^*\Ucal^{[\tilde{D}]}_{D'\tilde{t}}=\\
			=&([D']\times \psi^\vee)^*\Ucal^{[\tilde{D}]}_{\psi^\vee(\tilde{t})}\circ ([\tilde{D}]\times T_{\tilde{t}})^*\gamma_{\psi,[D']}
		\end{align*}
		\item\label{ch_EP_lem_Ustomega_a} For $\alpha \in (\ker\psi)(S)$ and $\beta\in (E'^\vee[D'])(S)$ with $\psi^\vee(\beta)\neq e$ we have
		\[
			(\id\times e)^* U^{\psi,[D']}_{\alpha,\beta}(\scan)=T_\alpha^*\omega^{[D']}_{\psi^\vee(\beta)}\,.
		\]
		\item\label{ch_EP_lem_Ustomega_c} For $\tilde{t}\in E'^\vee[\tilde{D}](S)$ we have:
		\[
			([D']\times\psi^\vee)^*\Ucal^{[\tilde{D}]}_{\psi^\vee(\tilde{t})}\circ ([\tilde{D}]\times T_{\tilde{t}})^*\Ucal^{\psi,[D']}_{\tilde{D}s,t}=\Ucal^{\psi,[D'\cdot\tilde{D}]}_{s,t+\tilde{t}}
		\] 
	\end{enumerate}
	\end{lem}
	\begin{proof}
		\ref{ch_EP_lem_Ustomega_b}: We have the following commutative diagrams:
		\[
			\begin{tikzcd}[column sep=1.5in]
				(\id \times T_{\tilde{t}})^*(\psi\circ[\tilde{D}] \times [D'])^*\Po'
					\ar[r,"{(\id \times T_{\tilde{t}})^*(\psi \times [D'])^*\gamma_{[\tilde{D}],\id}}"]
					\ar[d,equal] 
				& 	(\id \times T_{\tilde{t}})^*(\psi \times [D'\cdot\tilde{D}])^*\Po'
					\ar[d,"{(\id \times T_{\tilde{t}})^*\gamma_{\psi,[D'\cdot\tilde{D}]}}"]\\
				(\id \times T_{\tilde{t}})^*(\psi\circ[\tilde{D}] \times [D'])^*\Po'
					\ar[r,"{(\id \times T_{\tilde{t}})^*([\tilde{D}] \times \id )^*\gamma_{\psi,[D']}}"] 
				& (\id \times T_{\tilde{t}})^*([D'\cdot\tilde{D}] \times \psi^\vee)^*\Po
			\end{tikzcd}
		\]
		\[
			\begin{tikzcd}[column sep=1.5in]
				(\psi \times [D'\cdot\tilde{D}])^*\Po'
					\ar[r,"{(\psi \times [D'])^*\gamma_{\id,[\tilde{D}]}}"]
					\ar[d,"{(\id \times T_{\tilde{t}})^*\gamma_{\psi,[D'\cdot\tilde{D}]}}"]
				& 	(\psi\circ[\tilde{D}] \times [D'])^*\Po'
					\ar[d,"{\gamma_{\psi\circ[\tilde{D}],[D']}}"]\\
				 (\id \times T_{\tilde{t}})^*([D'\cdot\tilde{D}] \times \psi^\vee)^*\Po
					\ar[r,"{(\id \times T_{\tilde{t}})^*([D'] \times \psi^\vee)^*\gamma_{[\tilde{D}],\id}}"]
				& 		([D'] \times \psi^\vee\circ[\tilde{D}])^*\Po
			\end{tikzcd}
		\]
		\[
			\begin{tikzcd}[column sep=1.5in]
				(\psi\circ[\tilde{D}] \times [D'])^*\Po'
					\ar[r,"{([\tilde{D}] \times \id)^*\gamma_{\psi,[D']}}"]
					\ar[d,"{\gamma_{[\tilde{D}]\circ\psi,[D']}}"]
				& 	([D'\cdot\tilde{D}] \times \psi^\vee)^*\Po
					\ar[d,equal]\\
				([D'] \times \psi^\vee\circ[\tilde{D}])^*\Po
					\ar[r,"{([D'] \times \psi^\vee)^*\gamma_{\id,[\tilde{D}]}}"]
				& 	([D'\cdot\tilde{D}] \times \psi^\vee)^*\Po
			\end{tikzcd}
		\]		
		The composition of the upper horizontal arrows in the three diagrams is
		\begin{align*}
			([\tilde{D}]\times \id)^*\gamma_{\psi,[D']}\circ (\psi\times[D'])^*\gamma_{\id,[\tilde{D}]}\circ (\id\times T_{\tilde{t}})^*(\psi\times[D'])^*\gamma_{[\tilde{D}],\id}
		\end{align*}
		while the composition of the lower horizontal arrows is:
		\[
			([D']\times\psi^\vee)^*\gamma_{\id,[\tilde{D}]}\circ (\id\times T_{\tilde{t}})^*([D']\times\psi^\vee)^*\gamma_{[\tilde{D}],\id}\circ ([\tilde{D}]\times T_{\tilde{t}})^*\gamma_{\psi,[D']}
		\]
		The commutativity shows that both compositions are equal, i.\,e. it gives the middle equality in
		\begin{align*}
			&([\tilde{D}]\times\id)^*\gamma_{\psi,[D']}\circ(\psi\times [D'])^*\Ucal^{[\tilde{D}]}_{D'\tilde{t}}=\\
			=&([\tilde{D}]\times \id)^*\gamma_{\psi,[D']}\circ (\psi\times[D'])^*\gamma_{\id,[\tilde{D}]}\circ (\id\times T_{\tilde{t}})^*(\psi\times[D'])^*\gamma_{[\tilde{D}],\id}=\\
			=&([D']\times\psi^\vee)^*\gamma_{\id,[\tilde{D}]}\circ (\id\times T_{\tilde{t}})^*([D']\times\psi^\vee)^*\gamma_{[\tilde{D}],\id}\circ ([\tilde{D}]\times T_{\tilde{t}})^*\gamma_{\psi,[D']}=\\
			=&([D']\times \psi^\vee)^*\Ucal^{[\tilde{D}]}_{\psi^\vee(\tilde{t})}\circ ([\tilde{D}]\times T_{\tilde{t}})^*\gamma_{\psi,[D']}.
		\end{align*}
		\ref{ch_EP_lem_Ustomega_a}:	Let us first prove the following equality
		\begin{equation}\label{eq2}
			\Ucal^{\psi,[D']}_{\alpha,\beta}=([D']\times\id)^*\left[\gamma_{\psi,\id}\circ(T_{D'\alpha}\times\id)^*\gamma_{\id,\psi^\vee}\right]\circ(T_\alpha\times\psi^\vee)^*\Ucal_{\psi^\vee(\beta)}^{[D']}
		\end{equation}
		The proof of this equality follows the same lines as the proof of \ref{ch_EP_lem_Ustomega_b}. So let us only write the equations instead of all commutative diagrams:
		\begin{align*}
			\tag{I} &(\psi\times\id)^*\gamma_{\id,[D']}=([D']\times\id)^*\gamma_{\id,\psi^\vee}\circ\gamma_{\psi,[D']}\\
			\tag{II} &(T_\alpha\times\psi^\vee)^*\gamma_{\id,[D']}\circ(T_\alpha\times T_\beta)^*\gamma_{[D']\circ\psi,\id}=\\
			=&(T_\alpha\times\id)^*\gamma_{\psi,[D']}\circ(\psi\times T_\beta)^*\gamma_{[D'],\id}\\
			\tag{III} &(\id\times\psi^\vee)^*\gamma_{[D'],\id}=\gamma_{[D']\circ\psi,\id}\circ([D']\times\id)^*\gamma_{\id,\psi^\vee}
		\end{align*}
		These identities follow easily from the fact that $\gamma_{[D'],\id}$ is induced from the universal property of the Poincar\'e bundle. Using these identities we compute
		\begin{align*}
			&(\psi\times\id)^*\gamma_{\id,[D']}\circ (T_\alpha\times T_\beta)^*\left[ (\psi\times\id)^*\gamma_{[D'],\id}\circ([D']\times\id)^*\gamma_{\id,\psi^\vee} \right]=\\
			=&(T_\alpha\times\id)^*(\psi\times\id)^*\gamma_{\id,[D']}\circ (\psi\times T_\beta)^*\gamma_{[D'],\id}\circ(T_\alpha\times T_\beta)^*([D']\times\id)^*\gamma_{\id,\psi^\vee} \stackrel{(I)}{=}\\
			=&([D']\circ T_\alpha\times\id)^*\gamma_{\id,\psi^\vee}\circ(T_\alpha\times\id)^*\gamma_{\psi,[D']}\circ (\psi\times T_\beta)^*\gamma_{[D'],\id}\circ([D']\circ T_\alpha\times T_\beta)^*\gamma_{\id,\psi^\vee} \stackrel{(II)}{=}\\
			=&([D']\circ T_\alpha\times\id)^*\gamma_{\id,\psi^\vee}\circ(T_\alpha\times\psi^\vee)^*\gamma_{\id,[D']}\circ(T_\alpha\times T_\beta)^*\gamma_{[D']\circ\psi,\id}\circ([D']\circ T_\alpha\times T_\beta)^*\gamma_{\id,\psi^\vee} \stackrel{(III)}{=}\\
			=&([D']\times\id)^*(T_{[D'](\alpha)}\times\id)^*\gamma_{\id,\psi^\vee}\circ(T_\alpha\times\psi^\vee)^*\gamma_{\id,[D']}\circ(T_\alpha\times T_\beta)^*(\id\times\psi^\vee)^*\gamma_{[D'],\id}.
		\end{align*}
		Now equation \eqref{eq2} follows by precomposing this computation with $([D']\times\id)^*\gamma_{\psi,\id}\circ$. Observing that 
		\[
		(\id\times e)^*\left[\gamma_{\psi,\id}\circ(T_{D'\alpha}\times\id)^*\gamma_{\id,\psi^\vee}\right]
		\]
		is just the canonical isomorphism $T_{D'\alpha}^*\Ocal_{E}\righteq \Ocal_{E}$, we deduce $(a)$ from \eqref{eq2} by applying $(\id\times e)^*$.\\
		\ref{ch_EP_lem_Ustomega_c}: Follows along the same lines as \ref{ch_EP_lem_Ustomega_a} and \ref{ch_EP_lem_Ustomega_b}.
 	\end{proof}

	\begin{proof}[Proof of the distribution relation]
	Since $M:=N\cdot D\cdot D'\cdot \deg\psi$ is a non-zero-divisor, we may prove the equality after inverting $M$. Thus, let us assume that $M$ is invertible on $S$.	Let us write
	\begin{align*}
		A&:=\sum_{\substack{\alpha\in (\ker\psi)(S),\\ \beta\in (E'^\vee[D'])(S)}} U^{[N]\circ \psi,~[D\cdot D']}_{s+\alpha,~t+\beta}(\scan )\\
		B&:=(D')^2\cdot \left( ([D]\times[N])^*\gamma_{\psi,[D']} \right)\left((\psi\times [D'])^*U^{[N],[D]}_{\psi(s),[D'](t)}(s_{\text{can},E'})\right).
	\end{align*}
	Our aim is to prove $A=B$. Using the Zariski covering $(S[\frac{1}{\tilde{D}}])_{\tilde{D}>1, (\tilde{D},M)=1}$ together with a density of torsion section argument we reduce the proof of the equality $(\id\times\tilde{t})^*A=(\id\times\tilde{t})^*B$ for all $\tilde{D}$ torsion points $\tilde{t}\in E'^\vee[\tilde{D}](T)$ with $\tilde{D}$ invertible on $T$. Further, since $([D\cdot D']\times e)^*\Ucal^{[\tilde{D}]}_{([N]\circ\psi^\vee)(\tilde{t})}$ is an isomorphism, we are reduced to prove the following: For all $\tilde{D}>1$ coprime to $M$ we have the following equality:
	\begin{enumerate}
	\item[$(*)_{\tilde{D}}$] For all pairs $(T,\tilde{t})$ with $T$ an $S$-scheme, $\tilde{D}$ invertible on $T$ and $e\neq \tilde{t}\in E'^\vee[\tilde{D}](T)$ we have 
	\begin{align}\label{ch_EP_thmdist_eq2}
		\notag&\left(([D\cdot D']\times e)^*\Ucal^{[\tilde{D}]}_{([N]\circ\psi^\vee)(\tilde{t})} \right)\left[([\tilde{D}]\times \tilde{t})^*A\right]=\\
		=&\left(([D\cdot D']\times e)^*\Ucal^{[\tilde{D}]}_{([N]\circ\psi^\vee)(\tilde{t})} \right)\left[([\tilde{D}]\times \tilde{t})^*B\right]=
	\end{align}
	\end{enumerate}	
	We compute the left hand side for arbitrary $(T,\tilde{t})$:
		\begin{align}\label{ch_EP_thmdist_eq3}
		&\left(([D\cdot D']\times e)^*\Ucal^{[\tilde{D}]}_{([N]\circ\psi^\vee)(\tilde{t})}\right)\left(([\tilde{D}]\times \tilde{t})^*A\right)\stackrel{\text{Lem.} \ref{ch_EP_lem_Ustomega} (c)}{=}\\
		\notag =&(\id \times e)^*\left(\sum_{\substack{\alpha\in (\ker\psi)(S),\\ \beta\in (E'^\vee[D'])(S)}} U^{[N]\circ\psi,~[D\tilde{D} D']}_{(\tilde{D})^{-1}(s+\alpha),~\tilde{t}+t+\beta}(\scan) \right)\stackrel{\text{Lem.} \ref{ch_EP_lem_Ustomega}(b)}{=}\\
		\notag =&\sum_{\substack{\alpha\in (\ker\psi)(S),\\ \beta\in (E'^\vee[D'])(S)}} (T_{(\tilde{D})^{-1}(s+\alpha)})^* \omega^{[D\tilde{D} D']}_{([N]\circ\psi^\vee)(\tilde{t}+t+\beta)}=\\
		\notag =&(T_{(\tilde{D})^{-1}(s)})^*\sum_{\substack{\alpha\in (\ker\psi)(S),\\ \beta\in (E'^\vee[D'])(S)}} T_{\alpha}^* \omega^{[D\tilde{D} D']}_{([N]\circ\psi^\vee)(\tilde{t}+t+\beta)}\stackrel{\text{Lem.} \ref{ch_EP_lemTrace2}\ref{ch_EP_lemTrace2_b}}{=}\\
		\notag =&(T_{(\tilde{D})^{-1}(s)})^*\sum_{\beta\in (E'^\vee[D'])(S)} \psi^* \omega^{[D\tilde{D} D']}_{N(\tilde{t}+t+\beta)}\stackrel{\text{Lem.} \ref{ch_EP_lemTrace}\ref{ch_EP_lemTrace_b}}{=}\\
		\notag =&(D')^2\cdot(T_{(\tilde{D})^{-1}(s)})^* \psi^* \omega^{[D\tilde{D}]}_{(D'\cdot N)(\tilde{t}+t)}
	\end{align}
	Before we simplify the right hand side of the above equation, we use \cref{ch_EP_lem_Ustomega} to simplify the following expression:
	\begin{align*}
		&([D\cdot D']\times[N]\circ\psi^\vee)^*\Ucal^{[\tilde{D}]}_{([N]\circ\psi^\vee)(\tilde{t})}\circ([\tilde{D}]\times T_{\tilde{t}})^*([D]\times[N])^*\gamma_{\psi,[D']}\circ\\
		&\circ([\tilde{D}]\times T_{\tilde{t}})^*(\psi\times[D'])^*\Ucal^{[N],[D]}_{\psi(s),D't}=\\
		=&([D]\times[N])^*\left[ ([D']\times\psi^\vee)^*\Ucal^{[\tilde{D}]}_{([N]\circ\psi^\vee)(\tilde{t})}\circ([\tilde{D}]\times T_{\tilde{Nt}})^*\gamma_{\psi,[D']}\right]\circ\\
		&\circ([\tilde{D}]\times T_{\tilde{t}})^*(\psi\times[D'])^*\Ucal^{[N],[D]}_{\psi(s),D't}\stackrel{\text{Lem.}\ref{ch_EP_lem_Ustomega}\ref{ch_EP_lem_Ustomega_b}}{=}\\
		=&([D\tilde{D}]\times[N])^*\gamma_{\psi,[D']}\circ(\psi\times[D'])^*\Big[([D]\times[N])^*\Ucal_{ND'\tilde{t}}^{[\tilde{D}]}\circ\\
		&\circ([\tilde{D}]\times T_{D'\tilde{t}})^*\Ucal^{[N],[D]}_{\psi(s),D't}\Big]\stackrel{\text{Cor.}\ref{ch_EP_cor_compU}}{=}\\
		=&([D\tilde{D}]\times[N])^*\gamma_{\psi,[D']}\circ(\psi\times[D'])^*\Ucal^{[N],[D\tilde{D}]}_{\tilde{D}^{-1}(\psi(s)),D'(t+\tilde{t})}
	\end{align*}	
	Using this and again \cref{ch_EP_lem_Ustomega}, the right hand side of \eqref{ch_EP_thmdist_eq2} is:
	\begin{align*}
		&\left(([D\cdot D']\times e)^*\Ucal^{\tilde{D}}_{([N]\circ\psi^\vee)(\tilde{t})} \right)\left[([\tilde{D}]\times \tilde{t})^*B\right]=\\
		=&(D')^2\cdot(\id\times e)^*\left( (\psi\times[D'])^*U^{[N],[D\tilde{D}]}_{\tilde{D}^{-1}\psi(s),[D'](t+\tilde{t})}(s_{\mathrm{can},E'}) \right)\stackrel{\text{Lem.} \ref{ch_EP_lem_Ustomega} (a)}{=}\\
		=& (D')^2\cdot \psi^* T_{\tilde{D}^{-1}\psi(s)}^*\omega^{[D\tilde{D}]}_{ND'(t+\tilde{t})}=(D')^2\cdot  T_{\tilde{D}^{-1}s}^*\psi^*\omega^{[D\tilde{D}]}_{ND'(t+\tilde{t})}
	\end{align*}
	Comparing this last equation to \eqref{ch_EP_thmdist_eq3} shows that the equation in $(*)_{\tilde{D}}$ holds for all pairs $(T,\tilde{t})$.
	\end{proof}

\bibliographystyle{amsalpha} 
\bibliography{PaperEisensteinPoincare}

\providecommand{\bysame}{\leavevmode\hbox to3em{\hrulefill}\thinspace}
\providecommand{\MR}{\relax\ifhmode\unskip\space\fi MR }
\providecommand{\MRhref}[2]{%
  \href{http://www.ams.org/mathscinet-getitem?mr=#1}{#2}
}
\providecommand{\href}[2]{#2}
\begin{thebibliography}{EvdGM12}

\bibitem[BK10a]{bannai_kings}
K.~Bannai and G.~Kings, \emph{{{$p$}-adic elliptic polylogarithm, {$p$}-adic
  {E}isenstein series and {K}atz measure}}, Amer. J. Math. \textbf{132} (2010),
  no.~6, 1609--1654. \MR{2766179}

\bibitem[BK10b]{bannai_kobayashi}
K.~Bannai and S.~Kobayashi, \emph{{Algebraic theta functions and the {$p$}-adic
  interpolation of {E}isenstein-{K}ronecker numbers}}, Duke Math. J.
  \textbf{153} (2010), no.~2, 229--295. \MR{2667134}

\bibitem[BK11]{bannai_kings2}
K.~Bannai and G.~Kings, \emph{{{$p$}-adic {B}eilinson conjecture for ordinary
  {H}ecke motives associated to imaginary quadratic fields}}, {Algebraic number
  theory and related topics 2009}, {RIMS K{\^o}ky{\^u}roku Bessatsu, B25}, Res.
  Inst. Math. Sci. (RIMS), Kyoto, 2011, pp.~9--30. \MR{2868567}

\bibitem[BKT10]{BKT}
K.~Bannai, S.~Kobayashi, and T.~Tsuji, \emph{{On the de {R}ham and {$p$}-adic
  realizations of the elliptic polylogarithm for {CM} elliptic curves}}, Ann.
  Sci. {\'E}c. Norm. Sup{\'e}r. (4) \textbf{43} (2010), no.~2, 185--234.
  \MR{2662664}

\bibitem[EvdGM12]{moonen_vdgeer}
B.~Edixhoven, G.~van~der Geer, and B.~Moonen, \emph{{Abelian Varieties}}, 2012,
  available at \url{http://gerard.vdgeer.net/AV.pdf}, p.~331.

\bibitem[Gro66]{EGA4_3}
A.~Grothendieck, \emph{{{\'E}l{\'e}ments de g{\'e}om{\'e}trie alg{\'e}brique.
  {IV}. {\'E}tude locale des sch{\'e}mas et des morphismes de sch{\'e}mas.
  {III}}}, Inst. Hautes {\'E}tudes Sci. Publ. Math. (1966), no.~28, 255.
  \MR{0217086}

\bibitem[GRR72]{SGA7_I}
A.~Grothendieck, M.~Raynaud, and D.~S. Rim, \emph{{Groupes de monodromie en
  g{\'e}om{\'e}trie alg{\'e}brique. {I}}}, {Lecture Notes in Mathematics, Vol.
  288}, Springer-Verlag, 1972, S{\'e}minaire de G{\'e}om{\'e}trie
  Alg{\'e}brique du Bois-Marie 1967--1969 (SGA 7 I).

\bibitem[GW10]{goertz}
U.~G{\"o}rtz and T.~Wedhorn, \emph{{Algebraic geometry {I}}}, {Advanced
  Lectures in Mathematics}, Vieweg + Teubner, Wiesbaden, 2010, Schemes with
  examples and exercises. \MR{2675155}

\bibitem[Kat76]{katz_padicinterpol}
N.~M. Katz, \emph{{{$p$}-adic interpolation of real analytic {E}isenstein
  series}}, Ann. of Math. (2) \textbf{104} (1976), no.~3, 459--571.
  \MR{0506271}

\bibitem[Kat77]{katz_eismeasure}
\bysame, \emph{{The {E}isenstein measure and {$p$}-adic interpolation}}, Amer.
  J. Math. \textbf{99} (1977), no.~2, 238--311. \MR{0485797}

\bibitem[Kat04]{kato}
K.~Kato, \emph{{{$p$}-adic {H}odge theory and values of zeta functions of
  modular forms}}, Ast{\'e}risque (2004), no.~295, ix, 117--290, Cohomologies
  $p$-adiques et applications arithm{\'e}tiques. III. \MR{2104361}

\bibitem[MM74]{mazur_messing}
B.~Mazur and W.~Messing, \emph{{Universal extensions and one dimensional
  crystalline cohomology}}, {Lecture Notes in Mathematics, Vol. 370},
  Springer-Verlag, Berlin-New York, 1974. \MR{0374150 (51 \#10350)}

\bibitem[Nor86]{norman_padictheta}
P.~Norman, \emph{{Explicit {$p$}-adic theta functions}}, Invent. Math.
  \textbf{83} (1986), no.~1, 41--57. \MR{813581}

\bibitem[Oda69]{oda}
T.~Oda, \emph{{The first de {R}ham cohomology group and {D}ieudonn{\'e}
  modules}}, Ann. Sci. {\'E}cole Norm. Sup. (4) \textbf{2} (1969), 63--135.
  \MR{0241435}

\bibitem[Sch14]{rene}
R.~Scheider, \emph{{The de Rham realization of the elliptic polylogarithm in
  families}}, Ph.D. thesis, Universit{\"a}t Regensburg, January 2014.

\bibitem[Spr17]{PhD}
J.~Sprang, \emph{{Eisenstein series via the Poincar{\'e} bundle and
  applications}}, Ph.D. thesis, Universit{\"a}t Regensburg, January 2017.

\bibitem[Spr18]{deRham}
\bysame, \emph{{The algebraic de Rham realization of the elliptic polylogarithm
  via the Poincar{\`e} bundle}}, available at
  \url{https://arxiv.org/abs/1802.04999}, 2018.

\bibitem[Spr19]{Syntomic}
\bysame, \emph{{The syntomic realization of the elliptic polylogarithm via the
  Poincar{\'e} bundle}}, Doc. Math. \textbf{24} (2019), 1099--1134.

\bibitem[Tsu04]{tsuji}
T.~Tsuji, \emph{{Explicit reciprocity law and formal moduli for Lubin-Tate
  formal groups}}, J. Reine Angew. Math. \textbf{569} (2004), 103--173.

\bibitem[Urb14]{urban}
E.~Urban, \emph{{Nearly Overconvergent Modular Forms}}, pp.~401--441, Springer
  Berlin Heidelberg, Berlin, Heidelberg, 2014.

\end{thebibliography}
\end{document}